\documentclass[11pt]{elsarticle}
\usepackage{amsmath,amssymb,amsthm}
\usepackage{mathrsfs}
\usepackage{hyperref} \hypersetup{ hidelinks }

\usepackage{bm}
\usepackage{array}
\usepackage{float}
\usepackage{color}
\usepackage{lineno}
\usepackage{subcaption}
\usepackage{stmaryrd}
\usepackage{extarrows}
\usepackage{subcaption}
\usepackage{multirow}
\newtheorem{thm}{Theorem}[section]

\newtheorem{rmk}{Remark}[section]
\newtheorem{example}{Example}[section]

\newtheorem{definition}{Definition}[section]
\newproof{pf}{Proof}
\numberwithin{equation}{section}
\numberwithin{figure}{section}
\numberwithin{table}{section}
\allowdisplaybreaks

\newcommand\dd{\mathrm{d}}

\newcommand\bF{\bm{F}}

\newcommand\bv{\bm{v}}

\newcommand\bx{\bm{x}}

\newcommand\bV{\bm{V}}
\newcommand\bU{\bm{U}}

\newcommand\bT{\bm{T}}

\newcommand\tbF{\widetilde{\bm{F}}}
\newcommand\tbU{\widetilde{\bm{U}}}
\newcommand\thF{\widehat{\bm{F}}}

\newcommand\pd[2]{\dfrac{\partial {#1}}{\partial {#2}}}

\newcommand\abs[1]{\lvert #1 \rvert}
\newcommand\jump[1]{\llbracket #1 \rrbracket}
\newcommand\mean[1]{\{\!\!\{ #1 \}\!\!\}}
\newcommand\meanln[1]{\{\!\!\{ #1 \}\!\!\}^{\text{ln}}}

\begin{document}

	\begin{frontmatter}
	
	\title{High-order accurate entropy stable schemes for  compressible Euler equations  with van der Waals equation of state on adaptive moving meshes}
	
	            \author{Shangting Li}
            	\ead{shangtl@pku.edu.cn}
            	\address{ Center for Applied Physics and Technology, HEDPS and LMAM,
            		School of Mathematical Sciences, Peking University, Beijing 100871, P.R. China}
		   \cortext[cor1]{Corresponding author. Fax:~+86-10-62751801.}
	         \author{Huazhong Tang\corref{cor1}}
	         \ead{hztang@math.pku.edu.cn}
			\address{Nanchang Hangkong University, Jiangxi Province, Nanchang 330000, P.R. China;  Center for Applied Physics and Technology, HEDPS and LMAM,
				School of Mathematical Sciences, Peking University, Beijing 100871, P.R. China}

	
	\begin{abstract}
		This paper develops the high-order entropy stable (ES)  finite difference schemes for multi-dimensional compressible Euler equations with the van der Waals equation of state (EOS) on adaptive moving meshes.
 Semi-discrete schemes are first nontrivially constructed  built on  the newly derived high-order entropy conservative (EC) fluxes in curvilinear coordinates and scaled eigenvector matrices as well as the multi-resolution WENO reconstruction, and
then the fully-discrete schemes are  given by using the high-order explicit
strong-stability-preserving Runge-Kutta time discretizations.
 The
high-order EC fluxes in curvilinear coordinates are derived by using the discrete geometric conservation laws and the linear combination of the two-point symmetric  EC fluxes, while
the two-point  EC fluxes
 are delicately selected by using their sufficient condition, the thermodynamic entropy and the technically selected parameter vector.
%
%
The adaptive moving meshes are iteratively generated by solving the mesh redistribution equations, in which 
the fundamental derivative related to the occurrence of non-classical waves is involved to  produce high-quality mesh.
   Several numerical tests on the
   parallel computer system with the MPI programming are conducted to validate the
   accuracy, the ability to capture the classical and non-classical waves, and the high efficiency of our schemes in comparison with their counterparts  on the uniform mesh.

	\end{abstract}
	
	
\end{frontmatter}
\section{Introduction}\label{section:Intro}
This paper is concerned with  high-order accurate entropy stable (ES)  schemes to  the  compressible Euler equations with the van der Waals equation of state (EOS),
which  are given by
\begin{align}\label{eq:ConserLaw}
&\frac{\partial \bU}{\partial t}+\sum_{k=1}^{d} \frac{\partial \bF_{k}(\bU)}{\partial x_{k}}=0, \quad d=1,2,3,\\ \label{eq:ConserLaw2}
&\boldsymbol{U}= \left(
\rho, \rho \bm{v}^\mathrm{T},  E
\right)^\mathrm{T}, \\  \label{eq:ConserLaw3}
&\boldsymbol{F}_{k}=\left(
\rho v_{k},\rho v_{k}\bm{v}^\mathrm{T} + p\bm{e}_{k}^\mathrm{T}, (E+p) v_{k}
\right)^\mathrm{T},
\end{align}
and
\begin{align}\label{eq:vdWEOS}
p=\frac{\rho R T }{1-\rho b}-a \rho^{2}, \quad e = c_vT - a\rho,
\end{align}
where $\rho$, $\bm{v} = (v_{1}, \cdots, v_{d})^{\mathrm{T}}$, $T$, and
$E=\rho e  +\rho|\bm{v}|^{2} /2$
are the density,  the velocity  vector, the temperature, and the total energy, respectively.
Moreover, $\bm{e}_{k}$ denotes the $k$th column of the $d \times d$ unit matrix,   $e$ is  the specific internal energy, $R$ is the positive gas constant, $c_v$ is the specific heat at constant volume, and $a \geqslant 0$ and $b \geqslant 0$ are two  van der Waals constants depending on the intermolecular forces and the size of the molecules.
Obviously, when $a=0$ ad $b=0$, the van der Waals EOS \eqref{eq:vdWEOS} reduces to the ideal gas law.
The van der Waals EOS is often used to  depict the potentially non-classical phenomena  (occurrences of non-classical waves)  above the saturated vapour curve near the thermodynamic critical point and has important applications in engineering such as heavy gas wind tunnels and organic Rankine cycle engines \cite{anders1999transonic,wagner1978theoretical,Colonna2007}. Several researchers investigated the creation and evolution of the  non-classical waves
\cite{Bethe1998,Thompson1972,Zeldovich1946,doi:10.1063/1.866082,doi:10.1063/1.857855,doi:10.1063/1.858183},
%
and studied the numerical schemes such as 
the TVDM (total variation diminishing MacCormack) methods \cite{Argrow1996,2000Two,brown1998nonclassical} and the Roe type schemes \cite{ABGRALL1991171,GLAISTER1988382,Guardone2001,CINNELLA20061264,GUARDONE200250}
for the compressible Euler equations with the van der Waals EOS.

The van der Waals EOS \eqref{eq:vdWEOS} can be rewritten into a cubic equation with respect to the specific volume $\nu := 1/\rho$ as follows
\begin{align}\label{eq:vdWEOS-v02}
\nu^3 - \left(b+\frac{RT}{p}\right)\nu^2 + \frac{a}{p}\nu - \frac{ab}{p} = 0,
\end{align}
so that the roots for the specified gas at a given pressure
 are no more than  three cases: three identical real roots, three different real roots,
 and  one real root and two imaginary roots \cite{PRODANOV2022414077}.
 The thermodynamic critical point at the critical temperature $T_c$ (where the  derivatives $p_\nu$ and $p_{\nu\nu}$ are zero) corresponds to the case of that the cubic equation
 \eqref{eq:vdWEOS-v02} has three identical real roots,
and thus the van der Waals constants
 $a$ and $b$ can be obtained by analyzing  \eqref{eq:vdWEOS-v02}  at such critical point as follows
\begin{align}\label{eq:vdwconstant}
b = \frac{1}{3\rho_{c}}, \quad a = \frac{9p_{c}}{8Z_{c}\rho_{c}^2}, \quad Z_{c} = \frac{p_{c}}{R\rho_{c}T_{c}} = \frac{3}{8},
\end{align}
where $p_c$ and $\rho_c$  are the critical pressure and density, respectively,
and
will be used to non-dimensionalize  the compressible Euler equations with the van der Waals EOS \eqref{eq:ConserLaw}-\eqref{eq:vdWEOS}
in the next section. 

This paper only considers the van der Waals gas, whose
 physically admissible state should satisfy
\begin{align}
\label{eq:coefflimit}
0<\rho < 1/b, \quad T > T_0, \quad {R} T-2 a \rho(1-b \rho)^{2}>0.
\end{align}
In fact,
the van der Waals EOS \eqref{eq:vdWEOS} implies that the pressure $p$ is positive at any volume
if $ 0<\rho < 1/b$ and $R T  - a\rho(1-\rho b) > 0$.
The term $R T  - a\rho(1-\rho b)$ can be viewed as the quadratic equation about $\rho$ so that
the expression $R T  - a\rho(1-\rho b) $ is always   positive if its discriminant $-4abRT + a^2 < 0$, i.e.
\begin{align}
\label{eq:Tlimit}
T > T_0 := \dfrac{a}{4Rb}.
\end{align}
So, one has $ p>0$ if $  0<\rho < 1/b$ and $T > T_0 $.
On the other hand, from the fact that the product of the thermal expansion coefficient $-\frac{1}{\rho}(\partial_{T} \rho)|_{p}$ and the  Gr\"{u}neisen coefficient $\frac{1}{\rho}\left( \partial_{e} p\right)|_{\rho}$
is nonnegative \cite{SHYUE199943}, i.e.
\begin{align*}
\left(-\left.\frac{1}{\rho}\left(\partial_{T} \rho\right)\right|_{p}\right)\left(\left.\frac{1}{\rho}\left(\partial_{e} p\right)\right|_{\rho}\right)
=\left(\frac{{R}(1-b \rho)}{{R} T-2 a \rho(1-b \rho)^{2}}\right)\left(\frac{\delta}{1-b \rho}\right)\geqslant 0,
\end{align*}
with $\delta= \gamma-1 = R/c_v >  0$,
one has
\begin{align*}
{R} T-2 a \rho(1-b \rho)^{2}  
>0,
\end{align*}
here the case of  $RT-2a\rho(1-b\rho)^2 = 0$ has been excluded, relating to the infinite thermal expansion coefficient happening at the critical points.

{Additionally, it is worth noting that}
the fundamental derivative $G$ introduced by
\cite{thompson1971fundamental}  is a vital  parameter
governing the nonlinear dynamics of gases
\begin{align}\label{eq:fundamental}
G=\frac{\nu^{3}}{2 c_s^{2}}\left(\frac{\partial^{2} p}{\partial \nu^{2}}\right)_{s} =1+\frac{\rho}{c_s}\left(\frac{\partial c_s}{\partial \rho}\right)_{s},
\end{align}
where $c_s$  and $s$ are respectively the speed of sound and the specific entropy and
expressed as
\begin{align*}
c_{s}  
= \left[(\delta + 1)\frac{RT}{(1-b\rho)^2} - 2a\rho\right]^{1/2},\quad
s = c_v{\ln T} - R\ln(\nu- b).
\end{align*}
%
The sign of $G$ can be viewed as a parameter related to the occurrence of non-classical waves
such as the  expansion shocks etc.
To show that, 
	by using the Taylor expansion and the Rankine-Hugoniot jump conditions
	for a weak shock wave,  the relationship  between the entropy change
	$\Delta s$ and  the specific volume change $\Delta \nu$ is given as follows \cite{Bethe1998}
$$\Delta s=-\left(\frac{\partial^{2} p}{\partial \nu^{2}} \right)_s\frac{(\Delta \nu)^{3}}{12 T_f} + \mathcal{O}\left(\Delta \nu^4\right),$$
where $\left(\frac{\partial^{2} p}{\partial \nu^{2}}\right)_s$
presents the isentropic curvature which may be $0$
at the critical point and $T_f$ is the temperature before the shock wave.
Neglecting the  term $\mathcal{O}\left(\Delta \nu^4\right)$, one may have
the following result.
 For the ideal gas away from the critical point,
 it has $\left(\frac{\partial^{2} p}{\partial \nu^{2}}\right)_s> 0$,
  equivalently $G>0$.
  {Thus, if $\Delta\nu < 0$, then $\Delta s >0$ and a  compressive shock wave is formed.}
  For the gas with sufficiently large specific heats,
  the isentropic curvature $\left(\frac{\partial^{2} p}{\partial \nu^{2}}\right)_s $ and $G$ may be negative,
  so $\Delta s$ may be positive and the expansion shock wave without violating the entropy conditions may be formed when $\Delta\nu > 0$.

 For the hyperbolic system \eqref{eq:ConserLaw}-\eqref{eq:vdWEOS}, the physically relevant solutions may be  discontinuous even if the initial data are smooth, so one should consider weak solutions which are  not unique in general,  and the {entropy conditions}  satisfied by non-classical and classical waves are the significant property to single out the physically relevant solution out of  weak solutions.

 To select the physically relevant solution out of all  weak solutions, it is significant to construct the {ES schemes} satisfying  discrete or semi-discrete entropy conditions.  For the scalar conservation laws, the fully-discrete monotone conservative  schemes can converge to the entropy solution satisfying the entropy conditions \cite{Crandall1980Monotone,Harten1976}.
Semi-discrete E-schemes \cite{Osher1984Riemann,SOsher1988} satisfy the
entropy conditions for any convex entropy pairs.
 For the system of hyperbolic conservation laws,
the semi-discrete second-order entropy conservative (EC) scheme satisfying the  entropy identity was  constructed in  \cite{Tadmor1987The,Tadmor2003Entropy}, and the higher-order extension was proposed in \cite{Lefloch2002Fully}.
In order to  suppress possible oscillations near the discontinuities,
with the help of  the ``sign"  property of the ENO reconstruction,
some numerical dissipative terms
 were added to EC schemes to get  the semi-discrete ES schemes  \cite{Fjordholm2012Arbitrarily}. The ES  schemes {were} also extended to the ES discontinuous Galerkin (DG) based on  the summation-by-parts   operators \cite{Hiltebrand2014Entropy, Gassner2013A,Carpenter2014Entropy,Chen2020Review}.
 Recently, the EC or ES schemes were constructed for
 the (multi-component) compressible Euler equations  \cite{Ismail2009Affordable,chandrashekar_2013,li2022}, the relativistic (magneto)hydrodynamic equations \cite{Bhoriya2020Entropy,Wu2020Entropy,DUAN2021109949,duan2021highorder},
 and so on.
In a mimicking manner,	the energy stable schemes for the shallow water equations could also be developed, see  \cite{Fjordholm2011Well,zhang2023highorder,zhang2024highorder}.

In order to improve the efficiency and quality of the numerical solutions, adaptive moving mesh methods are considered,
which play an important role in solving partial differential equations, including the grid redistribution approaches \cite{Brackbill1993An,Brackbill1982Adaptive,Ren2000An,Wang2004A,Winslow1967Numerical}, the moving mesh PDEs methods \cite{CAO1999221,CENICEROS2001609,Stockie2001},
and \cite{DUAN2021109949, LI1997368,He2012RHD,He2012RMHD,Li2001Moving, Tang2003Adaptive,Tang2003An}.
For more details, readers are  referred to the
	review articles \cite{Budd2009Adaptivity,Tang2005Moving} for more exhaustive treatments.

This paper focuses on  constructing high-order adaptive moving mesh  ES schemes for the compressible Euler equations with the van der Waals  EOS.  Technically constructing the two-points EC fluxes is the key point which is difficult due to the non-linearity arising from the  van der Waals  EOS \eqref{eq:vdwconstant}. The explicit form of the two-points symmetric EC fluxes   in non-dimensional variables is first derived based on the carefully chosen parameter vectors and the thermodynamic entropy under rational conditions. Utilizing the linear combination of the two-points fluxes and  the high-order discrete geometric conservation laws,  the high-order
EC fluxes in curvilinear coordinates are constructed.
To avoid numerical oscillations near non-classical and classical waves, 
the
high-order dissipation terms, based on the newly derived complex scaled eigenmatrix and
the multi-resolution WENO
reconstruction \cite{WANG2021105138}, are added to the high-order EC fluxes to obtain
  the  semi-discrete high-order ES {schemes} satisfying the entropy inequality.
The mesh adaptation is implemented
by iteratively solving the Euler-Lagrange equations of the mesh adaptation functional with appropriate monitor function,  clustering the mesh  in the  region of interest. The monitor function is related to the fundamental derivative $G$ given in \eqref{eq:fundamental}, so the mesh can concentrate around the region where $G$ changes to capture the non-classical wave appearing in the van der Waals gas, which improves the performance of the solution.
The explicit third-order
strong-stability-preserving Runge-Kutta (SSP-RK) method \cite{Gottlieb2001Strong} is used for the time discretization.
	Several numerical results are provided to verify the efficiency of our schemes, which outperform their counterparts on a uniform mesh implemented on a parallel computer using MPI programming.

This paper is organized as follows. Section \ref{section: GoverningEquations} gives
the dimensionless versions of the governing equations \eqref{eq:ConserLaw}-\eqref{eq:vdWEOS}   in Cartesian and curvilinear coordinates, corresponding entropy conditions, 
and the physically admissible state.
Section \ref{section:Num} derives the explicit expression of
the two-points symmetric EC flux by solving a linear algebraic system deduced by  the chosen parameter vector, 
and then constructs
the high-order EC and ES schemes for the  compressible Euler equations with the van der Waals  EOS in curvilinear coordinates. 
Section \ref{section:MM} presents the
adaptive moving mesh strategy. Several numerical results are presented in Section \ref{section:Result} to validate the performance of high-order adaptive moving mesh  ES schemes. Some  conclusions are given in Section \ref{section:Conc}.

\section{Entropy conditions}\label{section: GoverningEquations}
This section introduces the dimensionless forms of the system
\eqref{eq:ConserLaw}-\eqref{eq:vdWEOS}  in Cartesian and curvilinear coordinates,  and  corresponding entropy conditions.
If the following non-dimensional thermodynamic variables
\begin{align*}
&	p_*=\frac{p}{{p}_{c}}, \quad \rho_*=\frac{\rho}{{\rho}_{c}}, \quad T_*=\frac{T}{{T}_{c}}, \quad v_{k,*}=\frac{v_k}{\left(R{ T}_{c}\right)^{1 / 2}},  \\
&	e_*=\frac{e}{R T_{c}}, \quad s_*=\frac{s-s_c}{R}, \quad c_{s,*}=\frac{c_s}{\left(R T_{c}\right)^{1 / 2}},   
\end{align*}	
are used, then the  conservative variable vector $\bU$  and the flux vectors $\bF_{k}$ in \eqref{eq:ConserLaw}-\eqref{eq:ConserLaw3} can be
 {rewritten in dimensionless form \cite{2000Two}  as follows }
\begin{equation}
\label{eq:Euler}
\begin{aligned}
&\bU_*=\left[\rho_*, \rho_*\bv^{\mathrm{T}}_*,  E_*\right]^{\mathrm{T}},
\\ &
\bF_{k,*} =\left[\rho_* v_{k,*}, \rho_* v_{k,*}\bv^{\mathrm{T}}_* + {\frac{3}{8}}p_* \bm{e}_{k}^\mathrm{T},\left(E_* + {\frac{3}{8}}p_*\right) v_{k,*}\right]^{\mathrm{T}}.
\end{aligned}
\end{equation}
Utilizing the van der Waals constants $a, b$ given in \eqref{eq:vdwconstant}, the van der Waals EOS \eqref{eq:vdWEOS} can be reformulated as
\begin{align}\label{eq:nondimEOS}
&p_* =\frac{1}{p_{c}}\left(\dfrac{R\rho_{c} T_{c}\rho_*T_*}{1-1/3\rho_*} - \dfrac{3p_{c}}{\rho_{c}^2}\rho_{c}^2\rho_*^2\right)
=
\dfrac{8\rho_*T_*}{3-\rho_*} - {3}\rho_*^2,\\
&e_* = \frac{c_vT_cT_*}{RT_{cr}} - \frac{3p_{c}}{\rho_{c}}\frac{1}{RT_{c}}\rho_* = \frac{T_*}{\delta} - \frac{9}{8}\rho_*,
\end{align}
and the  fundamental derivative, the thermodynamic entropy and the speed of sound   in the  non-dimensional variable form can be given by
 \begin{align*}
 G_*&  =
 \dfrac{(\delta+1)(\delta+2)\dfrac{p_{c}p_*+3p_{c}\rho_*^2}{\left({1}/{\rho_*}-1/3\right)^2{1}/{\rho_{c}^2}} - 18p_{c}\rho_{c}^2\rho_*^4}{2(\delta+1)\rho_{c}^2 \frac{p_{c}p_*+3p_{c}\rho_*^2}{{1}/\rho_*\left({1}/{\rho_*}-1/3\right) } - 12p_{c}\rho_{c}^2\rho_*^4}\\
 &=\frac{27}{2(c_{s,*})^{2}}\left[\left(\frac{1}{3-\rho_*}\right)^{3}\left(2+3 \delta+\delta^{2}\right) T_*-\frac{\rho_*}{4}\right],\\
  s_* &= \frac{\ln T_*}{\delta} + \ln(\frac{3-\rho_*}{2\rho_*}),\\
   c_{s,*} &= \frac{1}{(RT_c)^{1/2}}\left[\dfrac{(\delta + 1) RT_*}{\left(1-\frac{1}{3}\rho_*\right)^2}- 2 \frac{3p_c}{\rho_c}\rho_*\right]^{1/2}
 =\left[(\delta + 1) T_*\left(\dfrac{3}{3-\rho_*}\right)^{2} - \frac{9}{4}\rho_*\right]^{1/2}.
 \end{align*}
One can derive that the entropy $s_*$ satisfies
$$\partial_{t_*} s_* + \sum\limits_{k = 1}^{d} v_{k,*}
\dfrac{\partial {s_*}}{\partial {x_{k,*}}}  \geq 0,$$
or $$
\frac{\partial(\rho_*  s_*)}{\partial t_*}+ \sum\limits_{k
	= 1}^{d}\frac{\partial\left(\rho_* s_* v_{k,*}\right)}{\partial x_{k,*}} \geq 0.
$$

 In the subsequent discussion, the subscript $*$
 will be omitted for convenience,
 unless otherwise stated.
 	\begin{figure}[!ht]
 	\centering
 	\includegraphics[width=0.55\linewidth]{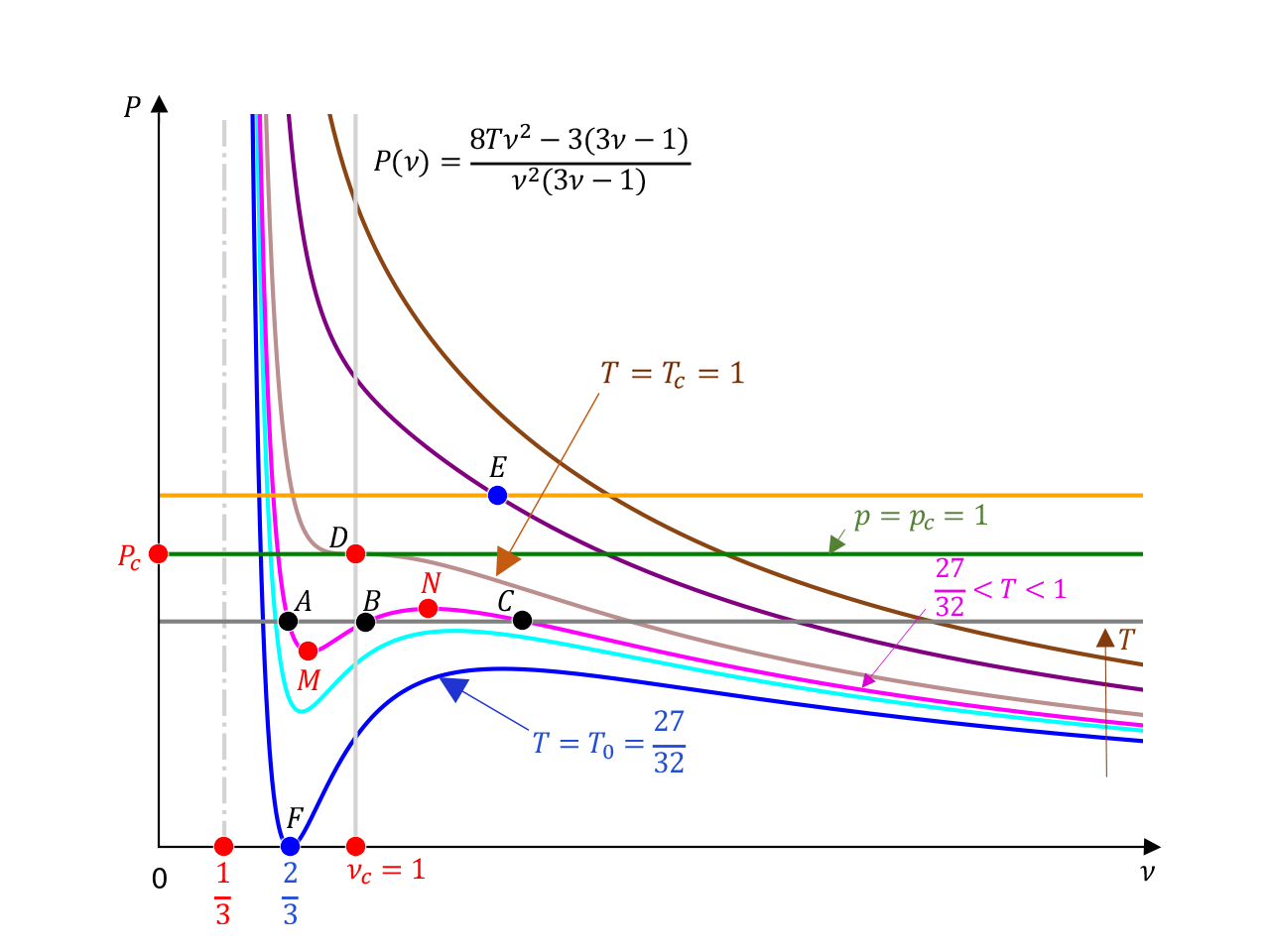}
 	\caption{The curves of the pressure $p = {8T}/{(3\nu-1)} - {3}/{\nu^2}$
 	with respect to specific volume $\nu=1/\rho$  at different temperature $T$.  }
 	\label{fig:phase_diagram}
 \end{figure}

\begin{rmk} \label{rem-eq:constraint} \rm
  The inequalities  \eqref{eq:coefflimit} and \eqref{eq:Tlimit}  in the  dimensionless variable form become
	$$
	 T > T_0 = 27/32, \quad{4T -\rho(3-\rho)^2}>0,
	 $$
	 respectively, so the dimensionless solution for the van der Waals gas should
	 satisfy them and $ 3>\rho>0$.
\hfill$\Box$
	\end{rmk}
	
\begin{rmk}\rm
The admissible states for the van der Waals gas can be
further discussed
by observing the $p$-$\nu$ curves
 in  Figure \ref{fig:phase_diagram},
similar to those  in \cite[Fig.1]{PRODANOV2022414077}.
For  the van der Waals gas, see Remark \ref{rem-eq:constraint},
the slope  of the $p$-$\nu$ curve
	\begin{align}\label{eq:diff_p}
	p_\nu =-6 \times \dfrac{4T\nu^3- (3\nu-1)^2}{(3\nu-1)^2},
	\end{align}
should be  less than zero.
The point D  is the thermodynamic critical point at the critical temperature $T_c$,
where both the first- and second-order derivatives $p_\nu $ and  $p_{\nu\nu} $  are $0$, so  the temperature  $T_D = T_c = 1$.
The  left and right  sides of the point D correspond to the
 liquid and gas, respectively.
The   $p$-$\nu$ curves at  the upper right side of  the isotherm $ T_c$
are the isotherms with $T> T_c=1$,
corresponding to the gas  and satisfying $ p_\nu<0$  when $\nu > 1/3$. The isotherm with $T \geqslant T_c$ has only one intersection point such as the point E with the isobaric line
(the horizontal line),
that is, the value of $\nu$ is uniquely determined   if $p$ and $T$ are given.
The point F is the local minima of the isotherm with $T = T_0$ and  $p = 0$.
	 Only when $T > T_0$, $p>0$ for all $\nu > 1/3$,
	 so that the solution should  satisfy the admissibility condition $T > T_0 $.
	  When $T_c  > T > T_0$, see e.g.  the isotherms where points A, B, and C are located, the  states between the local minimum (at the point M with $p_\nu=0$)
	  and the local maximum (at the point N with $p_\nu=0$)
	  of $p$  are unstable \cite{PRODANOV2022414077}.
Due to the discriminant $4T(T-1) < 0$, 	 the equation $p_\nu = 0$, written as $(\rho-2)^3 - 3(\rho-2) - 4T+2 = 0$,  has three unequal real roots  given explicitly by
	 \begin{align*}
	 &\rho^{*} = 2\cos\left(\dfrac{\arccos(2T-1)}{3}\right)+2,\quad
	 \rho^{**} = 2\cos\left(\dfrac{\arccos(2T-1)+2\pi}{3}\right)+2,\\
	 &\rho^{***} = 2\cos\left(\dfrac{\arccos(2T-1)+4\pi}{3}\right)+2.
	 \end{align*}
	For $1>2T-1> 11/16$, it is known that $\arccos(2T-1)\in(0,\pi/3)$, so
	that one has
	 $$\cos\left(\dfrac{\arccos(2T-1)}{3}\right)>\frac{1}{2},$$
which implies that $\rho^* > 3$ is not an admissible state, and
$\rho^{**}$ and $ \rho^{***}$ are the densities at  the points N and M, respectively,
satisfying
	 \begin{align*}
	 0 < &\rho_N= \rho^{**} < 2\cos\left(\dfrac{2\pi}{3}\right)+2 = 1, \\
	 1=2\cos\left(\dfrac{4\pi}{3}\right)+2
	 <& \rho_M= \rho^{***}
	 < 2\cos\left(\dfrac{5\pi}{3}\right)+2=3.
	 \end{align*}
Because  $p_\nu = 0$ should have two identical real roots for $\nu > 1/3$ if $T = T_c$,  one has $\nu_N(T_c) = \nu_M(T_c) = \nu_D = 1$,	 so that
for  $T_c> T> T_0$, $p_\nu<0$ iff
	 $\nu > \nu_N$ or $\nu < \nu_M$.
 The left side of the point M corresponds to the liquid, while the right side of the point N is  gas.	
	 Thus, the van der Waals gas should satisfy $\nu > \nu_N$ when $T_c> T > T_0$.
Lastly, by the way, a correction, see \cite{PRODANOV2022414077}, may be made in the region  $p_\nu > 0$
 with the help of  Maxwell's assumption that the areas of the two parts enclosed by isotherms and isobars are equal. Specifially, 	 for the isotherms where points A, B, and C are located,
 Using the equality of the two areas bounded by the isotherm and the
isobar
	 \begin{align*}
	 \int_{\nu_A}^{\nu_B}\left(p-\frac{8 T}{3\nu-1}+\frac{3}{\nu^2}\right) {\rm{d}} \nu=\int_{\nu_B}^{\nu_C}\left(\frac{8 T}{3\nu-1}-\frac{ 3}{\nu^2}-p\right) {\rm{d}} \nu,
	 \end{align*}
gives the expression of $p$ depending on the unknowns $\nu_A$ and  $\nu_C$, which
are  the smallest and  largest roots of the van der Waals equation, respectively. After that,
 the A-B-C  segment  is used to describe the state of the gas-liquid coexistence region.
\hfill$\Box$
	\end{rmk}
%


Now, let us recall the mathematical definition of the entropy function.
\begin{definition}[Entropy function] \rm\label{def: entropy function}
	If associated  with strictly convex scalar function $\eta(\bU)$ there exist scalar functions  $q_{k}(\bU)$, $k=1, \cdots, d$,  for the system \eqref{eq:ConserLaw}  satisfying
	$$
	q_{k}^{\prime}(\bU)=\eta^{\prime}(\bU)^{\mathrm{T}} \bF_{k}^{\prime}(\bU), \quad k=1, \cdots, d,
	$$
 then $\eta$ and  $q_k$ are called
	the entropy function  and the entropy fluxes, respectively, and
	$(\eta, q_1,\cdots, q_d)$ forms an entropy pair.
	The gradient vector of $\eta$ with respect to $\bU$, i.e. $\bV=\eta^{\prime}(\bU)^{\mathrm{T}}$, is called the entropy variables.
\end{definition}

Consider the scalar function pair for  \eqref{eq:ConserLaw}  taken as
\begin{align} \label{eq:van der WaalsEntropy}
\eta(\bU)&= - \rho s, \quad q_k(\bU)=-\rho s v_k,
\end{align}
then one has
\begin{align*}
\boldsymbol{V}=\eta^{\prime}(\boldsymbol{U})^{\mathrm{T}}& = \left(-s  +\frac{1}{\delta} -\frac{|\bv|^2}{2T} - \frac{9\rho}{4T}  +\frac{3}{3-\rho}, ~\frac{\bv^{\mathrm{T}}}{T}, -\frac{1}{T}\right)^{\mathrm{T}},
\end{align*}
and
$$\dfrac{\partial^2 \eta}{\partial \bU^2}=\dfrac{\delta}{\rho T^2}
\left[\begin{array}{ccccc}
\dfrac{T}{\delta}\left(\dfrac{9T}{(3-\rho)^2} - \dfrac{9}{4}\rho  + |\bv|^2\right)  + {\left(\mu_\delta  - \dfrac{T}{\delta}\right)^2}& - v_1\mu_{\delta} &- v_2\mu_{\delta} &- v_3\mu_{\delta}&  \mu_{\delta} -\dfrac{T}{\delta}\\
- v_1\mu_{\delta} & v_1^2 + \dfrac{T}{\delta }& v_1 v_2 &v_1 v_3 &-v_1 \\
- v_2\mu_{\delta}& v_2v_1&v_2^2+\dfrac{T}{\delta} &v_2v_3 & -v_2\\
- v_3\mu_{\delta}&v_3v_1&v_3v_2 &v_3^2 + \dfrac{T}{\delta} & -v_3\\
\mu_{\delta} -\dfrac{T}{\delta }& -v_1 & -v_2& -v_3&1\\
\end{array}\right].
$$
It can be verified that $\dfrac{\partial^2 \eta}{\partial \bU^2}$ is  symmetric positive definite, so that
 $(\eta, q_1, \cdots, q_d)$ forms an entropy pair of  \eqref{eq:ConserLaw}.
In this case, the entropy potential $\phi$ and entropy flux potential  $\psi_k$ can be explicitly calculated by
\begin{equation}\label{eq:van der WaalsPoten}
\begin{aligned}
\phi&=\boldsymbol{V}^{\mathrm{T}} \boldsymbol{U}-\eta(\boldsymbol{U})= - \left(\frac{9\rho}{8T} - \frac{3}{3-\rho}\right)\rho,\\
\psi_{k}&=\boldsymbol{V}^{\mathrm{T}} \boldsymbol{F}_{k}(\boldsymbol{U})-q_{k}(\boldsymbol{U})= - \left(\frac{9\rho}{8T} - \frac{3}{3-\rho}\right)\rho v_{k},
\end{aligned}
\end{equation}
and  the system  \eqref{eq:ConserLaw}
may be symmetrized with the change of variables $\bU \mapsto \bV$,
 because that  $\frac{\partial \bU}{\partial {\bV}}$ is symmetric positive definite, and   $\frac{\partial {\bF_k}}{\partial {\bU}} \frac{\partial {\bU}}{\partial {\bV}}$ is symmetric,
see Section \ref{sub: van der WaalsES}.

Following \cite{duan2021highorder},
%
by a time dependent
differentiable one-to-one coordinate mapping
\begin{align}\label{eq:transf}
t=\tau,\ \ \bx=\bx(\bm{\xi},\tau)\in \Omega_p,\ \
\bm{\xi}=(\xi_1,\cdots,\xi_d)\in\Omega_c,
\end{align}
where a reference mesh  in the computational domain $\Omega_c$ with coordinates $\bm{\xi}=(\xi_1,\cdots,\xi_d)$  can be viewed as the inverse image of the adaptive moving mesh in the physical domain $\Omega_p$ with coordinates $\bm{x}=(x_1,\cdots,x_d)$,
the system \eqref{eq:ConserLaw} can be transformed into
the following conservative form in curvilinear coordinates 
\begin{align}\label{eq:ConserLaw_curv}
\pd{\left(J\bU\right)}{\tau}+\sum_{k=1}^d\dfrac{\partial}{\partial\xi_k}\left[{\left(J\pd{\xi_k}{t}\bU\right)}
+\sum_{j=1}^d{\left(J\pd{\xi_k}{x_j}\bF_j\right)}\right]=0,
\quad J=\text{det}\left(\frac{\partial(t, \boldsymbol{x})}{\partial(\tau, \boldsymbol{\xi})}\right).
\end{align}
Moreover, for the transformation \eqref{eq:transf}, one should have
the following geometric conservation laws (GCLs) including   the
volume conservation law (VCL) and  the  surface conservation laws (SCLs)
\begin{equation}\label{eq:GCL}
\begin{aligned}
&\text{VCL:}\quad \pd{J}{\tau}+\sum_{k=1}^d\dfrac{\partial}{\partial\xi_k}{\left(J\pd{\xi_k}{t}\right)}=0,\\
&\text{SCLs:}\quad \sum_{k=1}^d\dfrac{\partial}{\partial\xi_k}{\left(J\pd{\xi_k}{x_j}\right)}=0,~ j=1,\cdots,d,
\end{aligned}
\end{equation}
where the VCL  indicates that the volumetric increment of a moving cell is equal to the sum of the changes along the surfaces that enclose the cell,
while the SCLs   imply that the cell volume should be closed by its surfaces \cite{Zhang1993Discrete}.
Utilizing  \eqref{eq:ConserLaw_curv} and the GCLs \eqref{eq:GCL} as shown in \cite{duan2021highorder},  the entropy condition in curvilinear coordinates is given by
\begin{align*}
\pd{\left(J\eta\right)}{\tau}+\sum_{k=1}^d\dfrac{\partial}
{\partial\xi_k}\left[{\left(J\pd{\xi_k}{t}\eta\right)}
+\sum_{j=1}^d{\left(J\pd{\xi_k}{x_j}q_j\right)}\right]\leqslant0,
\end{align*}
where the equality  is established
for the smooth solutions of  \eqref{eq:ConserLaw}, and
the  inequality  holds for the discontinuous solutions in the sense of {distributions}.

\section{Numerical schemes}\label{section:Num}
This section presents  the 3D adaptive moving mesh EC and ES finite difference schemes for the system \eqref{eq:ConserLaw_curv} on the structured  mesh, and  the 1D and 2D schemes can be viewed as the degenerative case which is shown in the appendices in \cite{duan2021highorder}. Dividing
a cuboid computational domain  $\Omega_c: [a_1,b_1]\times[a_2,b_2]\times[a_3,b_3]$
  into a   fixed orthogonal uniform mesh
$\{(\xi_{1,i_1},\xi_{2,i_2},
\xi_{3,i_3})$:
$a_k=\xi_{k,1}<\cdots<\xi_{k,i_k}
<\cdots<\xi_{k,N_k}=b_k$, $k=1,2,3\}$
with the constant mesh step size $\Delta \xi_k=\xi_{k,i_k+1}-\xi_{k,i_k}$,
 the semi-discrete conservative {$2w$}-order ($w\geqslant1$) finite difference schemes
for
\eqref{eq:ConserLaw_curv}
and the first equation in \eqref{eq:GCL} are given by
\begin{align}
\label{eq:semi_U}
&\dfrac{\dd}{\dd \tau}(J\bU)_{\bm{i}}=
-\sum_{k=1}^3\dfrac{1}{\Delta \xi_k}\left(\left(\widehat{\bF}_k\right)_{\bm{i},k,+\frac12}^{2w\rm{th}}-\left(\widehat{\bF}_k\right)_{\bm{i},k,-\frac12}^{2w\rm{th}}\right),
\\
\label{eq:semi_J}
&\dfrac{\dd}{\dd \tau}J_{\bm{i}}=
-\sum_{k=1}^3\dfrac{1}{\Delta \xi_k}\left(\left(\widehat{J\pd{\xi_k}{t}}\right)_{\bm{i},k,+\frac12}^{2w\rm{th}}-\left(\widehat{J\pd{\xi_k}{t}}\right)_{\bm{i},k,-\frac12}^{2w\rm{th}}\right),
\end{align}
where
the index $\bm{i}=(i_1,i_2,i_3)$   denotes the point
$(\xi_{1,i_1}, \xi_{2,i_2}, \xi_{3,i_3})$, the notation $\{\bm{i},k,m\}$ presents that the index $\bm{i}$ increases $m$ along $i_k$-direction, e.g. $\{\bm{i},1,+\frac12\}$ is $(i_1+\frac12, i_2, i_3)$,
$J_{\bm{i}}(\tau)$ and $(J\bU)_{\bm{i}}(\tau)$ approximate respectively   the point values of
$J\left(\tau,\bm{\xi}\right)$ and $(J\bU)(\tau,\bm{\xi})$ at $\bm{i}$,
and
{ $\left(\widehat{\bF}_k\right)_{\bm{i},k,\pm\frac12}^{2w\rm{th}}$}
and {
	$\left(\widehat{J\pd{\xi_k}{t}}\right)_{\bm{i},k,\pm\frac12}^{2w\rm{th}}$} are the numerical fluxes approximating respectively
{$J\pd{\xi_k}{t}\bU+\sum\limits_{j=1}^3 J\pd{\xi_k}{x_j}\bF_j$} and
 $J\pd{\xi_k}{t}$
 at $\{\bm{i}, k, \pm\frac{1}{2}\}$,   $k=1,2,3$, 
see \eqref{eq:ECFlux_curv_2p} and \eqref{eq:GCL_flux} in Section \ref{sec:sufficon} for specific expressions. The SCLs   \eqref{eq:GCL} should also be satisfied in the   discrete form
\begin{equation}\label{eq:SCL_dis}
\sum_{k=1}^3\dfrac{1}{\Delta \xi_k}\left(\left(\widehat{J\pd{\xi_k}{x_j}}\right)_{\bm{i},k,+\frac12}^{2w\rm{th}}-\left(\widehat{J\pd{\xi_k}{x_j}}\right)_{\bm{i},k,-\frac12}^{2w\rm{th}}\right)=0,~j=1,2,3,
\end{equation}
where $\left(\widehat{J\pd{\xi_k}{x_j}}\right)_{\bm{i},k,\pm\frac12}^{2w\rm{th}}$ is given by \eqref{eq:GCL_flux} in Section \ref{sec:sufficon}.


\subsection{Two-point EC fluxes in curvilinear coordinates}\label{sec:sufficon}

%
%
This part proposes two-point symmetric  EC fluxes in curvilinear coordinates  according to the  sufficient condition given in \cite{duan2021highorder}.
	If   a  two-point symmetric flux  $\thF_k\left(\bU_l, \bU_r, \left(J\pd{\xi_k}{\zeta}\right)_l, \left(J\pd{\xi_k}{\zeta}\right)_r \right)$, $\zeta=t,x_1,x_2,x_3$ is consistent with
$J\pd{\xi_k}{t}\bU+\sum\limits_{j=1}^3 J\pd{\xi_k}{x_j}\bF_j$,
chosen as
\begin{align}\label{eq:ECfluxMM}
\thF_k\left(\bU_l,\bU_r,\left(J\pd{\xi_k}{\zeta}\right)_l,\left(J\pd{\xi_k}{\zeta}\right)_r\right)
=&~\dfrac12\left(\left(J\pd{\xi_k}{t}\right)_l + \left(J\pd{\xi_k}{t}\right)_r\right)\widetilde{\bU} \nonumber\\
&+\sum_{j=1}^3\dfrac12\left(\left(J\pd{\xi_k}{x_j}\right)_l + \left(J\pd{\xi_k}{x_j}\right)_r\right)\widetilde{\bF}_j,
\end{align}
with $\widetilde{\bU}=\widetilde{\bU}(\bU_l,\bU_r)=\widetilde{\bU}(\bU_r,\bU_l)$ and $\widetilde{\bF}_j=\widetilde{\bF}_j(\bU_l,\bU_r)=\widetilde{\bF}_j(\bU_r,\bU_l)$  satisfying   
\begin{align}\label{eq:ECCondition_comp}
\left(\bV_r-\bV_l\right)^\mathrm{T}\widetilde{\bU}=\phi(\bU_r) - \phi(\bU_l),\quad
\left(\bV_r-\bV_l\right)^\mathrm{T}\widetilde{\bF}_j=\psi_j(\bU_r) - \psi_j(\bU_l),
\end{align}
then the  semi-discrete  scheme \eqref{eq:semi_U}-\eqref{eq:semi_J} is EC, that is, the solution satisfies the following semi-discrete entropy identity
	\begin{align*}
	\dfrac{\dd}{\dd t}J_{\bm{i}}\eta(\bU_{\bm{i}}(t))
	+\sum_{k=1}^3\dfrac{1}{\Delta \xi_k}\left(\left(\widehat{q}_k\right)_{\bm{i},k,+\frac12}^{2w\rm{th}}(t)-\left(\widehat{q}_k\right)_{\bm{i},k,-\frac12}^{2w\rm{th}}(t)\right)=0,
	\end{align*}
	where  $\bU_l$ and $\bU_r$ represent  the left and right states, respectively, and the numerical entropy flux
	$\left(\widehat{q}_k\right)_{\bm{i},k,+\frac12}^{2w\rm{th}}$ is consistent with the entropy flux
	$J\pd{\xi_k}{t}\eta+\sum\limits_{j=1}^3J\pd{\xi_k}{x_j}q_j$
	but it is not  unique.

The explicit expressions  of  the two-point EC fluxes  $\thF_k$ can be establish under some admissible conditions.  The  notations
\begin{align*}
\mean{a} = (a_r +a_l)/2, \quad \jump{a} = a_r - a_l,
\end{align*}
are employed  for convenience to denote  respectively the arithmetic mean and jump of $a$,
and choose the parameter vector as
\begin{align}
\label{eq:parameter_vector}
\boldsymbol{z}=\left(z_{1}, z_{2}, z_{3}, z_{4}, z_{5}\right)^{\mathrm{T}}=(\rho, \bv^{\mathrm{T}}, {T})^{\mathrm{T}} = (\rho, \bv^{\mathrm{T}}, \dfrac{(3-\rho)(p+3\rho^2)}{8\rho})^{\mathrm{T}}.
\end{align}

Using the identities
\begin{align*}
\jump{ab} = \mean{a}\jump{b} + \jump{b}\mean{a},\quad
\jump{\dfrac{1}{b}} =-\dfrac{\jump{b}}{b_{l} b_{r}},
\quad
\jump{\dfrac{a}{b}}=\dfrac{\jump{ a}}{\mean{b}} - \mean{\dfrac{a}{b}}\dfrac{\jump{b}}{\mean{b}},
\end{align*}
%
gives the decomposition of $\jump{\bV}$ as follows
\begin{equation}
\label{jumpV}
\left\{
\begin{aligned}
\jump{\bV_1}& =  \left(\dfrac{1}{\meanln{z_1}} +\dfrac{1}{\meanln{3-z_1}} + \dfrac{3}{(3-z_{1,L})(3-z_{1,R})} - \frac{9}{4\mean{z_5}}\right)\jump{z_1} \\
&~+
\left(\frac{1}{2}\sum\limits_{m = 2}^{4}\dfrac{1}{\mean{z_5}}\mean{\dfrac{z_m^2}{z_5}} - \dfrac{1}{\delta\meanln{z_5} }+ \dfrac{9}{4}\dfrac{1}{\mean{z_5}} \mean{\dfrac{z_1}{z_5}}\right) \jump{z_5} -\sum\limits_{m = 2}^{4}\dfrac{\mean{z_m}\jump{z_m}}{\mean{z_5}},\\
\jump{\bV_m}& =- \dfrac{1}{\mean{z_5}}\mean{\dfrac{z_m}{z_5}}\jump{z_5} + \dfrac{1}{\mean{z_5}} \jump{z_m}, \quad m = 2,3,4,\\
\jump{\bV_5}&=  \dfrac{1}{\mean{z_5}}\mean{\dfrac{1}{z_5}}\jump{z_5},
\end{aligned}
\right.
\end{equation}
where $\meanln{a}:=\jump{a}/\jump{\ln{a}}, a>0$ is the logarithmic mean, see \cite{Ismail2009Affordable}. Similarly,  $\jump{\phi}$ and $\jump{\psi_1}$ can be expanded as
\begin{equation*}
\label{jumppsi}
\begin{aligned}
\jump{\phi} =& \mean{z_1} \left[-\frac{9}{8}\dfrac{1}{\mean{z_5}}\left(\jump{z_1}- \mean{\dfrac{z_1}{z_5}}\jump{z_5}\right) + \frac{3\jump{z_1}}{(3-z_{1,L})(3-z_{1,R})}
\right] + \jump{z_1}\mean{ -\frac{9}{8} \dfrac{z_1}{z_5} + \frac{3}{3-z_1}}\\
= &\left[ \mean{z_1}\left(-\frac{9}{8}\dfrac{1}{\mean{z_5}} + \frac{3}{(3-z_{1,L})(3-z_{1,R})}\right)
+\mean{ -\frac{9}{8} \dfrac{z_1}{z_5} + \frac{3}{3-z_1}}\right] \jump{z_1} \\
&+\frac{9}{8}\dfrac{\mean{z_1}}{\mean{z_5}}\mean{\dfrac{z_1}{z_5}} \jump{z_5}, \\
\jump{\psi_1} =& \mean{z_1 z_2} \left[-\frac{9}{8}\dfrac{1}{\mean{z_5}}\left(\jump{z_1}- \mean{\dfrac{z_1}{z_5}}\jump{z_5}\right) + \frac{3\jump{z_1}}{(3-z_{1,L})(3-z_{1,R})}
\right]\\
& + \left(\jump{z_1}\mean{z_2} + \jump{z_2}\mean{z_1}\right) \mean{ -\frac{9}{8} \dfrac{z_1}{z_5} + \frac{3}{3-z_1}}\\
= &\left[ \mean{z_1z_2}\left(-\frac{9}{8}\dfrac{1}{\mean{z_5}} + \frac{3}{(3-z_{1,L})(3-z_{1,R})}\right)
+\mean{ -\frac{9}{8} \dfrac{z_1}{z_5} + \frac{3}{3-z_1}}\mean{z_2}\right] \jump{z_1} \\
&+\frac{9}{8\mean{z_5}}\mean{z_1z_2}\mean{\dfrac{z_1}{z_5}} \jump{z_5} + \mean{z_1}\mean{ -\frac{9}{8} {\dfrac{z_1}{z_5}} + \frac{3}{3-z_1}} \jump{z_2}.
\end{aligned}
\end{equation*}
Substituting \eqref{jumpV} and the above identities into two identities in \eqref{eq:ECCondition_comp} for the two-point EC fluxes and
equating the coefficients of the same jump terms on both side of the identities derive
\begin{equation}\label{eq:linear_U}
 \left\{
 \begin{aligned}
&\left(\dfrac{1}{\meanln{z_1}} +\dfrac{1}{\meanln{3-z_1}} + \dfrac{3}{(3-z_{1,L})(3-z_{1,R})} - \frac{9}{4\mean{z_5}}\right) \tbU_{1} \\
&~~= \left( \mean{z_1}\left(-\frac{9}{8}\dfrac{1}{\mean{z_5}} + \frac{3}{(3-z_{1,L})(3-z_{1,R})}\right)
+\mean{ -\frac{9}{8} \dfrac{z_1}{z_5} + \frac{3}{3-z_1}}\right),\\
&-\dfrac{1}{\mean{z_5}}\mean{z_m}\widetilde{\bU}_{1} + \dfrac{1}{\mean{z_5}} \widetilde{\bU}_{m} =
0,\quad m = 2, 3,4,\\
& \left(\frac{1}{2}\sum\limits_{m = 2}^{4}\dfrac{1}{\mean{z_5}}\mean{\dfrac{z_m^2}{z_5}} - \dfrac{1}{\delta\meanln{z_5} }+ \dfrac{9}{4}\dfrac{1}{\mean{z_5}} \mean{\dfrac{z_1}{z_5}}\right) \widetilde{\bU}_{1} - \sum\limits_{m = 2}^{4} \dfrac{1}{\mean{z_5}} \mean{\dfrac{z_m}{z_5}}\widetilde{\bU}_{m}  \\
& + \dfrac{1}{\mean{z_5}}\mean{\dfrac{1}{z_5}}\widetilde{\bU}_{5}  = \frac{9}{8}\dfrac{\mean{z_1}}{\mean{z_5}}\mean{\dfrac{z_1}{z_5}},
\end{aligned}
\right.
\end{equation}
\begin{equation}\label{eq:linear_F}
\quad\quad \left\{
\begin{aligned}
&\left(\dfrac{1}{\meanln{z_1}} +\dfrac{1}{\meanln{3-z_1}} + \dfrac{3}{(3-z_{1,L})(3-z_{1,R})} - \dfrac{9}{4\mean{z_5}}\right) \tbF_{1,1} \\
&~~=  \mean{z_1z_2}\left(-\dfrac{9}{8}\dfrac{1}{\mean{z_5}} + \dfrac{3}{(3-z_{1,L})(3-z_{1,R})}\right)
+\mean{ -\dfrac{9}{8} \dfrac{z_1}{z_5} + \dfrac{3}{3-z_1}}\mean{z_2} ,\\
&-\dfrac{1}{\mean{z_5}}\mean{z_2}\widetilde{\bF}_{1,1} + \dfrac{1}{\mean{z_5}} \widetilde{\bF}_{1,2} =
\mean{z_1}\mean{ -\frac{9}{8} {\dfrac{z_1}{z_5}} + \frac{3}{3-z_1}},\\
&-\dfrac{1}{\mean{z_5}}\mean{z_m}\widetilde{\bF}_{1,1} + \dfrac{1}{\mean{z_5}} \widetilde{\bF}_{1,m} =
0,\quad m =  3,4,\\
& \left(\frac{1}{2}\sum\limits_{m = 2}^{4}\dfrac{1}{\mean{z_5}}\mean{\dfrac{z_m^2}{z_5}} - \dfrac{1}{\delta\meanln{z_5} }+ \dfrac{9}{4}\dfrac{1}{\mean{z_5}} \mean{\dfrac{z_1}{z_5}}\right) \widetilde{\bF}_{1,1} - \sum\limits_{m = 2}^{4} \dfrac{1}{\mean{z_5}} \mean{\dfrac{z_m}{z_5}}\widetilde{\bF}_{1,m}  \\
& + \dfrac{1}{\mean{z_5}}\mean{\dfrac{1}{z_5}}\widetilde{\bF}_{1,5}  = \frac{9\mean{z_1z_2}}{8\mean{z_5}}\mean{\dfrac{z_1}{z_5}}.
\end{aligned}
\right.
\end{equation}
\begin{thm}\label{thm:condition}\rm
	Assuming  that $\bm{z}_l, \bm{z}_r$ satisfy the constraints  in Remark \ref{rem-eq:constraint}  and
	$\mean{\bm{z}}$ satisfies
	\begin{align}\label{eq:mean_constraint}
	4\mean{z_5} - \mean{z_1}\left(3-\mean{z_1}\right)^2 > 0,
	\end{align}
	the linear equations \eqref{eq:linear_U}-\eqref{eq:linear_F} have a unique solution.
\end{thm}
\begin{proof}
		According to Remark \ref{rem-eq:constraint}, one can  easily deduce
		\begin{align}
		\label{eq:basis_cons}
	3 > \mean{z_1} > 0,\ \mean{z_5}  > \dfrac{27}{32} > 0, \ \mean{\dfrac{1}{z_5}}  =  \dfrac{z_{5,l}+z_{5,r}}{2z_{5,l}z_{5,r}} > 0.
	\end{align}
	The determinant of the coefficient matrix of
	the linear equations \eqref{eq:linear_U}-\eqref{eq:linear_F} is $D  \dfrac{2}{\mean{z_5}^4}\mean{\dfrac{1}{z_5}}$ with
	$$D :=\dfrac{1}{\meanln{z_1}} +\dfrac{1}{\meanln{3-z_1}} + \dfrac{3}{(3-z_{1,l})(3-z_{1,r})} - \dfrac{9}{4\mean{z_5}}.$$
	 The system  \eqref{eq:linear_U}-\eqref{eq:linear_F}
		has a unique solution only if the determinant of the coefficient matrix
	 is not equal to zero, i.e $D\neq 0$, based on the conditions \eqref{eq:basis_cons}. Because $3>\mean{z_1}>0$, $\meanln{z_1}$ and $\meanln{3-z_1}$ are well-defined.
Using $\mean{z_1} \geqslant \meanln{z_1} >0,
\mean{3-z_1} \geqslant \meanln{3-z_1} >0$,
  $\mean{3-z_1}^2 \geqslant (3-z_{1,l})(3-z_{1,r}) >0$ and the constraints \eqref{eq:mean_constraint} can derive the following inequality
	\begin{align*}
	D 
	&\geqslant \dfrac{1}{\mean{z_1}} +\dfrac{1}{\mean{3-z_1}} + \dfrac{3}{\mean{3-z_{1}}^2} - \dfrac{9}{4\mean{z_5}}\\
	&= \dfrac{9}{\mean{z_1}\mean{3-z_1}^2} - \dfrac{9}{4\mean{z_5}}\\
	& = \dfrac{9}{4\mean{z_1}\mean{3-z_1}^2\mean{z_5}}\left(4\mean{z_5}-\mean{z_1}\mean{3-z_1}^2\right) > 0.
	\end{align*}
	Then the proof is completed.
	\end{proof}
\begin{rmk}\rm
If  $\bm{z}_l, \bm{z}_r$ satisfy the conditions in Remark  \ref{rem-eq:constraint} and $\mean{z_5}  = \mean{T} > T_c$, the inequality \eqref{eq:mean_constraint} is obviously satisfied since $f(\mean{z_1}) = \mean{z_1}(3-\mean{z_1})^2, \mean{z_1}\in(0,3)$
achieves the maximum $4$ at $\mean{z_1} = 1$.
	\hfill$\Box$
\end{rmk}
\begin{rmk}\rm
	Because the van der waals EOS is only used to describe van der Waals gas in this paper,
	  $\nu > \nu_N$ if $T_0<T < T_c$,  as shown in Figure \ref{fig:phase_diagram}.
	  Thus,
	when $T_0<\mean{T} = T_l = T_r < T_c$   and  $\bm{z}_l, \bm{z}_r$ satisfy the constraints  in Remark  \ref{rem-eq:constraint},   $\rho_{l} < 1/\nu_N(\mean{T})$ and $\rho_r < 1/\nu_N(\mean{T})$, so that $\mean{\rho} < 1/\nu_N(\mean{T})$,  i.e. the inequality \eqref{eq:mean_constraint} is
	satisfied. The numerical tests used in Section \ref{section:Result} are consistent with the conditions given by the Theorem \ref{thm:condition}.
		\hfill$\Box$
	\end{rmk}

One can also similarly obtain different linear systems about $\widetilde{\bU}$ and $\widetilde{\bF}_1$ by using other parameter vectors, which are shown in \ref{section:DiffLinearSysm}.  However,  the existence and uniqueness of  the linear system  about $\widetilde{\bU}$ and $\widetilde{\bF}_1$  
 is the critical factor to deduce the explicit expression of the two-point symmetric EC flux, which is not trivial for the compressible Euler equations with the van der Waals EOS. Under the conditions given by Theorem \ref{thm:condition}, the system \eqref{eq:linear_U}-\eqref{eq:linear_F} has a unique solution, and the  explicit expressions of the solutions  $\widetilde{\bU}$ and $\widetilde{\bF}_1$ are given as follows
\begin{equation*}
\left\{
\begin{aligned}
& \widetilde{\bU}_{1} = \dfrac{   \mean{z_1}\left(-\dfrac{9}{8}\dfrac{1}{\mean{z_5}} + \dfrac{3}{(3-z_{1,L})(3-z_{1,R})}\right)
	+\mean{ -\dfrac{9}{8} \dfrac{z_1}{z_5} + \dfrac{3}{3-z_1}}} {\dfrac{1}{\meanln{z_1}} +\dfrac{1}{\meanln{3-z_1}} + \dfrac{3}{(3-z_{1,L})(3-z_{1,R})} - \dfrac{9}{4\mean{z_5}}},\\
& \widetilde{\bU}_{m} =  \mean{z_m}\widetilde{\bU}_{1}, \quad m = 2,3,4,\\
& \widetilde{\bU}_{5}  = \dfrac{1}{\mean{{1}/{z_5}}}\left(
-\left(\frac{1}{2}\sum\limits_{m = 2}^{4}\mean{\dfrac{z_m^2}{z_5}} - \dfrac{\mean{z_5}}{\delta\meanln{z_5} }+ \dfrac{9}{4} \mean{\dfrac{z_1}{z_5}}\right) \widetilde{\bU}_{1} + \sum\limits_{m = 2}^{4}  \mean{\dfrac{z_m}{z_5}}\widetilde{\bU}_{m} +
\frac{9}{8}{\mean{z_1}}\mean{\dfrac{z_1}{z_5}}
\right),
\end{aligned}
\right.
\end{equation*}
\begin{equation*}
\left\{
\begin{aligned}
& \widetilde{\bF}_{1,1} = \dfrac{  \mean{z_1z_2}\left(-\dfrac{9}{8}\dfrac{1}{\mean{z_5}} + \dfrac{3}{(3-z_{1,L})(3-z_{1,R})}\right)
	+\mean{ -\dfrac{9}{8} \dfrac{z_1}{z_5} + \dfrac{3}{3-z_1}}\mean{z_2}}
{\dfrac{1}{\meanln{z_1}} +\dfrac{1}{\meanln{3-z_1}} + \dfrac{3}{(3-z_{1,L})(3-z_{1,R})} - \dfrac{9}{4\mean{z_5}}},\\
& \widetilde{\bF}_{1,2} =
{\mean{z_1}}{\mean{z_5}}\mean{ -\frac{9}{8} \dfrac{z_1}{z_5} + \dfrac{3}{3-z_1}} +\mean{z_2}\widetilde{\bF}_{1,1}, \\
& \widetilde{\bF}_{1,m} =  \mean{z_m}\widetilde{\bF}_{1,1}, \quad m = 3,4,\\
&
\widetilde{\bF}_{1,5}  = \dfrac{1}{\mean{{1}/{z_5}}}
\left(-
\left(\frac{1}{2}\sum\limits_{m = 2}^{4}\mean{\dfrac{z_m^2}{z_5}} - \dfrac{\mean{z_5}}{\delta\meanln{z_5} }+ \dfrac{9}{4} \mean{\dfrac{z_1}{z_5}}\right) \widetilde{\bF}_{1,1} + \sum\limits_{m = 2}^{4}  \mean{\dfrac{z_m}{z_5}}\widetilde{\bF}_{1,m}
+\frac{9\mean{z_1z_2}}{8}\mean{\dfrac{z_1}{z_5}}
\right).
\end{aligned}
\right.
\end{equation*}
Similarly, $\tbF_{k}$ may be derived for $k = 2,3$.
\begin{rmk}\rm
	It can be easily verified that the above two-point EC flux  is consistent with $J\pd{\xi_k}{t}\bU+\sum\limits_{j=1}^3 J\pd{\xi_k}{x_j}\bF_j$ by  proving $\widetilde{\bU}$ and $\widetilde{\bF}_k$ are consistent with $\bU$ and $\bF_k$, respectively.
	If  $\left(\rho_l, \bv_l^{\mathrm{T}},p_l\right)^{\mathrm{T}} = \left(\rho_r, \bv_r^{\mathrm{T}}, p_r\right)^{\mathrm{T}} =\left(\rho, \bv^{\mathrm{T}}, p\right)^{\mathrm{T}} $, then
	\begin{align*}
	\widetilde{\bF}_{1,1}&= \dfrac{ \rho v_1\left(-\dfrac{9\rho}{(p+3\rho^2)(3-\rho)}+ \dfrac{3}{(3-\rho)^2}\right)
		+\left(  -\dfrac{9\rho}{(p+3\rho^2)(3-\rho)}\rho+ \dfrac{3}{(3-\rho)}\right)v_1}{\dfrac{1}{\rho} + \dfrac{1}{3-\rho} + \dfrac{3}{(3-\rho)^2} - \dfrac{18\rho}{(p+3\rho^2)(3-\rho)}}\\
	&= \dfrac{ \rho v_1\left(-\dfrac{18\rho}{(p+3\rho^2)(3-\rho)}+ \dfrac{3}{(3-\rho)^2}\right)
		+ \dfrac{3v_1}{(3-\rho)}}{   - \dfrac{18\rho}{(p+3\rho^2)(3-\rho)} \dfrac{3}{(3-\rho)^2} +  \dfrac{3}{\rho(3-\rho)}  }
	= \rho v_1   
{= {\bF}_{1,1} },\\
	\widetilde{\bF}_{1,2} &=
	\dfrac{(p+3\rho^2)(3-\rho)}{8}\left( -\frac{9\rho^2}{(p+3\rho^2)(3-\rho)} + \frac{3}{3-\rho}\right)+\rho v_1^2 = \rho v_1^2 + \frac{3}{8}p 
{= {\bF}_{1,2} }, \\
	\widetilde{\bF}_{1,3} &= \rho v_1 v_2  
{= {\bF}_{1,3}}, \quad 	\widetilde{\bF}_{1,4} = \rho v_1 v_3
{ ={\bF}_{1,4}},  \\
	\widetilde{\bF}_{1,5} &= \left(-\frac{1}{2}\abs{\bv}^2 + \dfrac{(3-\rho)(p+3\rho^2)}{8\delta\rho }- \frac{9}{4}\rho\right) \rho v_1 + v_1\left(\rho \abs{\bv}^2 + \frac{3}{8}p\right) + \frac{9}{8}\rho^2 v_1 = v_1\left(E+\frac{3}{8}p\right)
{= {\bF}_{1,5},}
	\end{align*}
	and the consistency of
	$\widetilde{\bU}$ and  $\widetilde{\bF}_k, k = 2,3$, can be similarly deduced.
	\hfill$\Box$
\end{rmk}
Using the linear combination of the two point EC fluxes \eqref{eq:ECfluxMM}, the semi-discrete  $2w$th-order EC schemes \eqref{eq:semi_U}-\eqref{eq:semi_J} can be constructed with the following $2w$th-order EC  fluxes  according to \cite{duan2021highorder}
		\begin{align}\label{eq:ECFlux_curv_2p}
	&{\left(\thF_k\right)}_{\bm{i},k,+\frac12}^{{2w\rm{th}}}
	=~\sum_{m=1}^w\alpha_{w,m}\sum_{s=0}^{m-1}\thF_k\left(\bU_{\bm{i},k,-s}, \bU_{\bm{i},k,-s+m},
	\left(J\pd{\xi_k}{\zeta}\right)_{\bm{i},k,-s}, \left(J\pd{\xi_k}{\zeta}\right)_{\bm{i},k,-s+m}\right),\\
	\label{eq:GCL_flux}
	&\left(\widehat{J\pd{\xi_k}{\zeta}}\right)_{\bm{i},k,+\frac12}^{2w\rm{th}}
	=\sum_{m=1}^w\alpha_{w,m}\sum_{s=0}^{m-1}\dfrac12\left(\left(J\pd{\xi_k}{\zeta}\right)_{\bm{i},k,-s}
	+\left(J\pd{\xi_k}{\zeta}\right)_{\bm{i},k,-s+m}\right) ,\quad\zeta=t,x_1,x_2,x_3,
	\end{align}
 where
	the constants $\{\alpha_{w,m}\}$  satisfy the following conditions  \cite{Lefloch2002Fully}
	$$
	\sum\limits_{m=1}^{w} m\alpha_{w, m}=1, \quad \sum\limits_{m=1}^{w} m^{2 s-1} \alpha_{w, m}=0, \
 \ s=2, \ldots, w.
	$$

\subsection{Discrete GCLs}\label{subsection:GCLs}
 This section constructs  {the} discrete GCLs, which need to be fulfiled, otherwise they  may lead to a misrepresentation of the convective velocities and extra sources or sinks in the physically conservative media  \cite{Zhang1993Discrete}. Moreover, the discrete GCLs play an important role in constructing the semi-discrete  EC/ES schemes \eqref{eq:semi_U}-\eqref{eq:semi_J} satisfying the semi-discrete entropy conditions.
Utilizing the $2w$th-order accurate discretizations for $J\pd{\xi_k}{x_j},~ k,j = 1,2,3,$ given in \cite{duan2021highorder},
for example,
\begin{equation}\label{eq:SCLCoeff}
\begin{aligned}
&\left(J\pd{\xi_1}{x_1}\right)_{\bm{i}}
=\left(\pd{x_2}{\xi_2}\pd{x_3}{\xi_3}-\pd{x_2}{\xi_3}\pd{x_3}{\xi_2}\right)_{\bm{i}}
=\dfrac{1}{\Delta\xi_2\Delta\xi_3}
\left(\delta_3\left[\delta_2\left[x_2\right]x_3\right]-\delta_2\left[\delta_3\left[x_2\right]x_3\right]\right)_{\bm{i}},
\end{aligned}
\end{equation}
with the $2w$th-order central difference operator  in the $\xi_k$-direction
\begin{align*}
\delta_k[a_{\bm{i}}]=\dfrac12\sum_{m=1}^w\alpha_{w,m}\left(a_{\bm{i},k,+m} - a_{\bm{i},k,-m}\right).
\end{align*}
and
combining \eqref{eq:SCLCoeff} with the $2w$th-order discretizations of the fluxes in \eqref{eq:GCL_flux}  can easily  obtain the discrete SCLs \eqref{eq:SCL_dis}.

Due to the transformation \eqref{eq:transf}, the metrics $J\pd{\xi_k}{t}$ can be efficiently  approximated  by
\begin{equation}\label{eq:VCLCoeff}
\left(J\pd{\xi_k}{t}\right)_{\bm{i}}=-\sum_{j=1}^3(\dot{x}_j)_{\bm{i}}\left(J\pd{\xi_k}{x_j}\right)_{\bm{i}},~k=1,2,3,
\end{equation}
with $\left(J\pd{\xi_k}{x_j}\right)_{\bm{i}}$ calculated by \eqref{eq:SCLCoeff},
 and the mesh velocities $(\dot{x}_j)_{\bm{i}}$, $j=1,2,3$,  determined by the adaptive moving mesh strategy shown in Section \ref{section:MM}.   Combining them with the  fluxes
 \eqref{eq:GCL_flux} can obtain the high-order semi-discrete VCL \eqref{eq:semi_J}.

At the end of this section, the steps of  deriving the   $2w$th-order semi-discrete EC scheme \eqref{eq:semi_U}-\eqref{eq:semi_J}  in  curvilinear coordinates are  summarized  as follows.
\begin{itemize}
	\item[1.]   The two-point EC flux \eqref{eq:ECfluxMM} in  curvilinear coordinates is first explicitly derived based on the given entropy pair and carefully chosen parameter vector. 
	\item[2.] Construct the $2w$th-order accurate discretizations for $\left(J\pd{\xi_k}{x_j}\right)_{\bm{i}}$ using \eqref{eq:SCLCoeff} so that the discrete SCLs hold.
 Combining $\left(J\pd{\xi_k}{x_j}\right)_{\bm{i}}$ with the mesh velocities $(\dot{x}_j)_{\bm{i}}$,  the $2w$th-order accurate discretizations for $\left(J\pd{\xi_k}{t}\right)_{\bm{i}}$ are given by \eqref{eq:VCLCoeff}.
	 \item[3.] Compute $\left(\widehat{J\pd{\xi_k}{\zeta}}\right)_{\bm{i},k,+\frac12}^{2w\rm{th}}$ by  using the linear combinations \eqref{eq:GCL_flux}, which can get $2w$th-order accurate semi-discrete VCL, and similarly obtain the $2w$th-order accurate EC flux ${\left(\thF_k\right)}_{\bm{i},k,+\frac12}^{{2w\rm{th}}}$ by combining with the  linear combination of two-point EC flux \eqref{eq:ECFlux_curv_2p}.
	
\end{itemize}

\subsection{ES schemes}\label{sub: van der WaalsES}
For the quasi-linear  hyperbolic conservation laws,
 the entropy identity holds only when the solution is smooth
  {, implying }that EC schemes work well only for smooth solutions.
  
 In order to capture the discontinuous solutions and 
 to suppress the numerical oscillations which may produce by the EC schemes near the discontinuities,
  the high-order accurate ES schemes satisfied the semi-discrete entropy inequality for the given entropy pair should be constructed. 
%
  Following \cite{duan2021highorder},
  the following flux  given by adding  suitable dissipation term  to the EC flux is ES
\begin{equation}\label{eq:HOstable}
\left({{\thF}_k}\right)_{\bm{i}, k,+\frac12}^{\nu{\rm th}}
=\left({\thF_k}\right)_{\bm{i}, k,+\frac12}^{2w{\rm th}}
-\dfrac12\left(\Lambda \bT^{-1}\bm{R}(\bT\bU)\bm{Y}
\right)_{\bm{i}, k,+\frac12}
\jump{\bm{W}}_{\bm{i}, k,+\frac12}^{\tt {WENOMR}},
\end{equation}
where
$\Lambda:=\max\limits_{m}\left\{ \left| J\pd{\xi_k}{t}+L_k\lambda_m(\bT\bU) \right|\right\}
$ with
  $L_{k}=\sqrt{\sum\limits_{j=1}^{3}\left(J \frac{\partial \xi_{k}}{\partial x_{j}}\right)^{2}}$,   $\lambda_m$ is the eigenvalues of the matrix $\pd{\bF_1}{\bU}$,
  the rotational matrix $\bT$ is defined by
	\begin{align*}
	&\bT =
	\begin{bmatrix}
	1 & 0                      & 0                      & 0           & 0 \\
	0 & \cos\varphi\cos\theta  & \cos\varphi\sin\theta  & \sin\varphi & 0 \\
	0 & -\sin\theta            & \cos\theta             & 0           & 0 \\
	0 & -\sin\varphi\cos\theta & -\sin\varphi\sin\theta & \cos\varphi & 0 \\
	0 & 0                      & 0                      & 0           & 1
	\end{bmatrix},\\
&\theta = \arctan\left(\left(J\pd{\xi_k}{x_2}\right)\Big/\left(J\pd{\xi_k}{x_1}\right)\right),\\
&\varphi = \arctan\left(\left(J\pd{\xi_k}{x_3}\right)\Bigg/\sqrt{\left(J\pd{\xi_k}{x_1}\right)^2
+\left(J\pd{\xi_k}{x_2}\right)^2}\right),
\end{align*}
and $\bm{R}$ is the scaled right eigenvector matrix  corresponding to $\pd{\bF_1}{\bU}$, satisfying
\begin{equation}\label{eq:eigen}
\pd{\bU}{\bV}=\bm{R}\bm{R}^\mathrm{T},\quad
\pd{\bF_1}{\bU}=\bm{R}\bm{\Lambda}\bm{R}^{-1}.
\end{equation}
 Let us derive the explicit expression of  $\bm{R}$.  The Jacobian $\pd{\bF_1}{\bU} $ is given  as follows
$$
\left[\begin{array}{ccccc}
0 & 1 &0 &0& 0 \\
c_{s}^2-v_1^2 - \frac{3p_e}{8 \rho} (H - |\bv|^2)& 2v_1  -  \frac{3p_e}{8 \rho} v_1& -  \frac{3p_e}{8 \rho} v_2 &-  \frac{3p_e}{8 \rho} v_3 & \frac{3p_e}{8 \rho} \\
-v_1 v_2& v_2&v_1 &0 & 0\\
-v_1 v_3& v_3  &0&v_1 &0\\
{v_1(c_s^2 - H) - \frac{3p_e}{8 \rho} v_1{(H - |\bv|^2) }}& H - v_1^2\frac{3p_e}{8 \rho}  & -\frac{3p_e}{8 \rho} v_1v_2& -\frac{3p_e}{8 \rho} v_1v_3& v_1 + \frac{3p_e}{8 \rho} v_1\\
\end{array}\right],
$$
where  $H = {(E+3/8p)}/{\rho}$ is  the total enthalpy, and $p_e$ denotes the partial derivative of $p(\rho, e)$ with respect to $e$, i.e. $p_e = \dfrac{8\delta \rho}{3-\rho}.$ The eigenvalues $\lambda_{m}, m = 1, \cdots, 5$ of matrix $\pd{\bF_1}{\bU}$ are given by
$$\left\{ v_1 - c_s,  v_1, v_1,v_1,  v_1+c_s\right\}, $$
and the right eigenvector matrix
$$
\widetilde{\bm{R}}_1 =
\left[\begin{array}{ccccc}
1 & 1 &0 &0& 1 \\
v_1-c_{s} & v_1 & 0 &0 & v_1+c_{s}\\
v_2& v_2& 1&0 &v_2\\
v_3 &v_3     &0&1     &v_3\\
{{H - v_1 c_s }}&  {H- \dfrac{8c_s^2\rho}{3p_e}} & v_2&v_3& { {H + v_1 c_s}}\\
\end{array}\right],
$$
 satisfies $\dfrac{\partial \bU}{ \partial \bV} = \bm{R}\bm{R}^\mathrm{T}=\widetilde{\bm{R}}_1 \bm{S}\bm{S}^{\rm{T}}\widetilde{\bm{R}}_1^{\rm{T}} $ with  the diagonal matrix $\bm{S}.$  Under the assumption that the numerical solution satisfies the conditions in Remark \eqref{rem-eq:constraint}, $\bm{S}$  can be expressed after calculation as follows
 $$
 \left[\begin{array}{ccccc}
 \dfrac{\rho T}{2 c_s^2}& 0 & 0&0 &0\\
 0 & \dfrac{4\delta \rho T^2}{(4T-\rho(3-\rho)^2 )c_s^2}& 0 &0&0\\
 0 & 0 & {\rho T }& 0&0\\
 0 & 0&0 & {\rho T }& 0\\
 0 & 0&0 & 0&	\dfrac{\rho T}{2 c_s^2}
 \end{array}\right]^{\frac{1}{2}},
 $$
so that  the scaled right eigenvector matrix $\bm{R}$ may be taken as
$\widetilde{\bm{R}}\bm{S}.$
The values of $\bm{R}_{\bm{i}, k, +\frac{1}{2}} $ and $\Lambda_{\bm{i}, k, +\frac{1}{2}}$ are
calculated by  some averaged values of the primitive variables in the numerical tests such as
\begin{align*}
\overline{\rho} = \meanln{\rho}_{\bm{i},k,+\frac{1}{2}}, ~ \overline{\bv} = \mean{\bv}_{\bm{i},k,+\frac{1}{2}}, ~ \overline{p}  = \left(\dfrac{ \meanln{\rho}}{ \meanln{\rho/p}}\right)_{\bm{i},k,+\frac{1}{2}}.
\end{align*}
To achieve  high-order accuracy, the fifth-order multi-resolution WENO  reconstruction   \cite{WANG2021105138} is applied to the scaled entropy variables $\bm{W} = \bm{R}^{\rm{T}}(\bm{TU})\bm{TV}$ to obtain the
 left and right limit values $\bm{W}_{\bm{i},k,+\frac12}^{\tt {WENOMR},-}$ and  $\bm{W}_{\bm{i},k,+\frac12}^{\tt {WENOMR},+}$ in the jump terms $\jump{\bm{W}}_{\bm{i}, k, +\frac12}^{\tt {WENOMR}} := \bm{W}_{\bm{i},k,+\frac12}^{\tt {WENOMR},+} -
 \bm{W}_{\bm{i},k,+\frac12}^{\tt {WENOMR},-}
 $  in \eqref{eq:HOstable}.
Due to  the fact that multi-resolution WENO reconstruction   may not  preserve the ``sign''  property  proposed in \cite{Biswas2018Low},
the diagonal matrix $\bm{Y}_{\bm{i},k,+\frac12}$ is  chosen as the following form
\begin{align*}
&\left(\bm{Y}_{m,m}\right)_{\bm{i},k,+\frac12} =
\left\{\begin{array}{ll}1, & \text{if} ~ {\rm{sign}}\left(\jump{\bm{W}_m}_{\bm{i}, k, \pm\frac12}^{\tt {WENOMR}} \right)
	{\rm{sign}}\left(
	\jump{\bm{W}_m}_{\bm{i}, k, \pm\frac12}
	\right)
	>0, \\ 0, & \text {otherwise},\end{array}\right.
\end{align*}
to make sure the flux given in \eqref{eq:HOstable} is ES.
\begin{rmk}\rm
		The free-stream condition (if the
initial value is constant,  so is the solution of the numerical scheme at following instant)  is an important property that should be met to prevent
		large errors and even numerical instabilities  \cite{VISBAL2002155}. According to \cite{duan2021highorder},   it is easy to prove that our
		high-order accurate fully-discrete adaptive moving mesh
		EC/ES finite difference  schemes  satisfy
		the free-stream condition with the third-order accurate explicit SSP-RK schemes \cite{Gottlieb2001Strong} for time discretization.
		\hfill$\Box$
\end{rmk}

\section{Adaptive moving mesh strategy}\label{section:MM}
This section introduces briefly  the adaptive moving mesh strategy  mentioned in \cite{duan2021highorder} 
for integrity, and especially concentrates on the selection of   monitor function which decides the way how mesh clusters. Unless otherwise stated, the dependence of the variables on $t = \tau$ from transformation \eqref{eq:transf}  will be omitted. Following  the Winslow variable diffusion method \cite{Winslow1967Numerical}, the adaptive  mesh redistribution is implemented by
iteratively solving the Euler-Lagrange equations of  mesh adaption functional
\begin{equation}\label{eq:mesh_EL}
\nabla_{\bm{\xi}}\cdot\left(\omega\nabla_{\bm{\xi}}x_k\right)=0,
~\bm{\xi}\in\Omega_c,~k=1,2,3,
\end{equation}
where the  positive monitor function  $\omega$ is related to some
physical variables  controlling the location of the mesh points, taken as
\begin{align}\label{eq:van der Waalsmonitor}
\omega=\left(1+\sum_{k=1}^{\kappa}\alpha_k\dfrac{
	{\abs{\nabla_{\bm{\xi}}\delta_k}^2}}{{\max\abs{\nabla_{\bm{\xi}}\delta_k}^2}} + \alpha_{\kappa+1}\dfrac{
	{\abs{\Delta_{\bm{\xi}}G}^2}}{{\max\abs{\Delta_{\bm{\xi}}G}^2}}
 \right)^{1/2},
\end{align}
with {the}  positive constant parameter $\alpha_k$, the physical variable $\delta_k$, the number of physical variables $\kappa$,  and the  vector $\left(\frac{\partial^2 a}{\partial \xi_1^2},\cdots,\frac{\partial^2 a}{\partial \xi_d^2}\right)$ denoted by $\Delta_{\bm{\xi}}a$  in this paper.
 The appearance of the non-classical waves is often   related to the sign of
 the fundamental derivative
 $G$, see Example \ref{ex:1DRiemann }. 
{Therefore, for the van der Waals gas, the fundamental derivative
	$G$  should be considered in the construction of
	 $\omega$  to enhance the capability of capturing non-classical waves.}
{\color{black}Using the  }second-order central difference scheme and the Jacobi iteration method,
the mesh equations \eqref{eq:mesh_EL} are approximated   by
\begin{equation*}
\begin{aligned}
&\sum\limits_{k =1}^{d} \left[\omega_{\bm{i}, k, +\frac12 }\left(\bx_{\bm{i}, k,1}^{[\nu]}-\bx_{\bm{i}}^{[\nu+1]}\right)
-\omega_{\bm{i}, k, -\frac12}\left(\bx_{\bm{i}}^{[\nu+1]}-\bx_{\bm{i}, k, -1}^{[\nu]}\right)\right] / \Delta{\xi_k^2}=0, \ \nu=0,1,\cdots,\mu,
\end{aligned}
\end{equation*}
where ${\bx}^{[0]}_{\bm{i}}
:=\bx^n_{\bm{i}}$, and
$\omega_{\bm{i}, k, \pm\frac12 } :=\frac12\left(
\omega_{\bm{i},k}+\omega_{\bm{i},k,\pm1}
\right).$ 	The iteration stops when  $\abs{\bm{x}^{[\nu]}  - \bm{x}^{[\nu+1]}} < \epsilon$ or $\nu > \mu$, where $\epsilon$ is the tolerable error.  The total iteration number $\mu$ is taken as $10$ in our numerical tests, unless otherwise stated.  Moreover, the following low pass filter
\begin{align*}
\omega_{i_1,i_2,i_3}\leftarrow&\sum_{j_1,j_2,j_3=0,\pm 1}\left(\dfrac{1}{2}\right)^{\abs{j_1}+\abs{j_2}+\abs{j_3}+3}
\omega_{i_1+j_1,i_2+j_2,i_3+j_3},
\end{align*}
is used  to smooth monitor function $3\sim 10$ times in this work.

The final adaptive mesh  at $t^{n+1}$ is obtained with the limiter of the movement of mesh points ${\Delta_\gamma}$ as follows
\begin{align*}
&\quad\bx^{n+1}_{\bm{i}}
:=\bx^n_{\bm{i}}
+  
{\Delta_\gamma}   (\delta{\bx})^{n}_{\bm{i}},\quad
(\delta {\bx})^{n}_{\bm{i}}:=
{\bx}^{[\mu]}_{\bm{i}}
-\bx^n_{\bm{i}},\\
&{\Delta_\gamma}\leqslant
\begin{cases}
-\frac{1}{2(\delta{x_k})_{\bm{i}}}\left[(x_k)^n_{\bm{i}}-(x_k)^n_{\bm{i}, k, -1}\right],
~ (\delta{x_k})_{\bm{i}}<0, \\
\quad\frac{1}{2(\delta{x_k})_{\bm{i}}}\left[(x_k)^n_{\bm{i}, k, +1}-(x_k)^n_{\bm{i}}\right],
~ (\delta{x_k})_{\bm{i}}>0,
\end{cases}
k = 1,2,3.
\end{align*}
Finally, the mesh velocity at $t=t_n$ in \eqref{eq:VCLCoeff} is determined by
$
\dot{\bx}^{n}_{\bm{i}}:=
{\Delta_\gamma}   (\delta {\bx})^{n}_{\bm{i}}/\Delta t_n
$
with the time step size $\Delta t_n$, obtained by \eqref{eq:cfl}.

\section{Numerical results}\label{section:Result}
This section provides some  numerical tests  for the dense gas with $\gamma = 1.0125$ to demonstrate the accuracy and effectiveness of our
fifth-order accurate adaptive moving mesh ES schemes.  
For convenience,
the fifth-order accurate   ES schemes with multi-resolution WENO on the  moving mesh are denoted by ``{\tt MM}", while those on the fixed uniform mesh are denoted by ``{\tt UM}".
%
In addition, our adaptive moving mesh ES  schemes, using the monitor function with or without the information of fundamental derivative, are employed for the 1D and 2D  cases  except for the accuracy tests to  illustrate the ability
	to capture the classical and non-classical waves.
 The corresponding scheme without the information of fundamental derivative in the monitor function is denoted by ``{\tt MM1}". 
The CPU nodes of the High-performance Computing Platform of Peking University (Linux redhat environment, two Intel Xeon E5-2697A V4  per node, and core frequency of 2.6GHz) are used for the computation,  and the schemes are implemented in parallel combining  the MPI parts of the PLUTO code \cite{2007PLUTO} with the   explicit third-order
SSP-RK   time discretization \cite{Gottlieb2001Strong}. The time step size $\Delta t_{n}$ is determined by the CFL condition
\begin{align} \label{eq:cfl}
\Delta t_{n} \leqslant \frac{\mathrm{CFL}}{\max\limits_{\bm{i}}\left\{\sum_{k =1}^{d}\varrho_{\bm{i}, k}^{n} / \Delta \xi_{k} \right\}},
\end{align}
where $\varrho_{\bm{i}, k}$  is the spectral radius of the eigen-matrix in the $\xi_k$-direction. Unless particularly  stated,
 CFL  is taken as $0.4$
  {for the case of $d= 1$ and $d =2$}, and $0.3$   {for $d= 3$} in numerical experiments.   The linear weights in multi-resolution WENO reconstruction are taken as $\lambda_0 = 0.95, \lambda_1 = 0.045$ and $\lambda_2 = 0.005$.
\begin{example}[1D Smooth Sine Wave]\label{ex:1Dsmooth} \rm
	The  sine wave propagating periodically is used to test the accuracy of {\tt MM}. The   exact solution is given by
 $
(\rho, v_1, p)  (x_1,t)= (0.5 + 0.2\sin(2\pi(x_1 - t)), 1, 1)$ for $t\geqslant0$.
The  domain $\Omega_p$ is $[0,1]$ with periodic boundary condition and the time step size $\Delta t_n$ is taken as CFL$ \Delta\xi_{1}^{5/3}$ to make the spatial error dominant.
	The monitor function  is chosen as
	\begin{align} \label{eq:van der Waalssinmonitor}
	\omega=\sqrt{1+\frac{\alpha_1\abs{\nabla_{\bm{\xi}} \rho}}{\max \abs{\nabla_{\bm{\xi}} \rho}} +\frac{\alpha_2\abs{\Delta_{\bm{\xi}} \rho}}{\max \abs{\Delta_{\bm{\xi}} \rho}} },
	\end{align}
	with $\alpha_1 = 20, \alpha_2 = 15,$ and the low pass filter is used to smooth the monitor function $20$ times.
	Figure \ref{fig:1DSin} presents the $\ell^1$-  and $\ell^\infty$-errors and the corresponding orders of convergence in  $\rho$ at $t = 1$ by using {\tt MM} with $N_1$ cells. The results show that the {\tt MM}  achieves the fifth-order accuracy as expected.

\begin{figure}[!ht]
	\centering
	\includegraphics[width=0.33\linewidth]{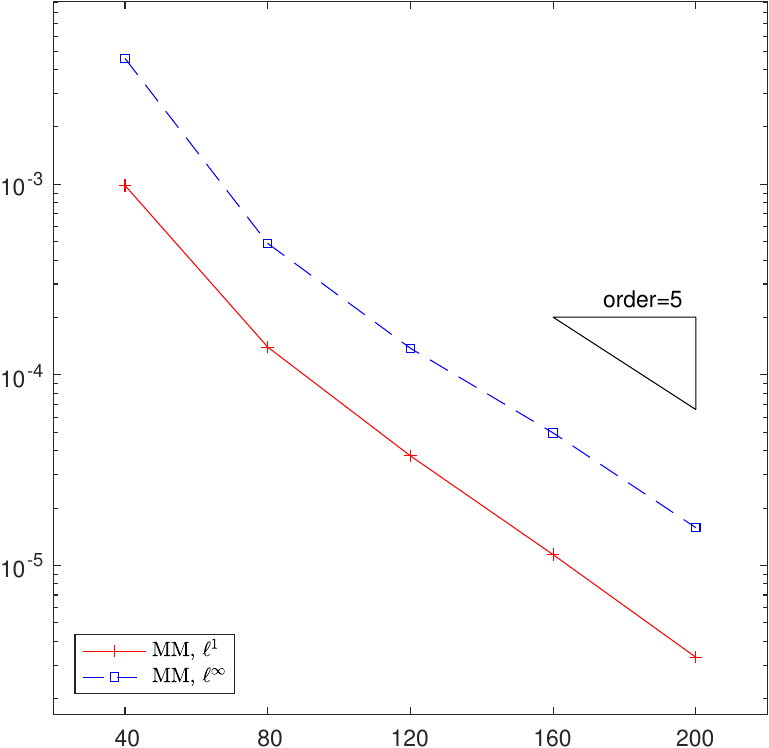}
	\caption{Example \ref{ex:1Dsmooth}:
				$\ell^1$- and $\ell^\infty$-errors  in $\rho$ with respect to $N_1$ at $t = 1.0$.
	}
	\label{fig:1DSin}
\end{figure}
\end{example}
%

\begin{example}[1D Riemann Problem]\label{ex:1DRiemann }\rm
	Several initial conditions of 1D Riemann problem are presented in Table \ref{tab:1DDenseRP}, which have been widely used in testing the
numerical schemes for the dense gas \cite{CINNELLA20061264}. The monitor function is chosen as  \eqref{eq:van der Waalsmonitor}  with $\delta_1 = \rho, \delta_2 = G $, $\alpha_1 = 1200, \alpha_2 = 3000, \kappa = 2, \alpha_3 = 5000$ for RP1 and RP2, while  $\delta_1 = \rho, \delta_2 = G $, $\alpha_1 = 1200, \alpha_2 = 1200, \kappa = 2, \alpha_3 = 1200$ for RP3. For a comparison, the counterparts of {\tt MM1} with $\alpha_2 = \alpha_3 = 0 $ are also implemented. The numerical results are computed by {\tt UM}, {\tt MM}, and {\tt MM1}, and the reference solution is computed
by using the second-order finite volume
scheme with $10000$ grid cells.

	The numerical results of RP1 at $t = 0.15 Z_c^{-1/2}$ are shown in Figure \ref{fig:DenseRP2}.
	The initial fundamental derivative  $G$ is positive,
	 as it evolves,  $G < 0$ emerges, forming a non-classical wave composed by a transonic rarefaction fan and a rarefaction shock wave.
	 Moreover, despite the sign of $G$ changes near $x_1= 0.74$, a compressible shock satisfying the Rankine-Hugoniot jump relations is formed.
The adaptive mesh points concentrate near the large gradient area of the
solution as expected. {\tt MM} with $N_1 = 100$ is better than {\tt MM1} with $N_1 = 100$ near  the contact
discontinuity, and superior to {\tt UM} with $N_1 = 300$ near the rarefaction shock wave,  the contact
discontinuity and the shock wave.
	
	RP2 is a classical example for the dense gas with  non-classical behavior. The fundamental derivative
keeps negative during the flow evolution in the whole flow field.
 Figure \ref{fig:DenseRP3} presents the numerical solution of $\rho$, $v_1$, $p$ and $G$ at $t = 0.45Z_c^{-1/2}$.
 One can see that the solution consists of
a rarefaction shock wave, a contact discontinuity and a right-running compression fan.
 The numerical results  show that {\tt MM} with $N_1 = 100$ gives  sharper transitions  than {\tt UM} with $N_1 = 300$ and {\tt MM1} with $N_1 = 100$ near the rarefaction shock wave and  the contact discontinuity  where the mesh points concentrate.

 {The last example RP3 contains a $G < 0$ region and a $G > 0$ region during the flow evolution.
In the region of the fundamental derivative changing its sign,}
a mixed rarefaction wave forms at $ t = 0.2Z_c^{-1/2}$ shown in Figure \ref{fig:DenseRP4},  while the shock wave is generated in $G > 0$ region. Figure \ref{fig:DenseRP4} shows that  {\tt MM} with $N_1 = 100$ gives  better results than {\tt MM1} with the same number of cells and {\tt UM} with $N_1 = 300$, verifying the ability in capturing the localized structures of  {\tt MM}.

	\begin{table}[h]
		\centering
		\begin{tabular}{l|cccccccc}
		\hline  & $\rho_{l}$ & $v_{1,l}$ & $p_{l}$ & $G_{l}$ & $\rho_{r}$ & $v_{1,r}$ & $p_{r}$ & $G_{r}$ \\\hline
		RP1 & $1.818$ & $0.0$ & $3.000$ & $4.118$ & $0.275$ & $0.0$ & $0.575$ & $0.703$ \\
		RP2 & $0.879$ & $0.0$ & $1.090$ & $-0.031$ & $0.562$ & $0.0$ & $0.885$ & $-4.016$ \\
		RP3 & $0.879$ & $0.0$ & $1.090$ & $-0.031$ & $0.275$ & $0.0$ & $0.575$ & $0.703$ \\
		\hline
	\end{tabular}
		\caption{Example \ref{ex:1DRiemann }. Initial data of the Riemann problem for the one-dimensional shock tube.}
		\label{tab:1DDenseRP}
	\end{table}

	\begin{figure}[!ht]
		\centering
		\begin{subfigure}[b]{0.3\textwidth}
			\centering
			\includegraphics[width=1.0\linewidth]{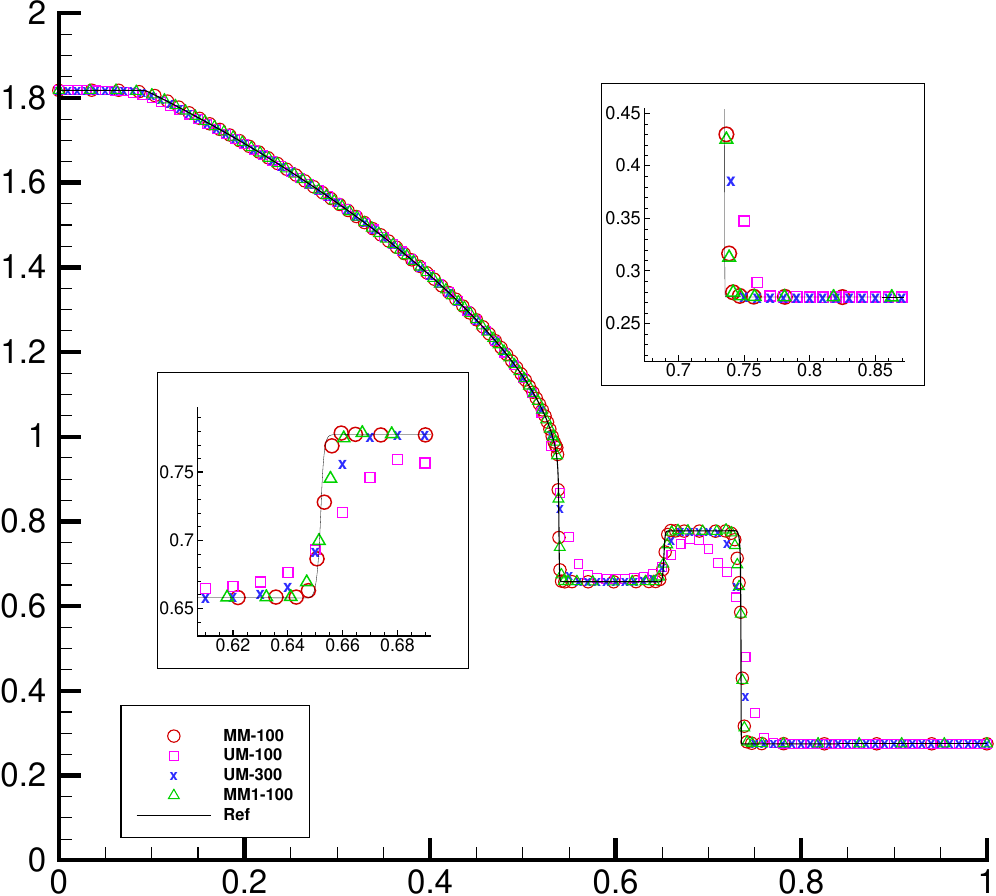}
			\caption{$\rho$}
		\end{subfigure}
		\begin{subfigure}[b]{0.3\textwidth}
			\centering
			\includegraphics[width=1.0\linewidth]{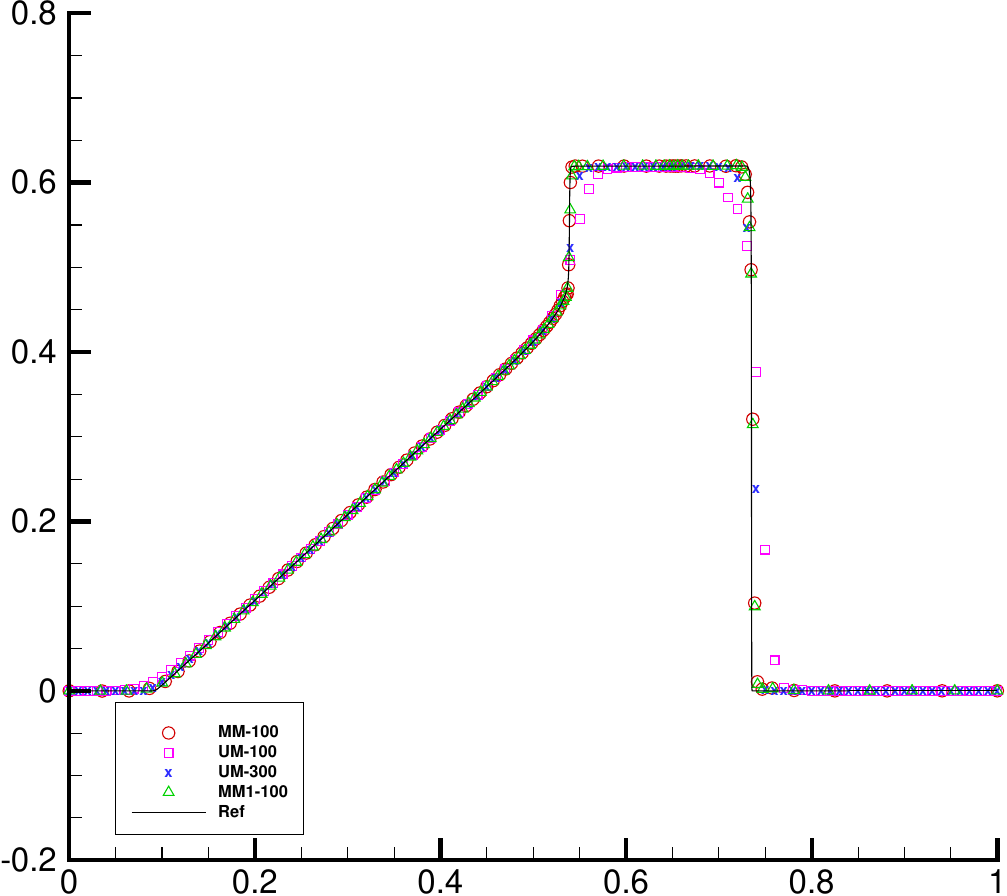}
			\caption{$v_1$}
		\end{subfigure}
		
		\begin{subfigure}[b]{0.3\textwidth}
			\centering
			\includegraphics[width=1.0\linewidth]{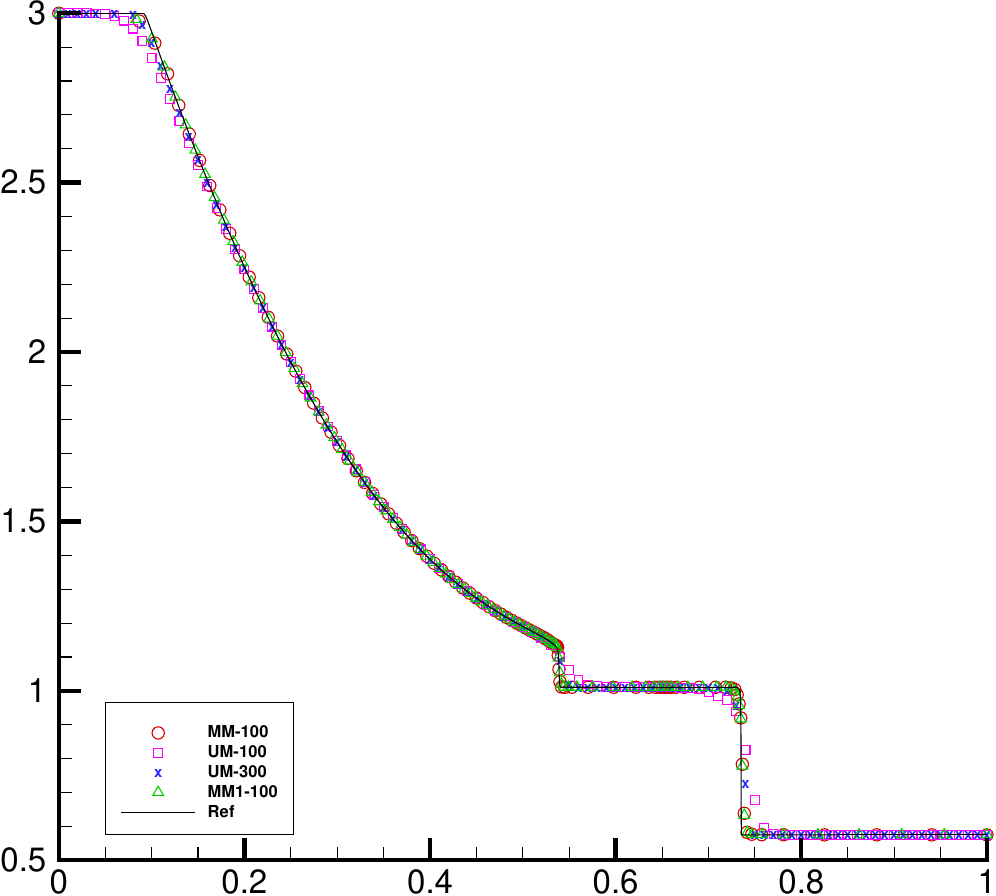}
			\caption{$p$}
		\end{subfigure}
		\begin{subfigure}[b]{0.3\textwidth}
			\centering
			\includegraphics[width=1.0\linewidth]{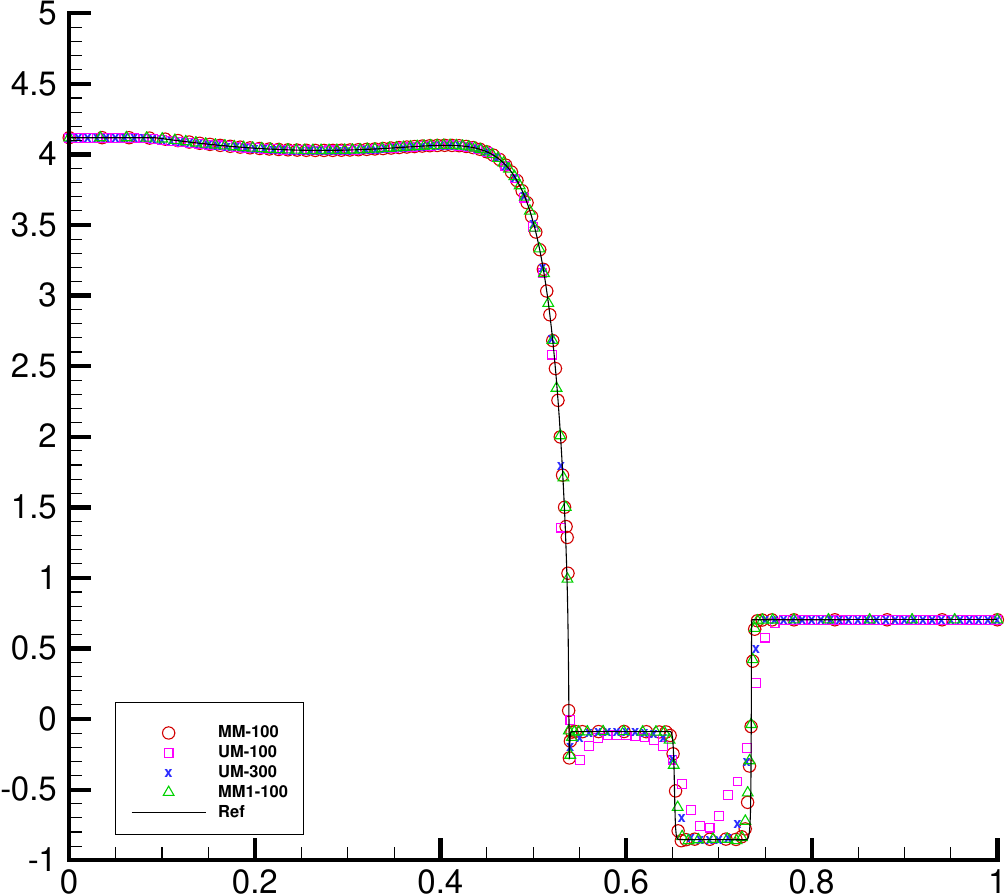}
			\caption{$G$}
		\end{subfigure}
	\caption{The numerical results of RP1 at $t = 0.15Z_c^{-1/2}.$}
	\label{fig:DenseRP2}
	\end{figure}

	\begin{figure}[!ht]
		\centering
	\begin{subfigure}[b]{0.3\textwidth}
	\centering
	\includegraphics[width=1.0\linewidth]{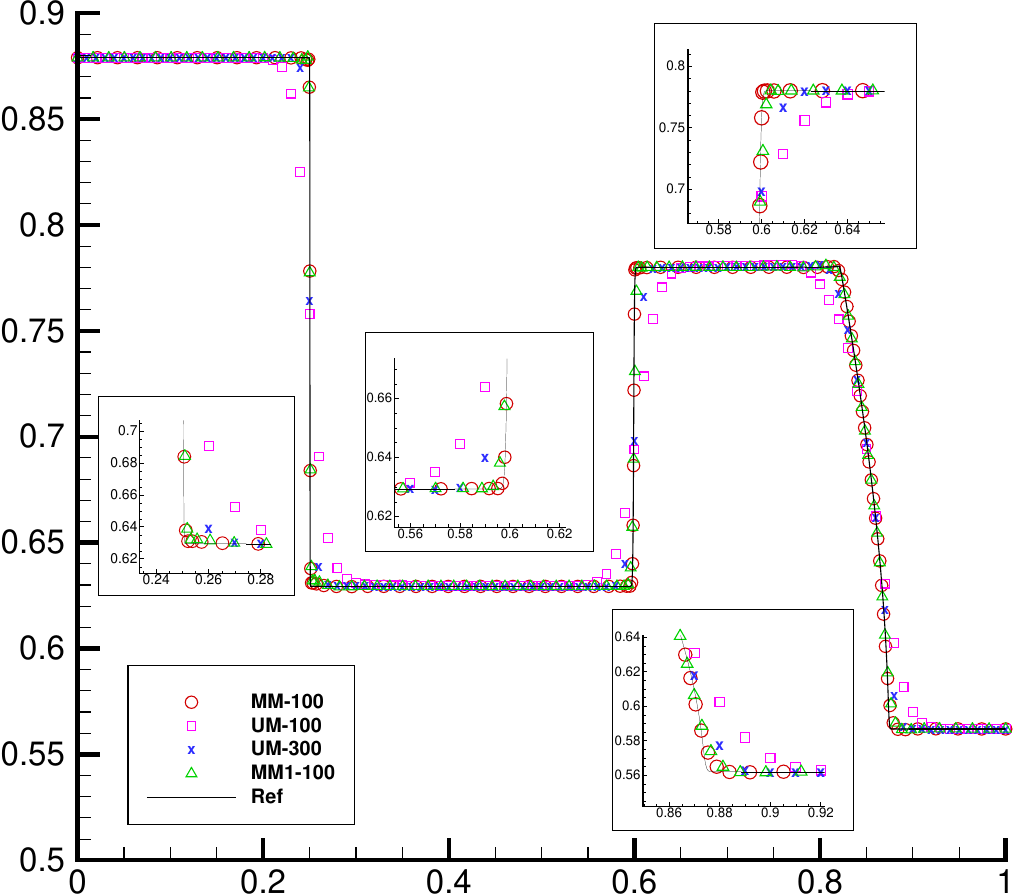}
	\caption{$\rho$}
\end{subfigure}
\begin{subfigure}[b]{0.3\textwidth}
	\centering		
	\includegraphics[width=1.0\linewidth]{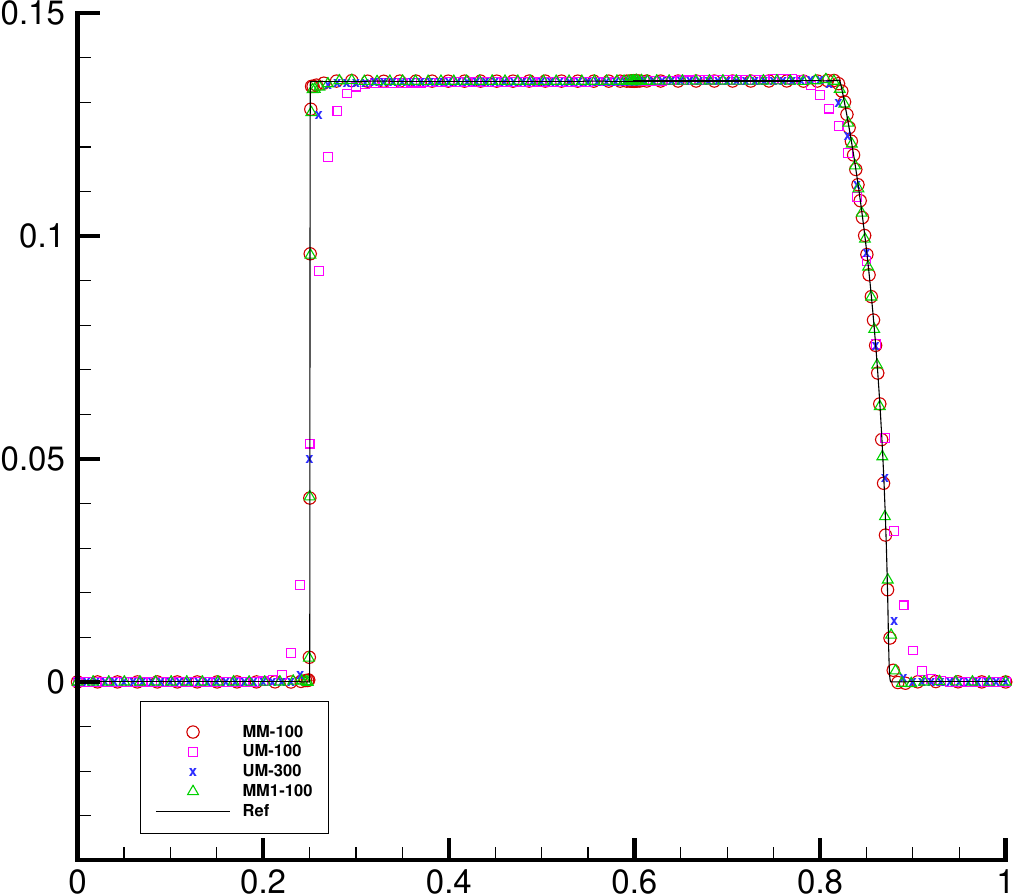}
	\caption{$v_1$}
\end{subfigure}

\begin{subfigure}[b]{0.3\textwidth}
	\centering
	\includegraphics[width=1.0\linewidth]{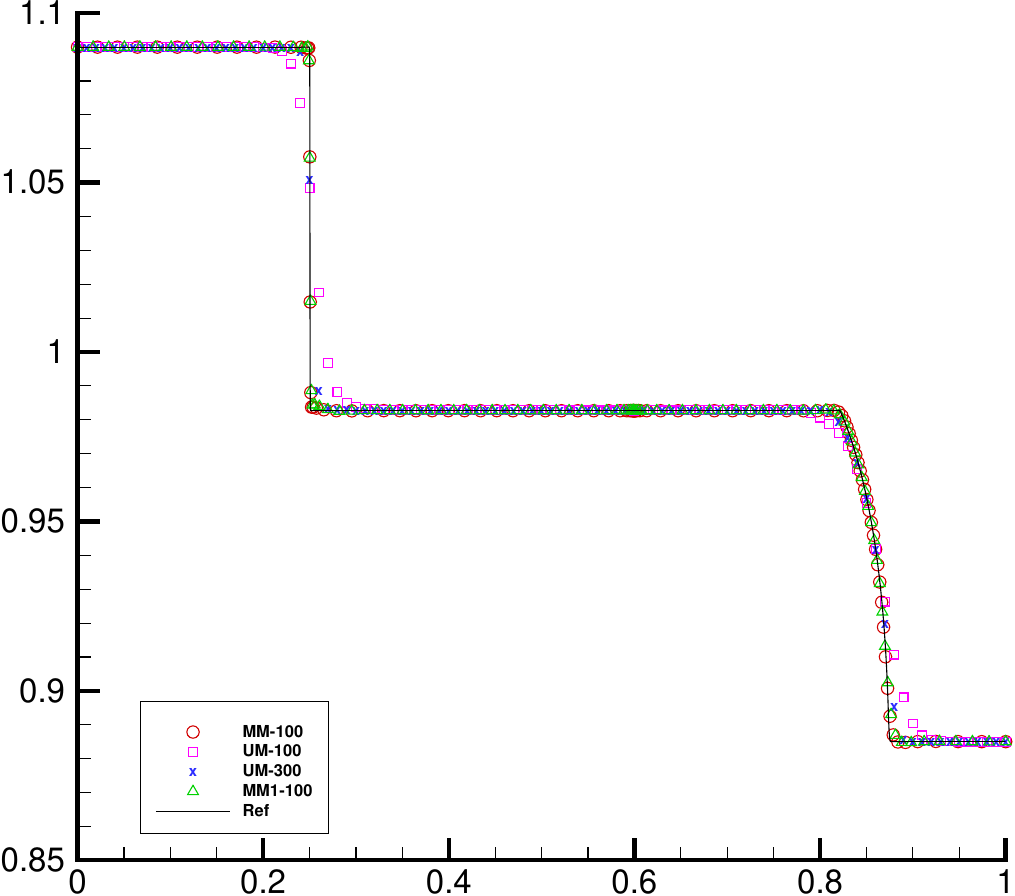}
	\caption{$p$}
\end{subfigure}
\begin{subfigure}[b]{0.3\textwidth}
	\centering
	\includegraphics[width=1.0\linewidth]{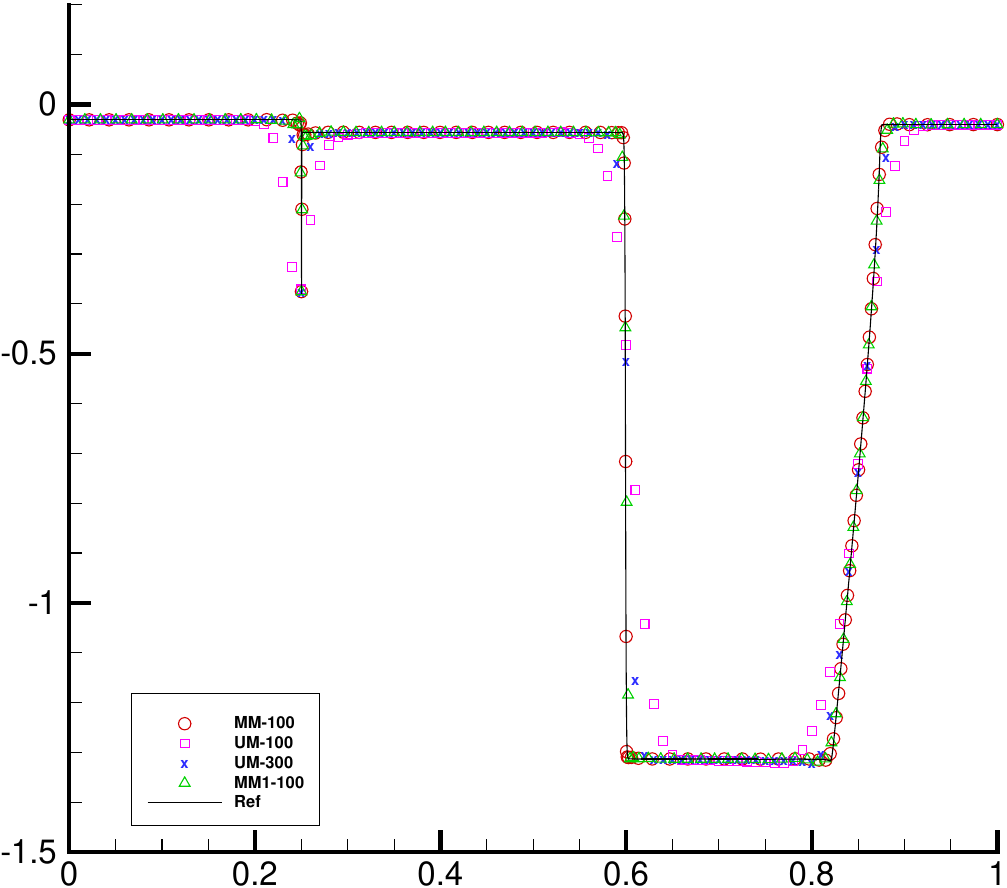}
	\caption{$G$}
\end{subfigure}
		\caption{ The numerical results of RP2 at $t = 0.45Z_c^{-1/2}.$}
		\label{fig:DenseRP3}
	\end{figure}
	
	\begin{figure}[!ht]
		\centering
		\begin{subfigure}[b]{0.3\textwidth}
		\centering
		\includegraphics[width=1.0\linewidth]{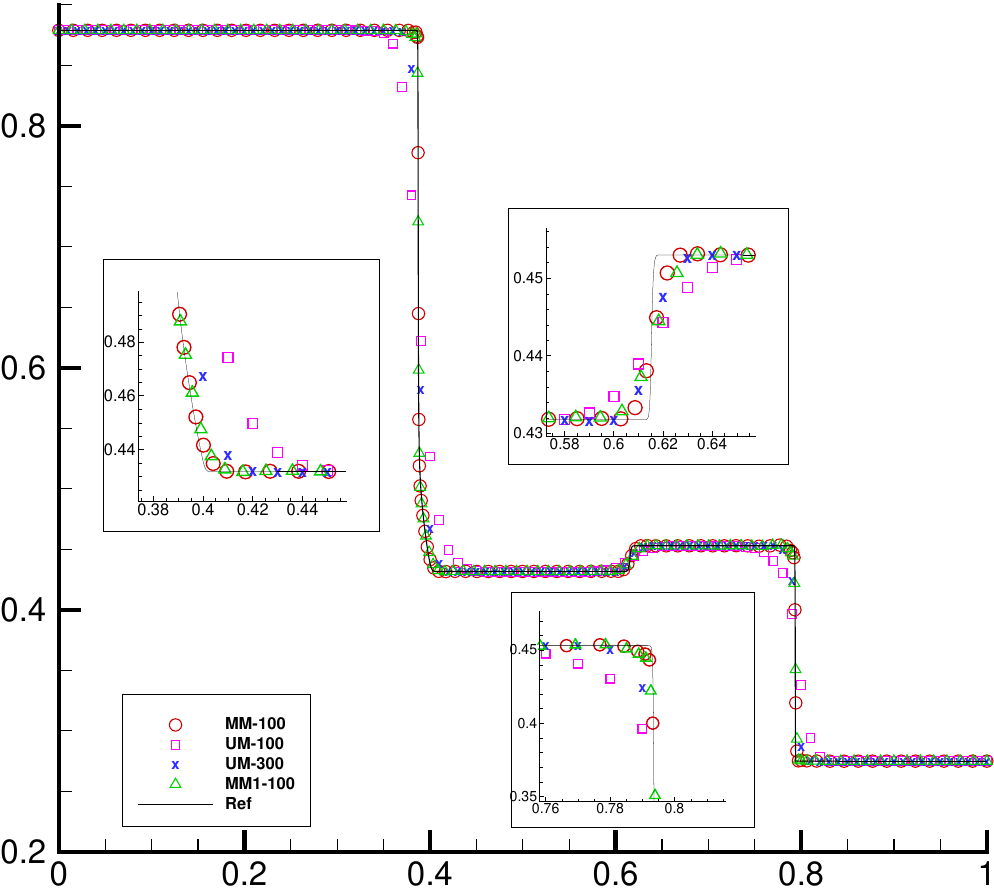}
		\caption{$\rho$}
	\end{subfigure}
	\begin{subfigure}[b]{0.3\textwidth}
		\centering
		\includegraphics[width=1.0\linewidth]{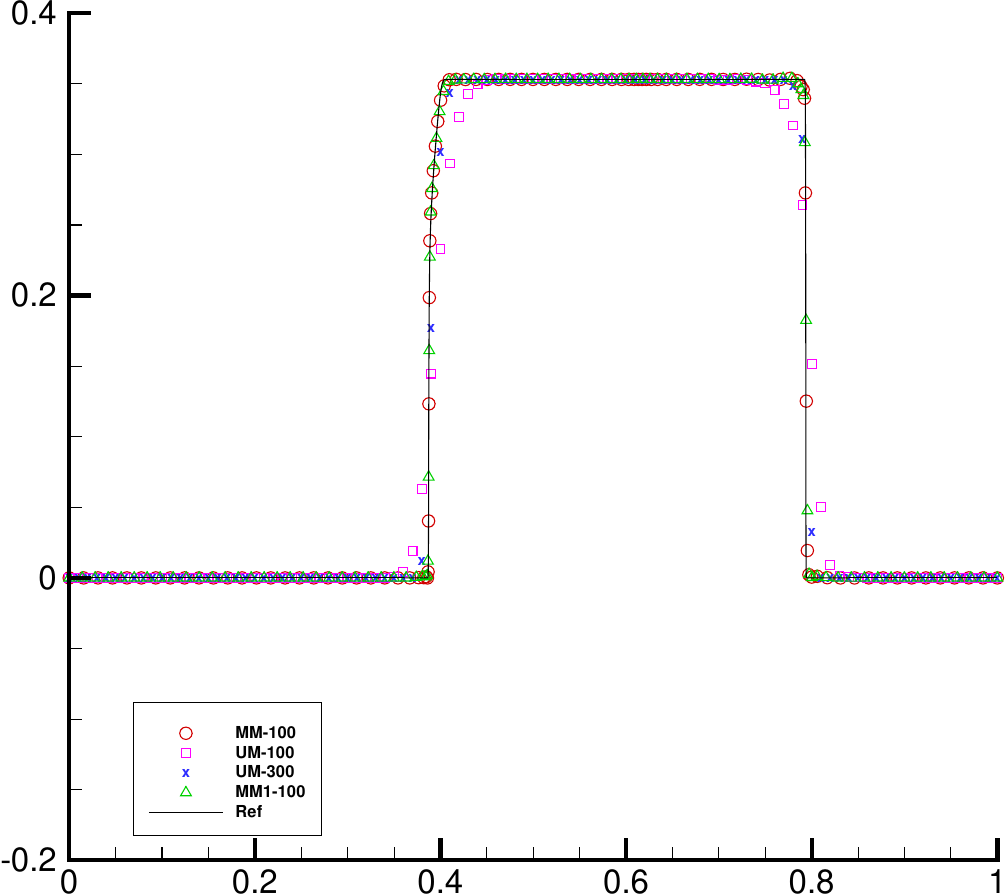}
		\caption{$v_1$}
	\end{subfigure}
	
	\begin{subfigure}[b]{0.3\textwidth}
		\centering
		\includegraphics[width=1.0\linewidth]{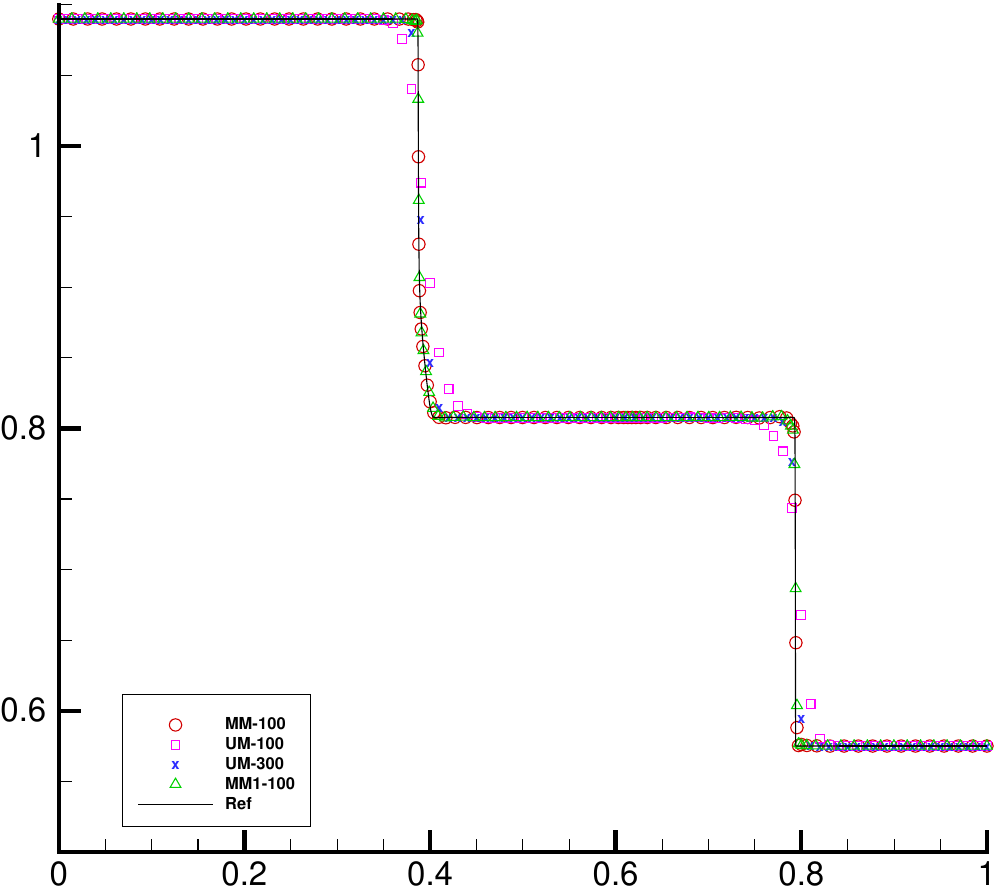}
		\caption{$p$}
	\end{subfigure}
	\begin{subfigure}[b]{0.3\textwidth}
		\centering
		\includegraphics[width=1.0\linewidth]{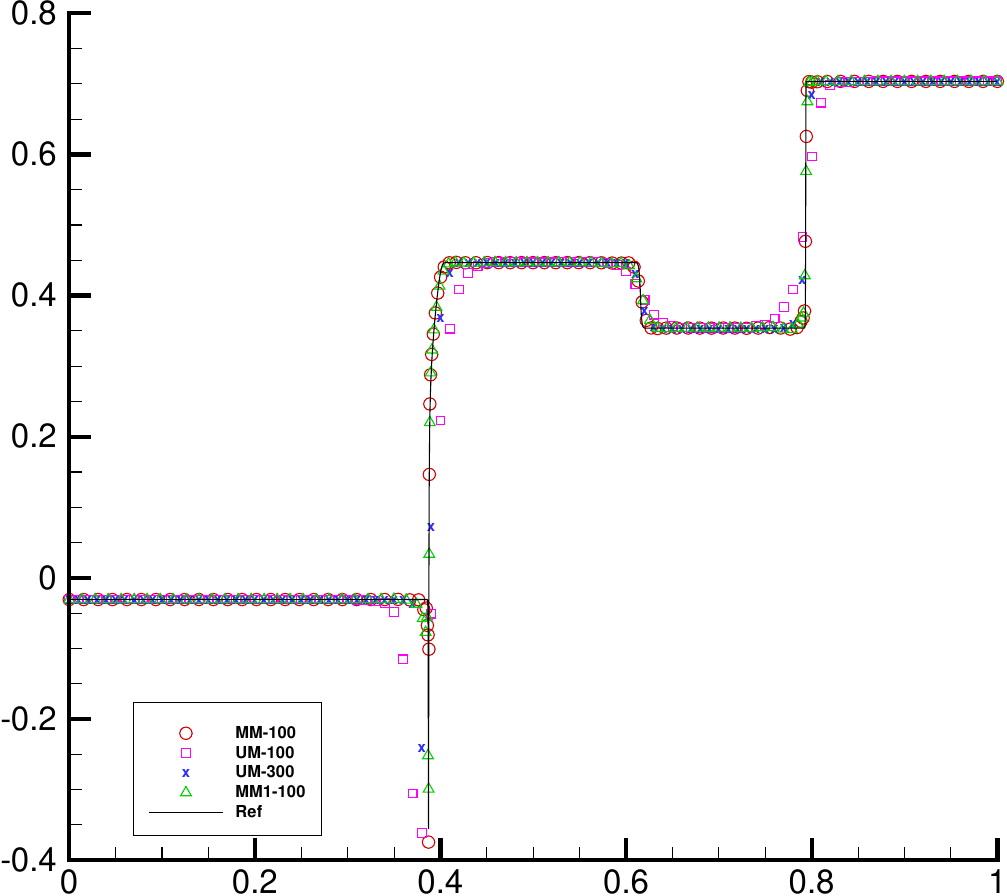}
		\caption{$G$}
	\end{subfigure}
		\caption{ The numerical results of RP3 at $t = 0.2Z_c^{-1/2}.$}
		\label{fig:DenseRP4}
	\end{figure}
\end{example}

\begin{example}[2D Smooth Sine Wave]\label{ex:2Dsmooth} \rm
	Similar to Example \ref{ex:1Dsmooth}, this case is employed to evaluate the accuracy of  2D {\tt MM}. The exact solutions  are
	$(\rho, v_1, v_2, p) (x_1,x_2,t) = (0.5 + 0.2\sin(2\pi(x_1+x_2 - t)), 0.5, 0.5, 1)$ for $t\geqslant0$.
	The
domain $\Omega_p$ is $[0, 1]\times [0,1]$  divided into $ N_1 \times N_1$  cells, and  periodic boundary conditions
are specified on each boundary.
The monitor function is chosen as Example \ref{ex:1Dsmooth} with $\alpha_1 = 20, \alpha_2 = 20,$ and the number of smooth times is the same as Example \ref{ex:1Dsmooth}.
{\color{black}
The boundary points move adaptively  according to the  periodic boundary conditions.}
Figure \ref{fig:2DSin} shows the $\ell^1$- and $\ell^\infty$-errors and the corresponding orders of convergence in  $\rho$ at $t = 1$ with  {\tt MM}, validating the fifth-order accuracy of {\tt MM}.
	\begin{figure}[!ht]
		\centering
		\includegraphics[width=0.33\linewidth]{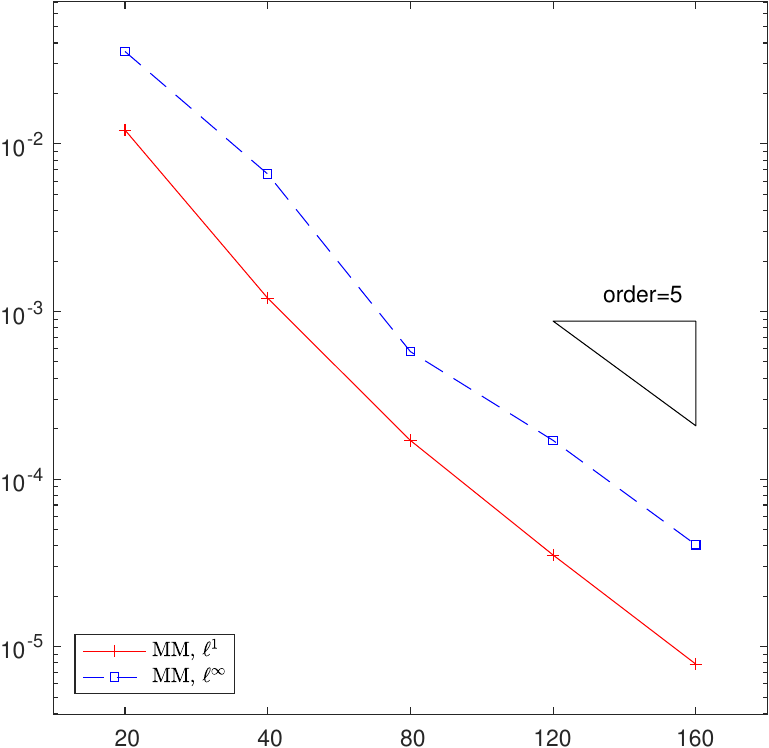}
		\caption{Example \ref{ex:2Dsmooth}.	$l^1$- and $l^\infty$-errors  in $\rho$ with respect to $N_1$ at $t = 1.0$.
		}
		\label{fig:2DSin}
	\end{figure}

\end{example}
\begin{example}[2D Mach Reflection]\label{ex:Mach}\rm
The Mach reflection is often  used to test codes for the dense gas.
 It describes an initially shock wave propagating to the right at an angle  $\theta$  to the  $x_1$-axis, forming complex reflection phenomena related to the angle of inclination and the shock Mach number
\cite{2000Two}.
 The domain $\Omega_p$ is $[0,1]\times [0,0.5] $ and
 the initial conditions are given in Table \ref{tab:Mach}. The monitor function of {\tt MM} is chosen as  \eqref{eq:van der Waalsmonitor} with
 $\kappa  = 1, \delta_1 = \rho$, $\alpha_1 = 1200$, $\alpha_2 = 500,$ and  {\tt MM1} is the same as  {\tt MM}
except for $\alpha_2 = 0.$

	The initial data of P1  describes the van der Waals gas with small specific heat. The incidence angle of the initial shock is $\theta = 20^\circ$, and the Mach  number is $M_s = 1.65$. It forms a single Mach reflection as shown in Figure \ref{fig:SMR_P1}. The slipstream is not very obvious due to the weak shock wave. The fundamental derivative keeps positive
 during the flow evolution. Figure  \ref{fig:SMR_P1}  presents the adaptive mesh of {\tt MM} with $200 \times 100$ cells, the densities obtained respectively by {\tt MM}, {\tt MM1}, and {\tt UM}, and the fundamental derivative given by {\tt MM} at $t = 1.0$. 
 It can be observed that the mesh points of the {\tt MM}  are adaptively concentrated near the incident shock wave, the reflected shock wave, and the Mach stem.
		 {\tt MM} with
	  $200 \times 100$ cells gives better results than {\tt MM1} with
	   $200 \times 100$ cells and {\tt UM} with $600 \times 300$ cells, when it only costs $35.5\%$ CPU time of {\tt UM} with the finer mesh shown in Table \ref{time:MachBack},
	   { demonstrating the efficiency of  {\tt MM}.}

	The D1 case describes an incident 
	 shock wave with $\theta = 30^\circ$ and $M_s = 1.05$. Figure  \ref{fig:SMR_D1}  gives the adaptive mesh of {\tt MM} with $200 \times 100$ cells, the densities and the fundamental derivative at $t = 0.06$.
	 It is worth noting that the wave near $x_1 = 0.52$ in Figure \ref{fig:SMR_D1} is  partially recompressed 
	 with the positive fundamental derivative.
	The solution of {\tt MM} with $200\times100$ cells gives sharper transitions near the shock waves than that on the uniform mesh with $600 \times 300$ cells, since the mesh points
adaptively concentrate near those regions according to the monitor function. {\tt MM} with $200 \times 100$ cells only  take   $28.6\%$ CPU time of the uniform mesh with $600 \times 300$ cells  shown in Table \ref{time:MachBack}, which highlight the efficiency of the adaptive moving mesh scheme.

	\begin{table}
		\centering
			\begin{tabular}{l|ccccccccc}
		\hline   &$\gamma $  & $\rho_{l}$ & $v_{1,l}$& $v_{2,l}$ & $p_{l}$ &$\rho_{r}$ & $v_{1,r}$ & $v_{2,r}$& $p_{r}$    \\
		\hline
		P1 & $1.4$ &$0.033$ & $5.016$ &$0.0$ & $3.001$ & $0.0156$ &$0.0$&$0.0$ & $1.0$   \\
		D1 & $1.0125$   & $0.629$ & $-0.135$ &$0.0$ & $0.983$ & $0.879$ & $0.0$ &$0.0$ & $1.09$  \\
		\hline
	\end{tabular}
		\caption{Example \ref{ex:Mach}. The initial data of the Mach reflection problem.}
		\label{tab:Mach}
	\end{table}
	\begin{table}
		\centering
		\begin{tabular}{l|ccccc}
		\hline & {\tt MM} ($200 \times 100$) & {\tt MM1} ($200 \times 100$)  & {\tt UM} ($200 \times 100$)  & {\tt UM} ($600 \times 300$)  \\
		\hline P1 & $3 \mathrm{m}53\mathrm{s}$
		& $3 \mathrm{m} 19 \mathrm{s}$ & $ 37 \mathrm{s}$
		& $ 10 \mathrm{m} 57 \mathrm{s}$
		\\
		D1 & $4\mathrm{m}$   & $3 \mathrm{m} 37 \mathrm{s}$ & $ 47\mathrm{s}$
		& $ 13 \mathrm{m} 9 \mathrm{s}$
		\\  TD1
		&$2 \mathrm{m} 21\mathrm{s}$
		& $1\mathrm{m} 38 \mathrm{s}$ &$27 \mathrm{s}$ &$11\mathrm{m} 57 \mathrm{s}$ \\
		TD2 &$3 \mathrm{m} 42\mathrm{s}$   & $2\mathrm{m} 13 \mathrm{s}$ &$ 21\mathrm{s}$ &$8\mathrm{m} 8 \mathrm{s}$  \\
		\hline
	\end{tabular}
		\caption{CPU times of Examples \ref{ex:Mach}-\ref{ex:backstep} (4 cores).}
		\label{time:MachBack}
	\end{table}
	\begin{figure}[!ht]
	\begin{subfigure}[b]{0.32\textwidth}
	\centering
	\includegraphics[width=1\linewidth]{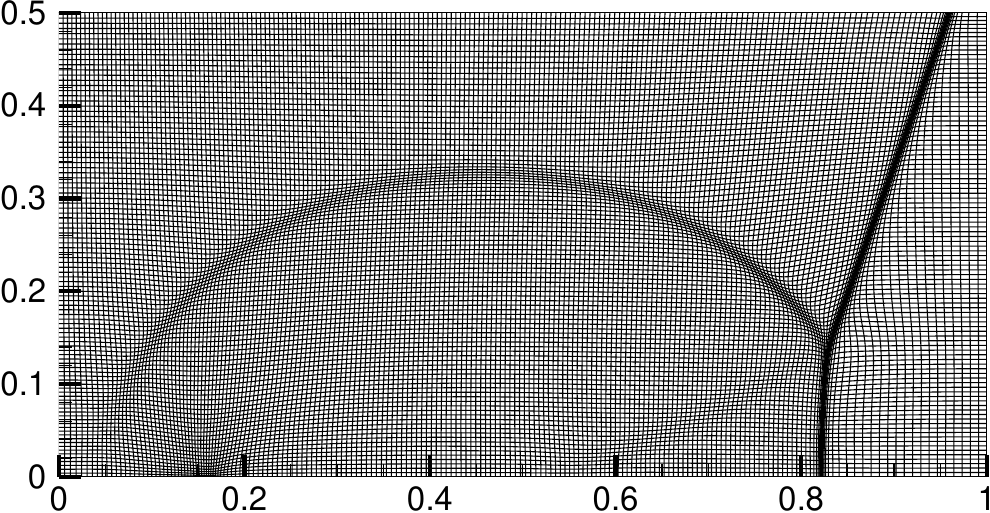}
	\caption{Mesh ({\tt {MM}}, $200 \times 100$)}
\end{subfigure}
\begin{subfigure}[b]{0.32\textwidth}
	\centering
	\includegraphics[width=1.01\linewidth]{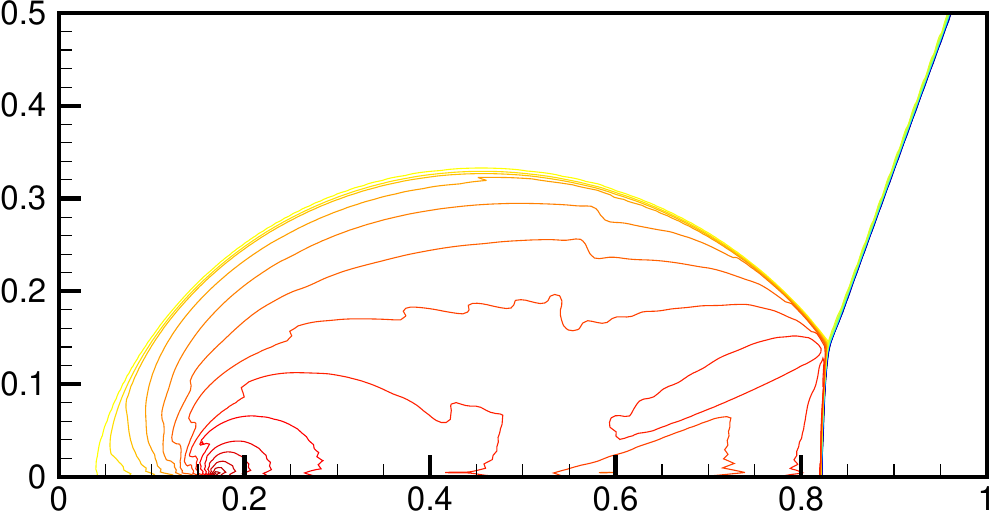}
	\caption{$\rho$ ({\tt{MM}}, $200 \times 100$)}
\end{subfigure}
\begin{subfigure}[b]{0.32\textwidth}
	\centering
	\includegraphics[width=1.0\linewidth]{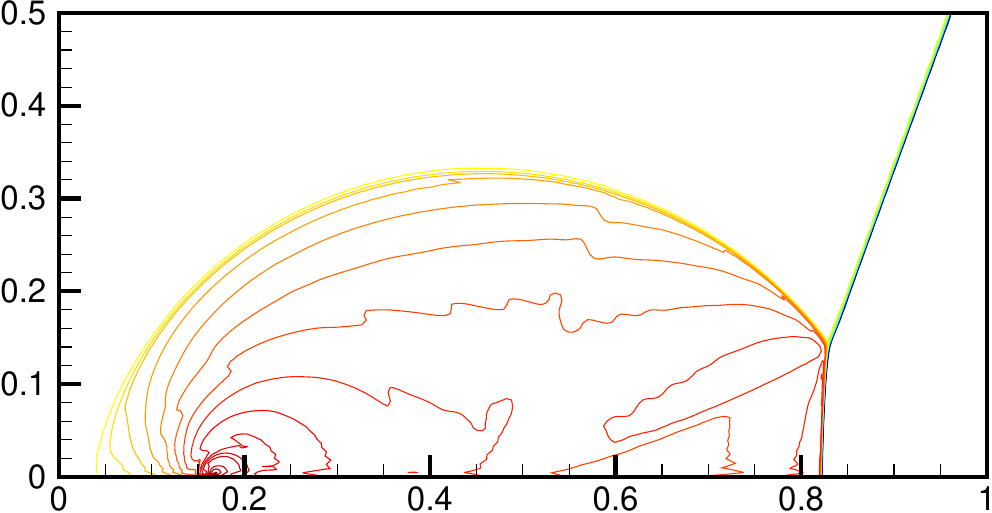}
	\caption{$\rho$ ({\tt{MM1}}, $200 \times 100$)}
\end{subfigure}

\begin{subfigure}[b]{0.32\textwidth}
	\includegraphics[width=1.0\linewidth]{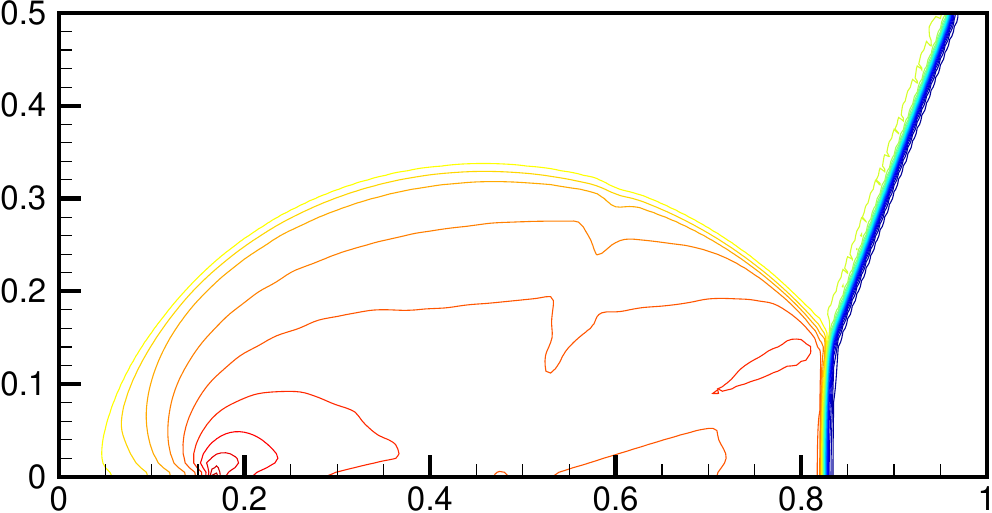}
	\caption{$\rho$ ({\tt UM}, $200 \times 100$)}
\end{subfigure}
\begin{subfigure}[b]{0.32\textwidth}
	\includegraphics[width=1.0\linewidth]{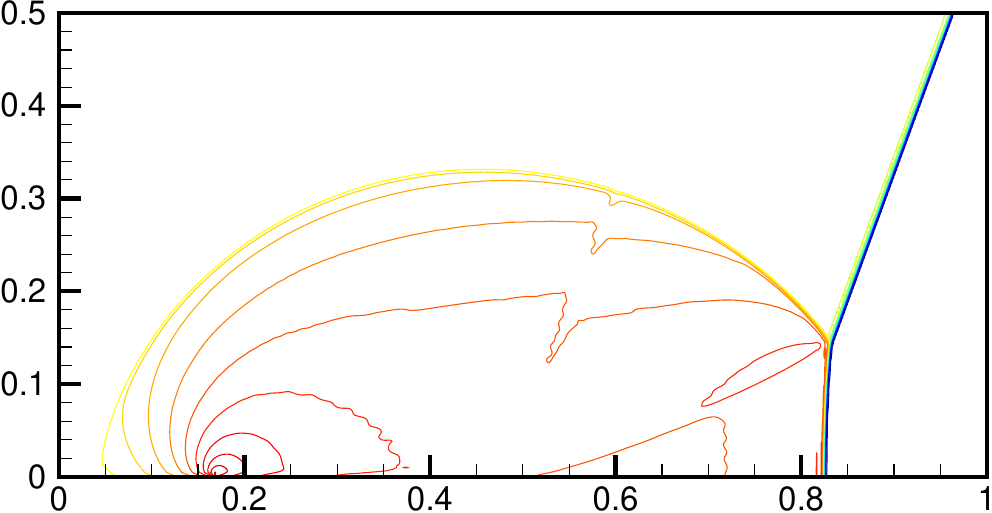}
	\caption{$\rho$ ({\tt UM}, $600 \times 300$)}
\end{subfigure}
\begin{subfigure}[b]{0.32\textwidth}
	\centering
	\includegraphics[width=1.0\linewidth]{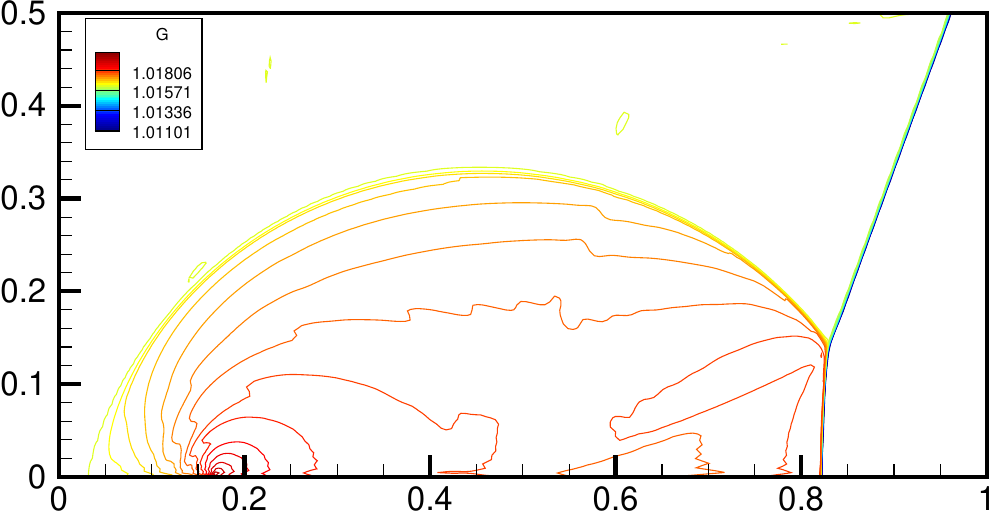}
	\caption{G ({\tt{MM}}, $200 \times 100$)}
	\label{fig: P1_G}
\end{subfigure}
	\caption{Example \ref{ex:Mach}.
		Adaptive mesh of {\tt {MM}} with $200 \times 100$ cells, density contours ($40$ equally spaced
		contour lines) of {\tt {MM}}, {\tt MM1} and {\tt UM}, and the
		fundamental derivative contours ($40$ equally spaced
		contour lines) of {\tt MM} at $t = 1$.}
	\label{fig:SMR_P1}
\end{figure}

	\begin{figure}[!ht]
		\begin{subfigure}[b]{0.32\textwidth}
		\centering
		\includegraphics[width=1.0\linewidth]{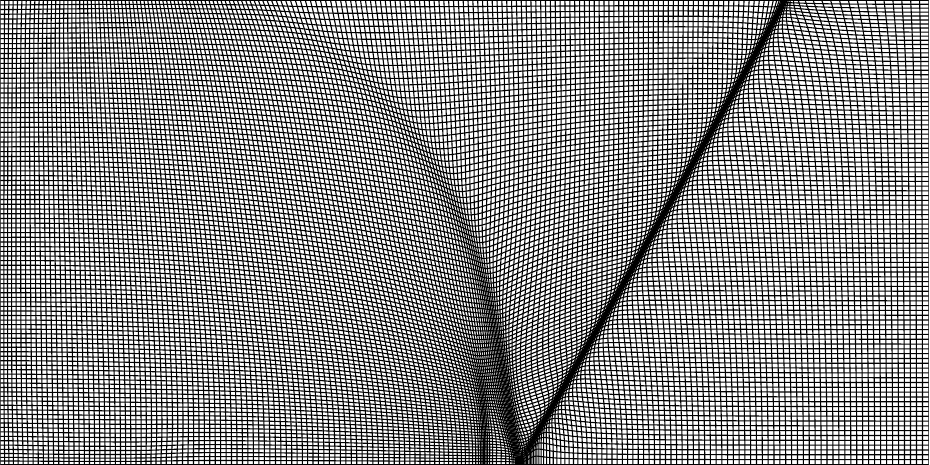}
		\caption{Mesh ({\tt {MM}}, $200 \times 100$)}
	\end{subfigure}
	\begin{subfigure}[b]{0.32\textwidth}
		\centering
		\includegraphics[width=1.0\linewidth]{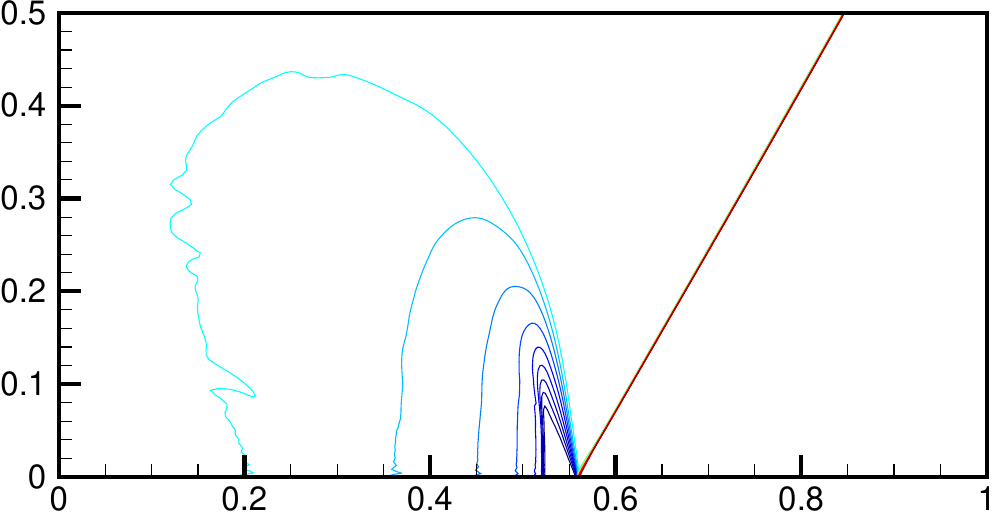}
		\caption{$\rho$ ({\tt {MM}}, $200 \times 100$)}
	\end{subfigure}
	\begin{subfigure}[b]{0.32\textwidth}
		\centering
		\includegraphics[width=1.0\linewidth]{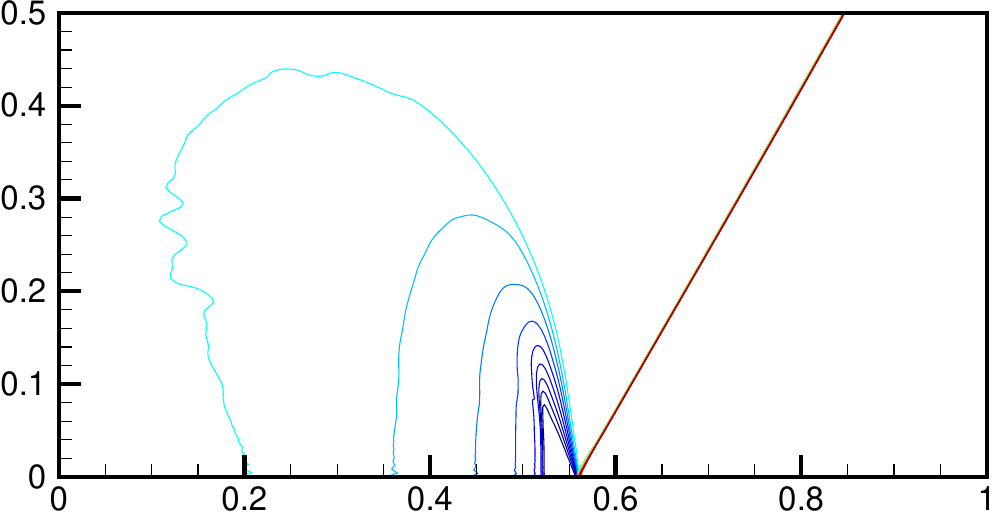}
		\caption{$\rho$ ({\tt {MM1}}, $200 \times 100$)}
	\end{subfigure}
	
	\begin{subfigure}[b]{0.32\textwidth}
		\centering
		\includegraphics[width=1.0\linewidth]{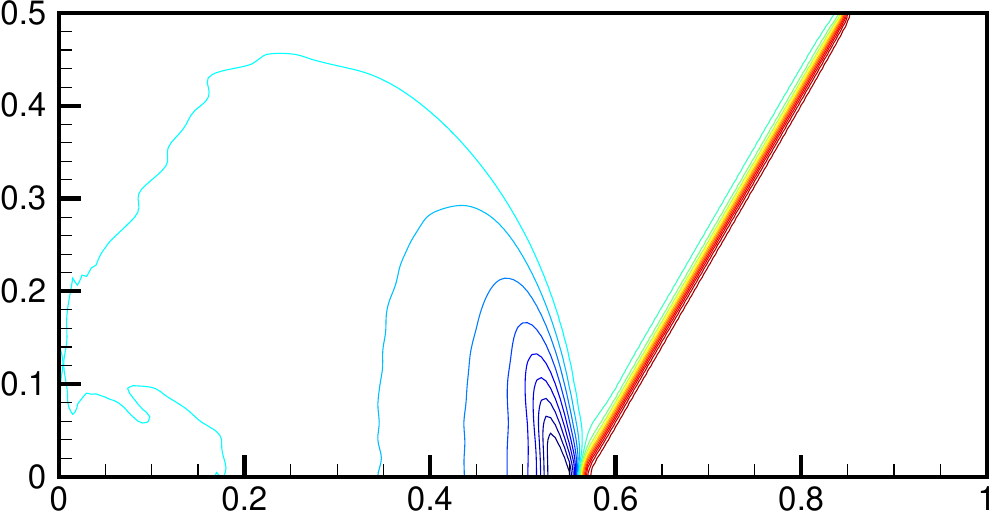}
		\caption{$\rho$ ({\tt UM}, $200 \times 100$)}
	\end{subfigure}
	\begin{subfigure}[b]{0.32\textwidth}
		\centering
		\includegraphics[width=1.0\linewidth]{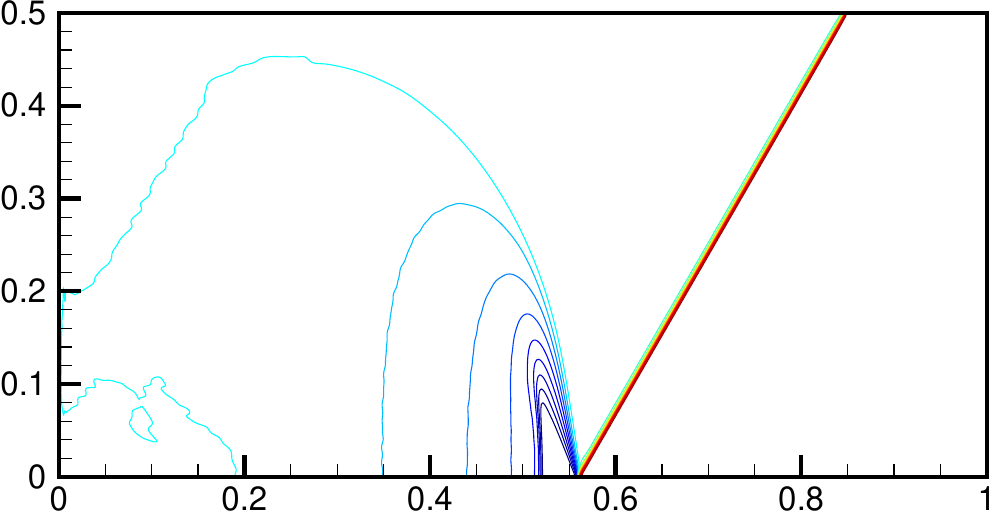}
		\caption{$\rho$ ({\tt {UM}}, $600 \times 300$)}
	\end{subfigure}
	\begin{subfigure}[b]{0.32\textwidth}
	\centering
	\includegraphics[width=1.0\linewidth]{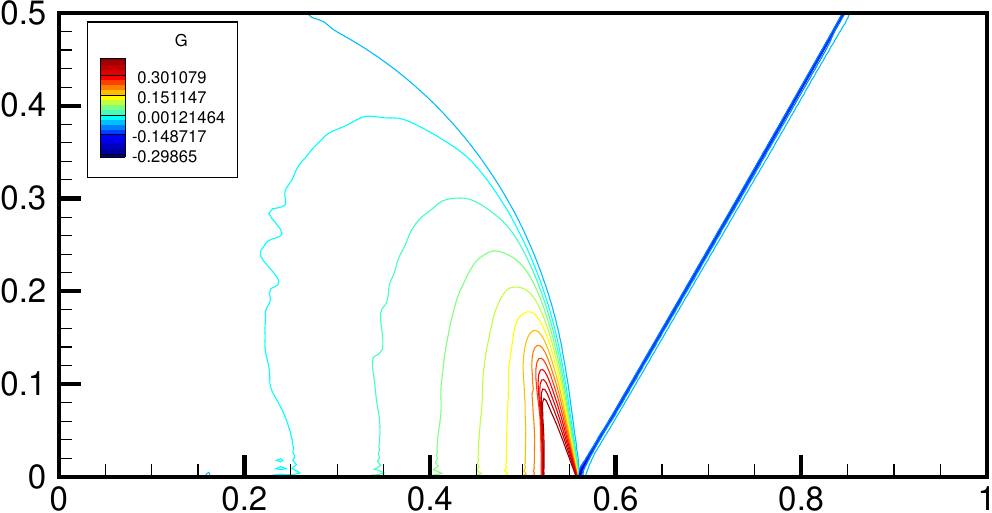}
	\caption{$\rho$ ({\tt {MM}}, $200 \times 100$)}
	\label{ex:D1_G}
     \end{subfigure}
		\caption{Example \ref{ex:Mach}.
	Adaptive mesh of {\tt {MM}} with $200 \times 100$ cells, density contours ($20$ equally spaced
	contour lines) of {\tt MM}, {\tt MM1} and {\tt UM}, and the
	fundamental derivative contours ($20$ equally spaced
	contour lines) of {\tt MM} at $t = 0.06$.
		}
		\label{fig:SMR_D1}
	\end{figure}

\end{example}

\begin{example}[Backward Step Problem]\label{ex:backstep}\rm	
	This test describes a moving shock wave diffracts around a $90^\circ$ corner given in \cite{brown1998nonclassical}. The initial data are shown in Table
	\ref{tab:back}. The monitor function of {\tt MM} is chosen as
	 \eqref{eq:van der Waalsmonitor} for TD1 and TD2 with $\kappa = 2, \delta_1 = \rho, \delta_2 = G $, $\alpha_1 = 1200, \alpha_2 = 1200, \alpha_3 = 0,$ and those of {\tt MM1} are identical to  {\tt MM}
	 except for $\alpha_2 = 0.$
	 		
 TD1 is  about the process of a shock wave with Mach number $M_s = 1.23$  moving to the right. The adaptive mesh of {\tt MM} with $200 \times 100$ cells and  the densities at $t = 0.55$ are shown in Figure \ref{fig:BackP1}.   Figure \ref{ex: BackD2_G}  presents the fundamental derivative obtained by {\tt MM} with $200\times 100$ cells. Unlike the case of ideal gas, an expansion shock wave evolves at the corner where the fundamental derivative G changes its sign. Figure \ref{fig:2DTD1} gives the densities  along the line connecting $(0.6, 0.55)$ and $(2,0.55)$.  One can seen that the solution given by {\tt MM} with $200\times 100$ cells is superior to the solutions given by {\tt UM} with $600 \times 300$ cells and {\tt MM1} with $200\times 100$ cells near $x_1 = 1.5.$
As shown in Table \ref{time:MachBack}, the CPU time required by the {\tt MM} scheme is only 19.7\% of that consumed by the {\tt UM} scheme with $600 \times 300$ cells, while {\tt MM}  provides a higher-resolution solution.

TD2 describes a shock wave with
$M_s = 1.05$ that spreads into a fan near the corner during flow evolution.
 Figure  \ref{fig:BackD1}   gives the adaptive mesh of {\tt MM} with $200 \times 100$ cells and the densities at $t = 1.5$. One can see that  the mesh points
adaptively concentrate near the large gradient area of the density and fundamental derivative.
The $G>0$ region only appears near the step  at $x=0.95$ shown in \ref{ex:BackD1_G}.  Figure \ref{fig:2DTD2} depicts  the densities  along the line connecting $(0.6, 0.55)$ and $(2,0.55)$. It is evident that  {\tt MM} captures the localized
 structures better than {\tt MM1} near $x_1=0.95$ with the same number of cells.
 Table \ref{time:MachBack} tells us that {\tt MM} is  efficient since it costs $45.5\%$ CPU time of {\tt UM} with a finer mesh, but gets better results.

	\begin{table}
		\centering
		\begin{tabular}{l|cccccccc}
			\hline  & $\rho_{l}$ & $v_{1,l}$ & $v_{2,l}$ & $p_{l}$ & $\rho_{r}$ & $v_{1,r}$ & $v_{2,r}$ & $p_{r}$  \\
		\hline TD1   &$0.85$& $0.64$ & $0.0$& $1.01$ & $0.28$ & $0.0$ & $0.0$ & $0.58$ \\
		TD2   & $0.62$  & $-0.14$ & $0.0$& $0.98$ & $0.88$ & $0.0$  & $0.0$& $1.09$ \\
		\hline
		\end{tabular}
		\caption{Example \ref{ex:backstep}. The initial data of flow over a backward step.}
		\label{tab:back}
	\end{table}

	\begin{figure}[!ht]
	\centering
	\begin{subfigure}[b]{0.32\textwidth}
		\centering
		\includegraphics[width=1.0\linewidth]{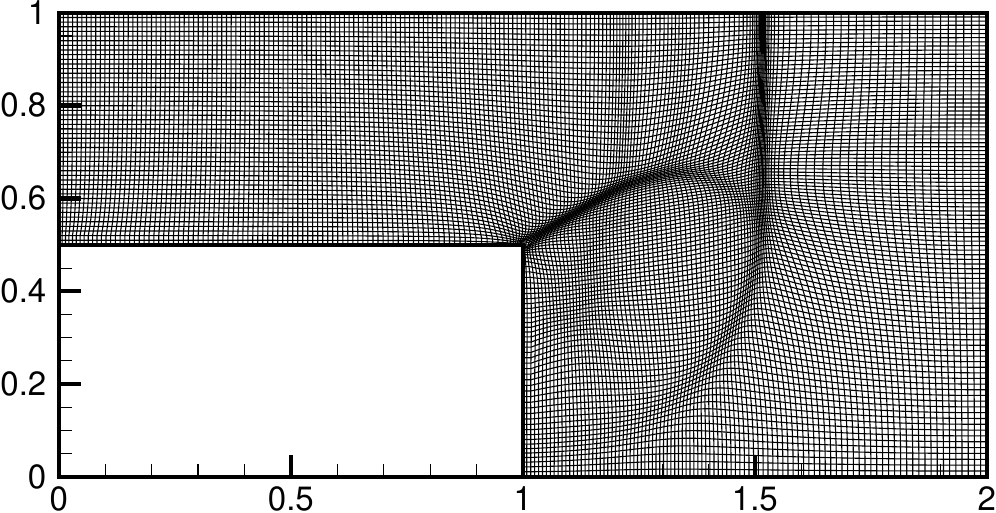}
		\caption{Mesh ({\tt {MM}}, $200 \times 100$)}
	\end{subfigure}
	\begin{subfigure}[b]{0.32\textwidth}
		\centering
		\includegraphics[width=1.0\linewidth]{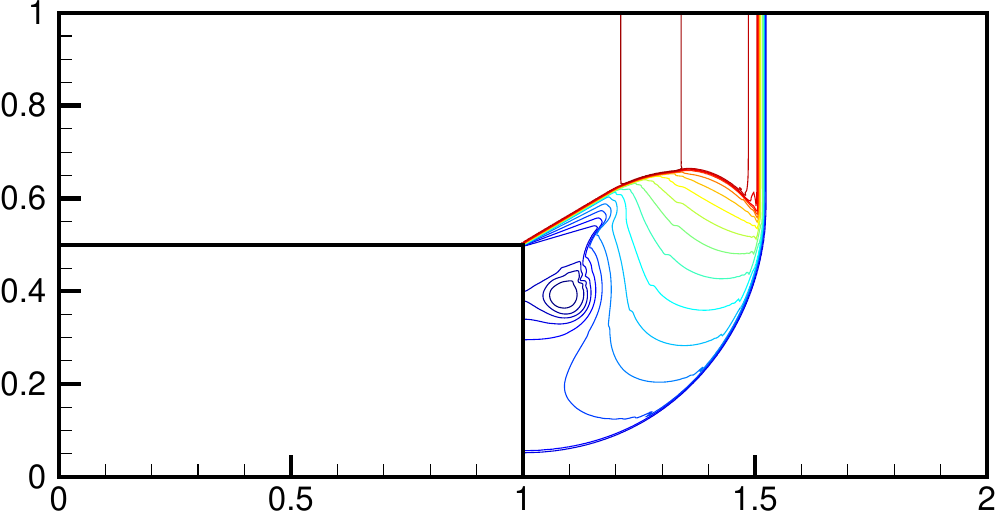}
		\caption{$\rho$ ({\tt {MM}}, $200 \times 100$)}
	\end{subfigure}
	\begin{subfigure}[b]{0.32\textwidth}
		\centering
		\includegraphics[width=1.0\linewidth]{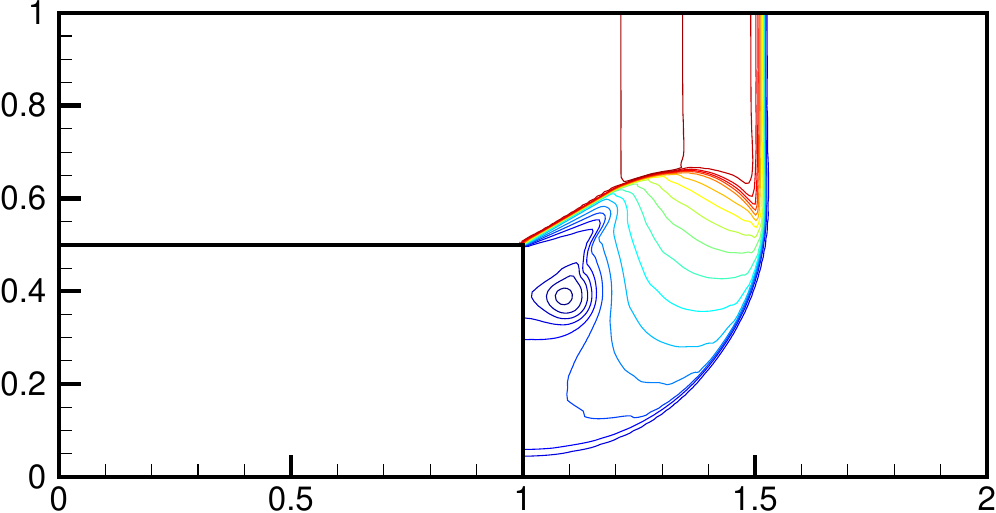}
		\caption{$\rho$ ({\tt {MM1}}, $200 \times 100$)}
	\end{subfigure}
	%
	
	\begin{subfigure}[b]{0.32\textwidth}
		\centering
		\includegraphics[width=1.0\linewidth]{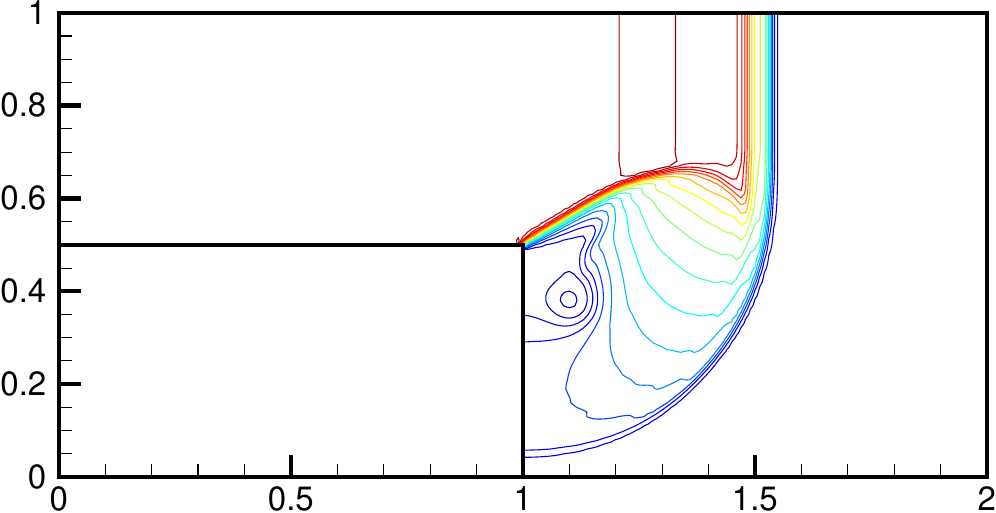}
		\caption{$\rho$ ({\tt UM}, $200 \times 100$)}
	\end{subfigure}
	\begin{subfigure}[b]{0.32\textwidth}
		\centering
		\includegraphics[width=1.0\linewidth]{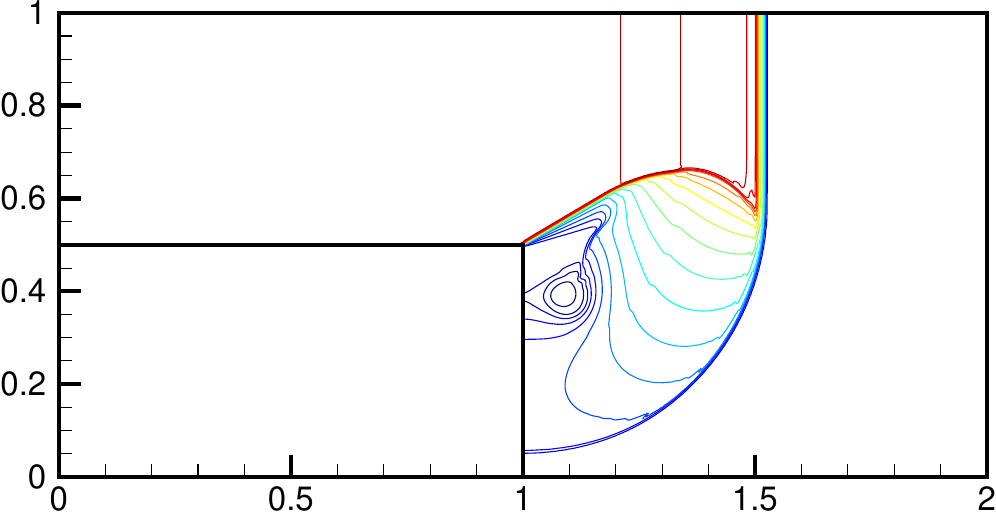}
		\caption{$\rho$ ({\tt UM}, $600 \times 300$)}
	\end{subfigure}
	
	\caption{
		Example \ref{ex:backstep}.
		Adaptive mesh of {\tt MM} with $200 \times 100$ cells, density contours ($20$ equally spaced
		contour lines) of {\tt MM}, {\tt MM1} and {\tt UM} at $t = 0.55$.}
	\label{fig:BackP1}
\end{figure}

	\begin{figure}[!ht]
		\centering
		\begin{subfigure}[b]{0.32\textwidth}
			\centering
			\includegraphics[width=1.0\linewidth]{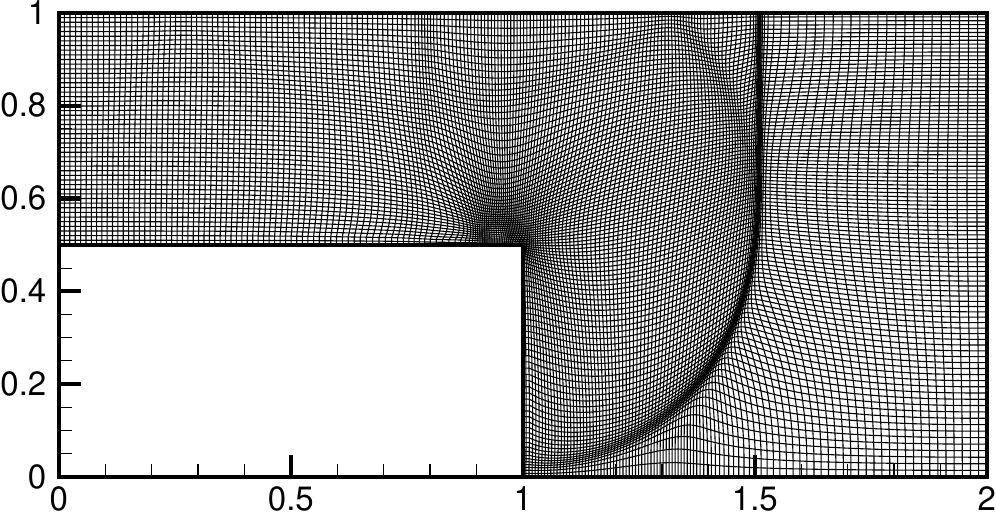}
			\caption{Mesh ({\tt {MM}}, $200 \times 100$)}
		\end{subfigure}
		\begin{subfigure}[b]{0.32\textwidth}
			\centering
			\includegraphics[width=1.0\linewidth]{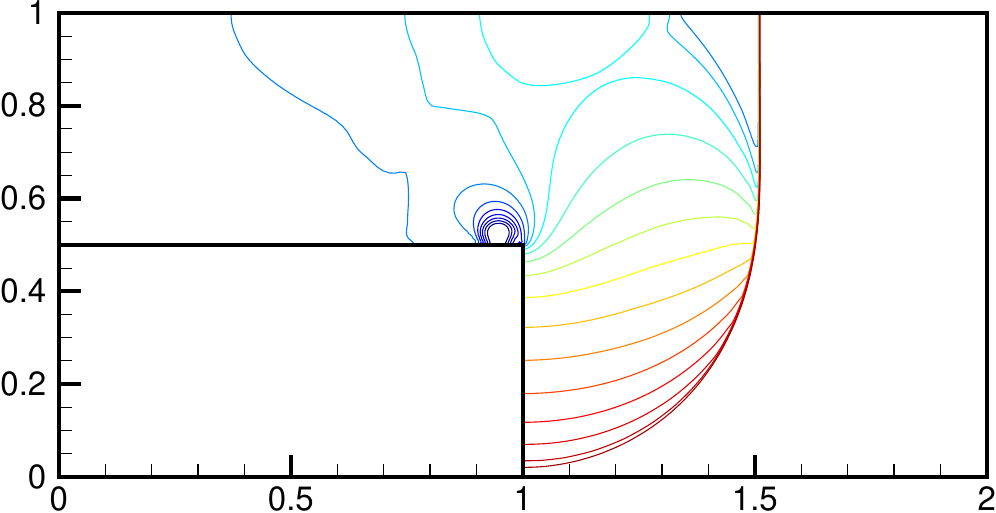}
			\caption{$\rho$ ({\tt {MM}}, $200 \times 100$)}
		\end{subfigure}
		\begin{subfigure}[b]{0.32\textwidth}
			\centering
			\includegraphics[width=1.0\linewidth]{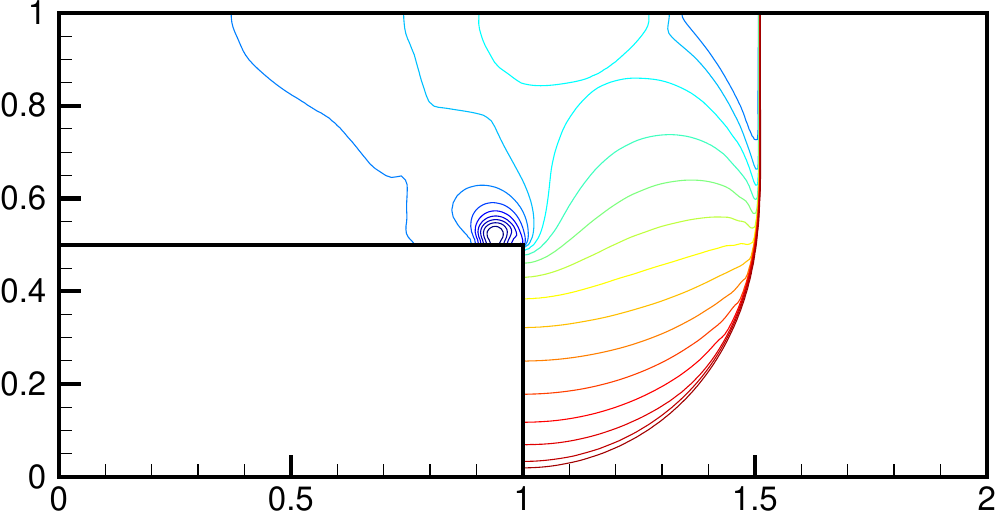}
			\caption{$\rho$ ({\tt MM1}, $200 \times 100$)}
		\end{subfigure}
		
		\begin{subfigure}[b]{0.32\textwidth}
			\centering
			\includegraphics[width=1.0\linewidth]{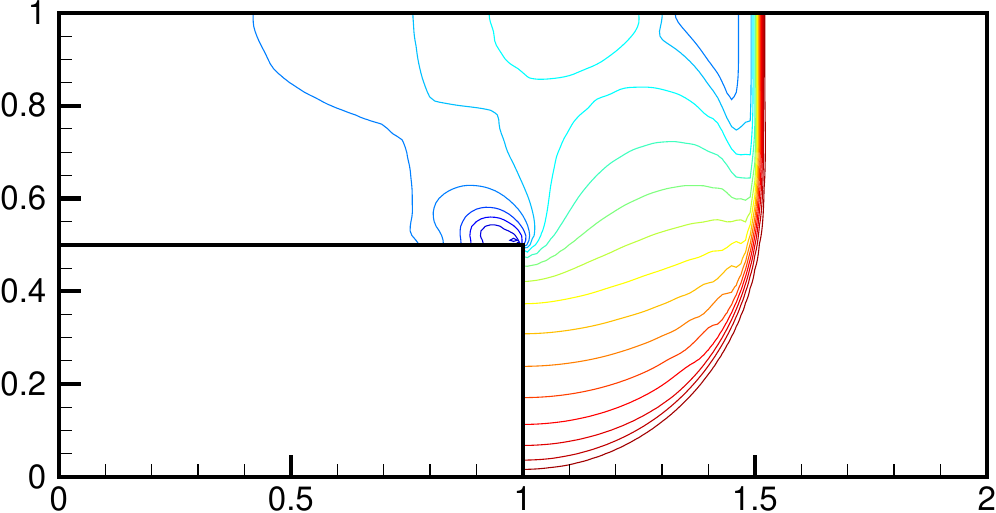}
			\caption{$\rho$ ({\tt UM}, $200 \times 100$)}
		\end{subfigure}
		\begin{subfigure}[b]{0.32\textwidth}
			\centering
			\includegraphics[width=1.0\linewidth]{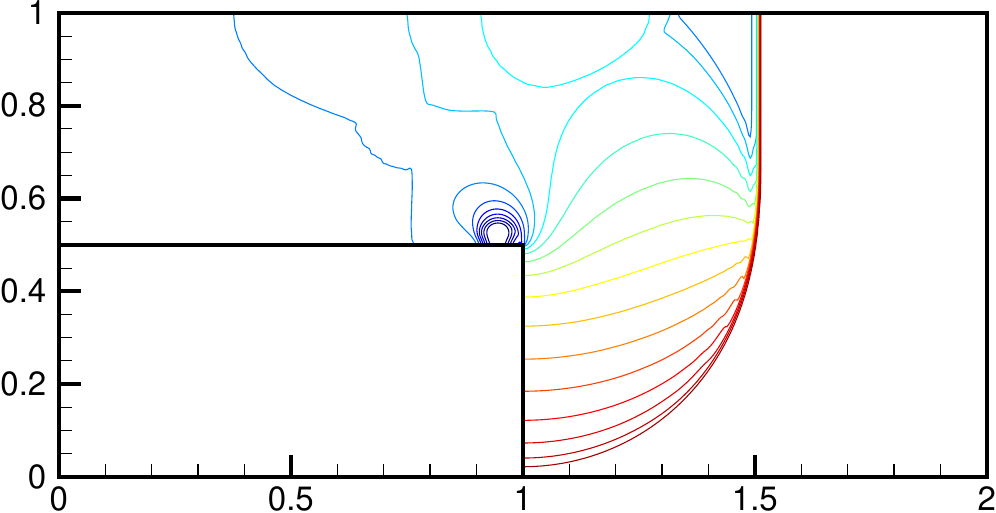}
			\caption{$\rho$ ({\tt UM}, $600 \times 300$)}
		\end{subfigure}
			\caption{Example \ref{ex:backstep}.
			Adaptive mesh of {\tt MM} with $200 \times 100$ cells, density contours ($20$ equally spaced
			contour lines) of {\tt MM}, {\tt MM1} and {\tt UM} at $t = 1.5$.}
		\label{fig:BackD1}
	\end{figure}

	\begin{figure}[!ht]
	\centering
	\begin{subfigure}[b]{0.35\textwidth}
		\centering
		\includegraphics[width=1.0\linewidth]{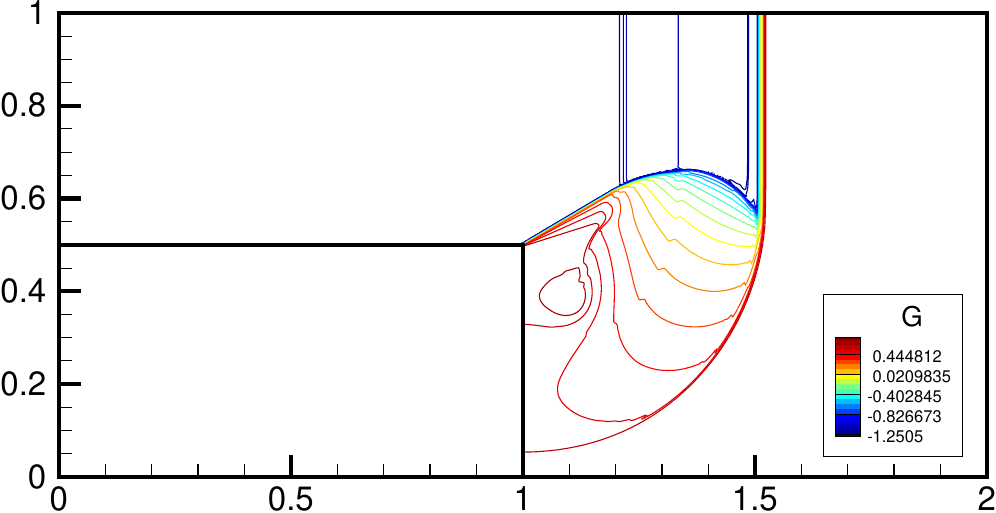}
		\caption{TD1($t = 0.55$)}
		\label{ex: BackD2_G}
	\end{subfigure}
	\begin{subfigure}[b]{0.35\textwidth}
		\centering
		\includegraphics[width=1.0\linewidth]{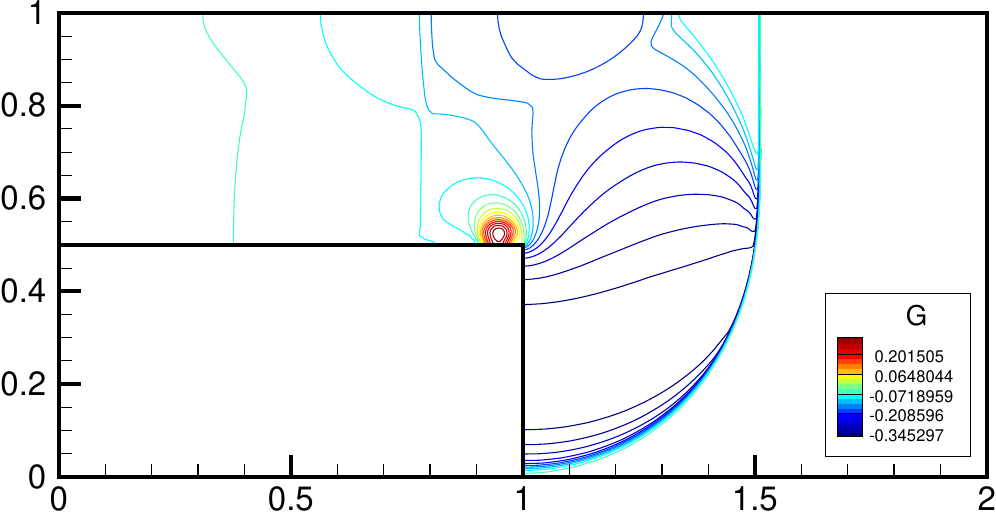}
		\caption{TD2($t = 1.5$)}
		\label{ex:BackD1_G}
	\end{subfigure}
	\caption{
Example \ref{ex:backstep}.
The
fundamental derivative contours ($20$ equally spaced
contour lines) obtained by {\tt MM} with $200 \times 100$ cells for TD1 (left) and TD2 (right).	
}
	\label{fig:Back_G}
\end{figure}
	\begin{figure}[!ht]
	\centering
	\begin{subfigure}[b]{0.32\textwidth}
		\centering
		\includegraphics[width=1.0\linewidth]{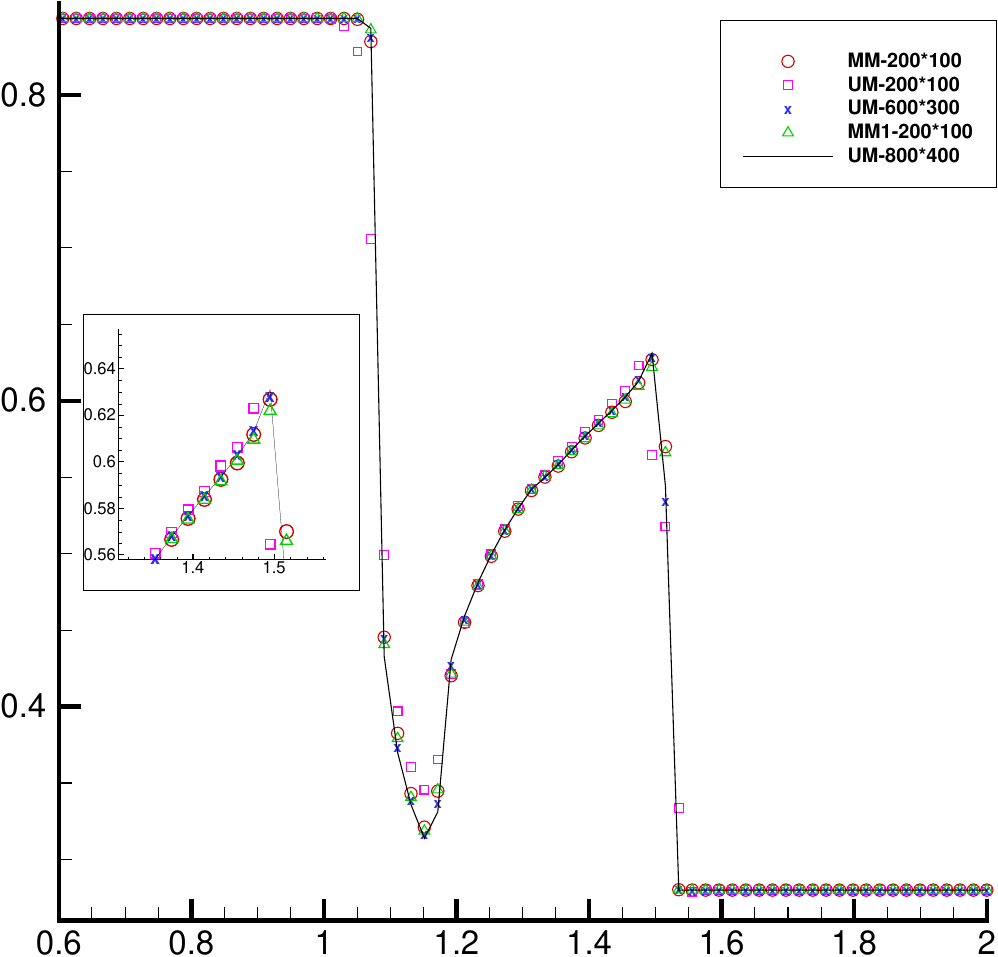}
		\caption{TD1}
		\label{fig:2DTD1}
	\end{subfigure}
	\begin{subfigure}[b]{0.32\textwidth}
		\centering
		\includegraphics[width= 1.1\linewidth]{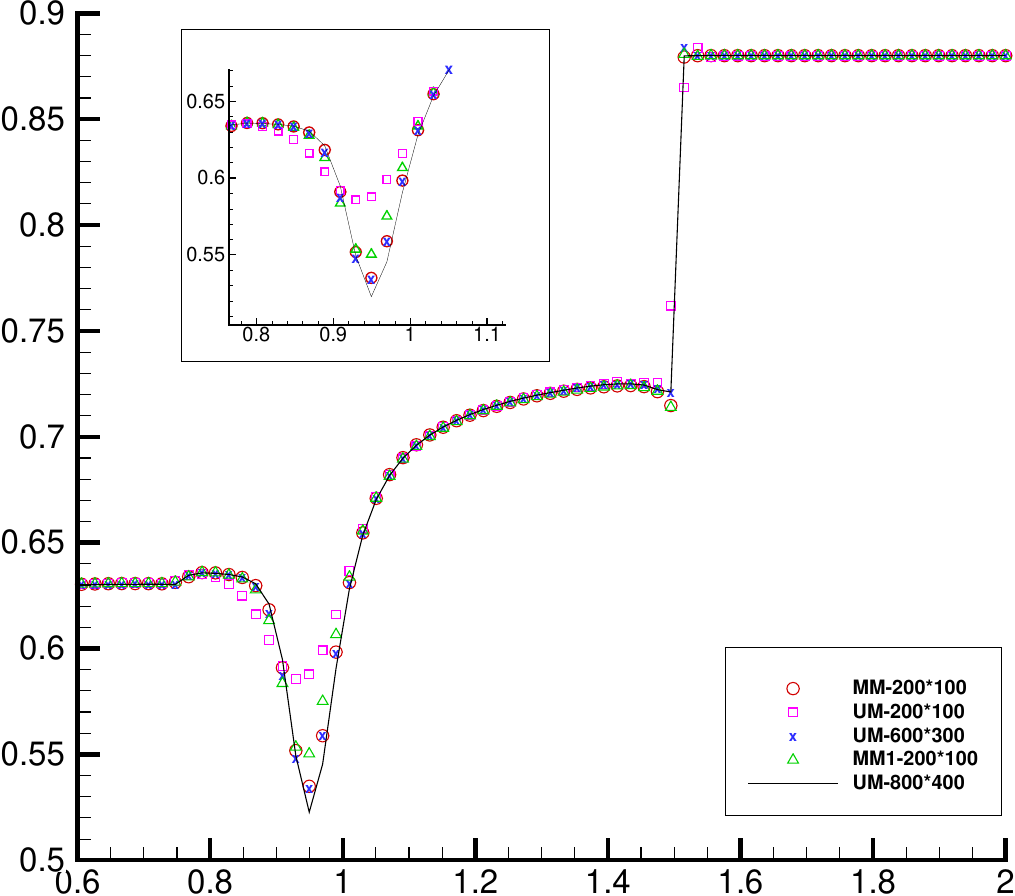}
		\caption{ TD2}
		\label{fig:2DTD2}
	\end{subfigure}
	\caption{Example {\ref{ex:backstep}}. Density  $\rho$ along the line connecting $ (0.6, 0.55)$ and $(2, 0.55)$ for TD1 (left) and TD2 (right), respectively.}
	\label{fig:2DBackComp}
\end{figure}

\end{example}


\begin{example}[3D Smooth Sin Wave]\label{ex:3Dsmooth} \rm
		Similar to Example \ref{ex:1Dsmooth}, this example is   to check the accuracy of  3D {\tt MM} via the 3D sin wave moving periodically with a  constant velocity in the domain  $\Omega_p=[0, 1]^3$.  
	 The exact solutions are given by
	$$(\rho, v_1, v_2,v_3, p) = (0.5 + 0.2\sin(2\pi(x_1+x_2+x_3-0.9 t) ), 0.3, 0.3, 0.3, 1).$$
	The monitor function {\color{black} and the number of smooth times are the same as Example \ref{ex:1Dsmooth}.
		The boundary points move according to the  periodic boundary conditions.} The linear weights in multi-resolution WENO reconstruction are taken as
$\lambda_0 =\lambda_1  =\lambda_2  = 1/3$ in this test.
	Figure \ref{fig:3DSin} presents the $\ell^1$- and $\ell^\infty$-errors and the corresponding orders of convergence in  $\rho$ at $t = 0.1$ with {\tt MM}, which shows that  {\tt MM} can achieve the fifth-order accuracy as expected.
%
		\begin{figure}[!ht]
		\centering
		\includegraphics[width=0.33\linewidth]{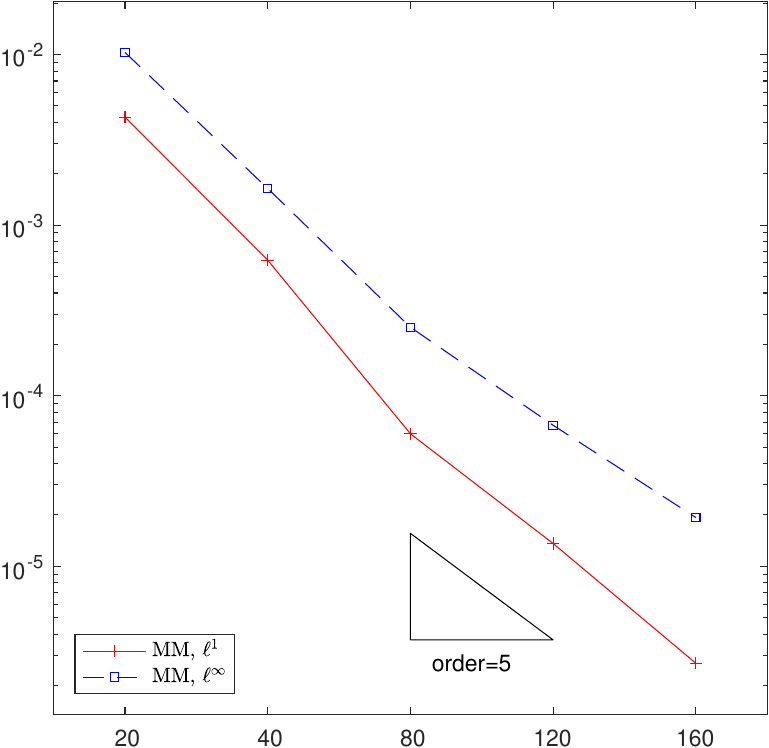}
		\caption{Example \ref{ex:3Dsmooth}.
					$\ell^1$- and $\ell^\infty$-errors and orders of convergence in $\rho$ with respect to $N_1$ at $t = 0.1$. }
		\label{fig:3DSin}
	\end{figure}
	
\end{example}

\begin{example}[3D Spherical Symmetric Shock Tube]\label{ex:3DShock}\rm
	This is an extension of 1D Riemann problem. The initial data are
	$$
	\left(\rho, v_{1}, v_{2},  v_{3}, p\right)=\left\{\begin{array}{ll}
	(1.818,0,0,0,3.0), & \sqrt{x_1^2+x_2^2+x_3^2} < 0.5,\\
	(0.275,0,0, 0,0.575), &\text {otherwise.}
	\end{array}\right.
	$$
	The  domain $\Omega_p$ is $[0,1]^{3}$, and the monitor function of {\tt MM} is the same as Example \ref{ex:backstep}.
The numerical results of $\rho$ along the volume diagonal connecting $(0, 0, 0)$
and $(1, 1, 1)$ are shown in \ref{fig:3DShock_Rho} with the reference solution obtained by a second-order finite volume
scheme using $10000$ cells in 1D spherical coordinates. The adaptive mesh points gather near
the large gradient area of the density and the fundamental derivative as shown in Figure \ref{fig:3DShock_Mesh}. {\tt MM} with $100^3$ mesh cells is better than {\tt UM} with $100^3$ mesh cells near the shock and the rarefaction shock,  comparable with {\tt UM} with  $200^3$ mesh cells. It indicates the  flow characteristics can be precisely captured by {\tt MM}. From Table \ref{time:3DJet}, the CPU time of  {\tt MM} with $100^3$ cells is $12\%$ of {\tt UM} with $200^3$ cells, verifying the efficiency of {\tt MM}.

	\begin{figure}[!ht]
		\centering
		\begin{subfigure}[b]{0.45\textwidth}
			\centering
			\includegraphics[width=1.0\linewidth]{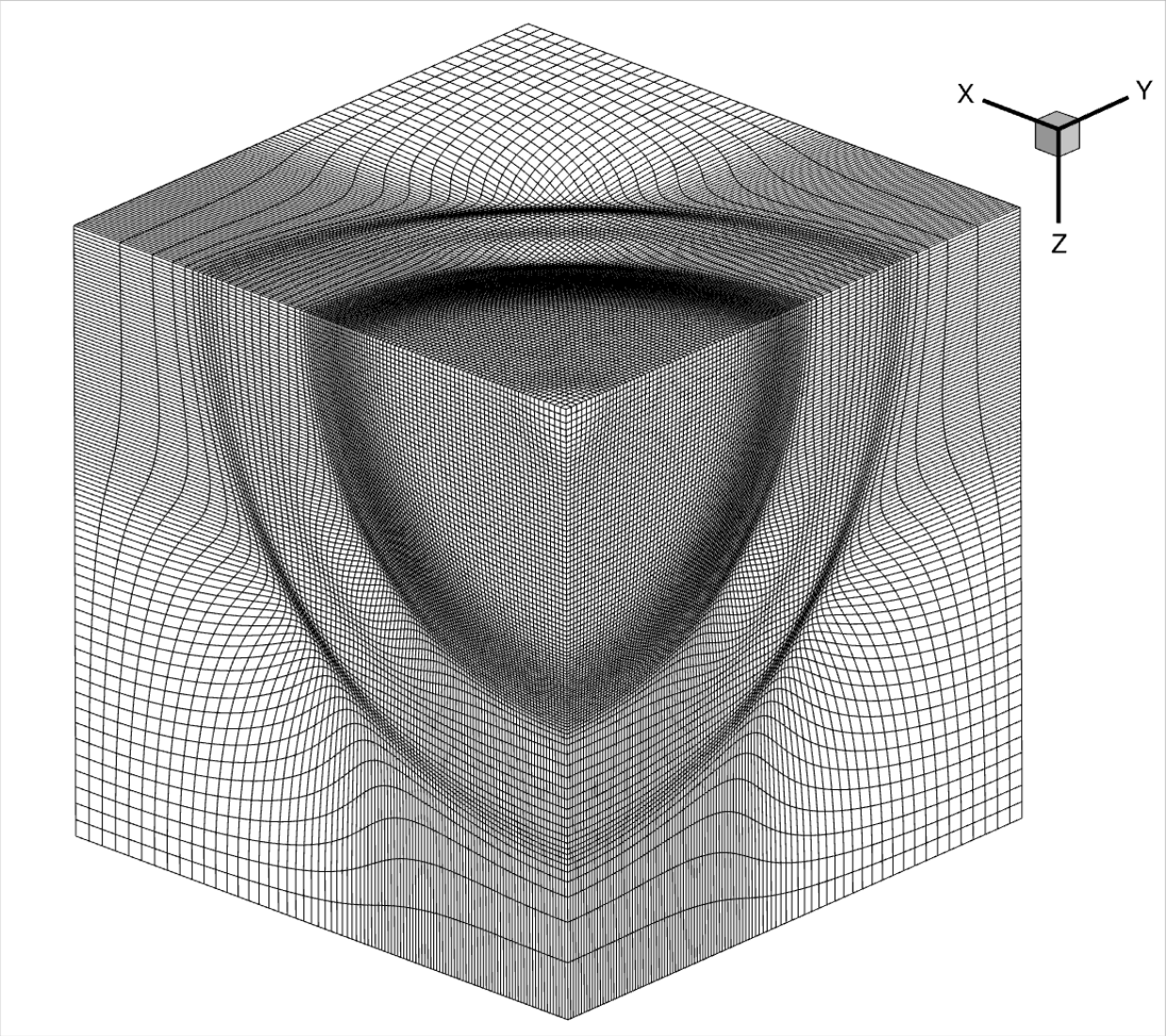}
			\caption{Adaptive mesh of  $100^3$ cells}
			\label{fig:3DShock_Mesh}
		\end{subfigure}
		\begin{subfigure}[b]{0.45\textwidth}
			\centering
			\includegraphics[width=2.6in,height=2.6in, trim=0 0 0 0, clip]{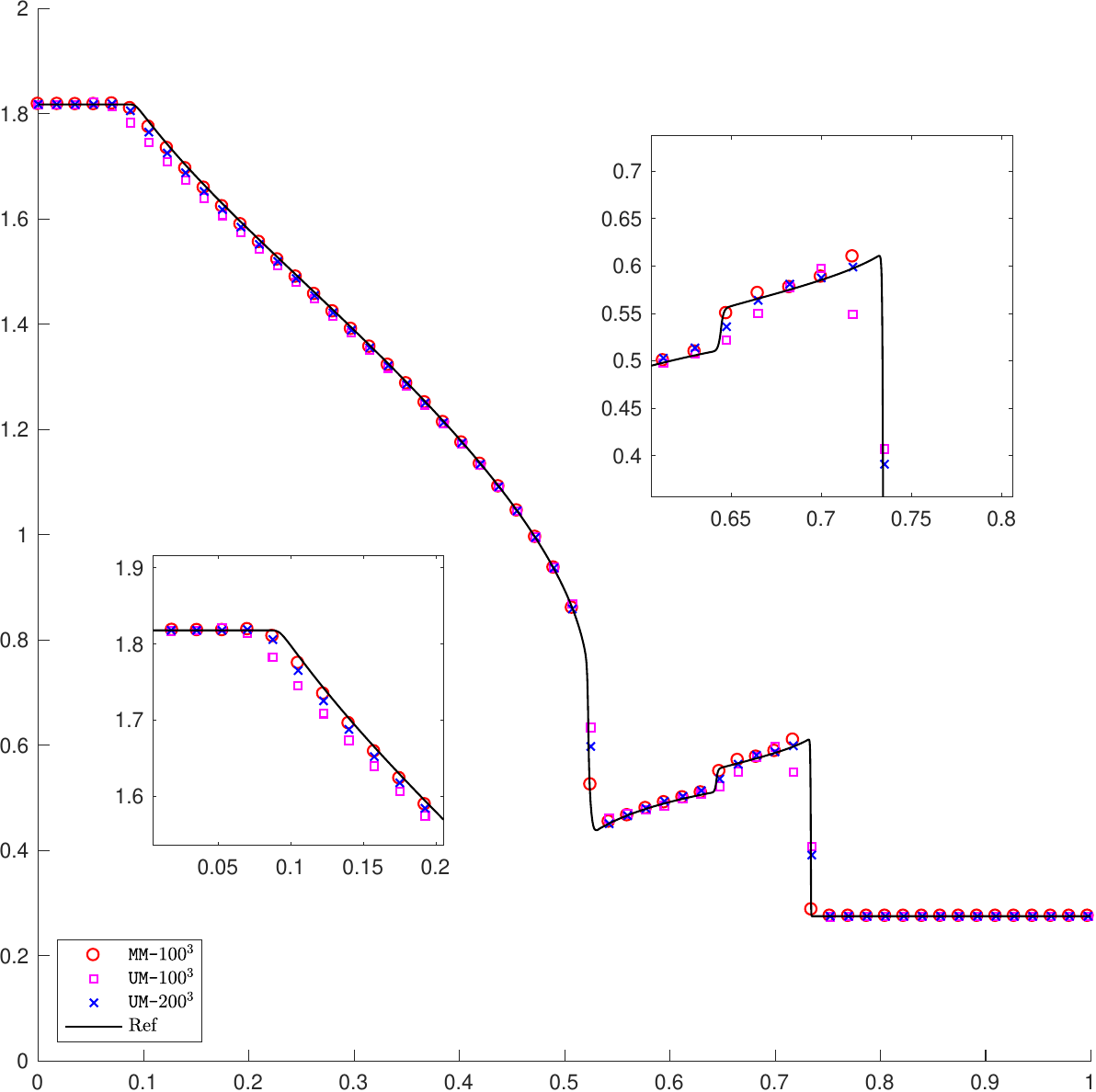}
			\caption{ $\rho$}
			\label{fig:3DShock_Rho}
		\end{subfigure}
		\caption{Example {\ref{ex:3DShock}}. Adaptive mesh of {\tt MM} and $\rho$ along the line connecting $(0, 0, 0)$ and $(1, 1, 1)$ at
$t = 0.2$.}
		\label{fig:3DShock}
	\end{figure}
	
\end{example}

\begin{example}[3D Jet Injection]\label{ex:3DJetInjection}\rm
	This test considers a simplified form of fuel injection
on the domain $\Omega_p=[0, 7.5] \times [-2.5, 2.5] \times [-2.5, 2.5]$. Here, a
higher-density  fluid is injected on the subdomain $[0, 0.6]\times[-0.3, 0.3]\times[-0.3, 0.3]$, the initial data are
$$
\left(\rho, v_{1}, v_{2}, v_{3}, p\right)=
\left(2, 1,0,0,10\right).
$$
The rest of  domain $\Omega_p$ is filled with a stationary lower-density fluid with
	$$
	\left(\rho, v_{1}, v_{2}, v_{3}, p\right)=
	(0.5,0,0,0,10).
	$$
	The monitor function is chosen as 	\eqref{eq:van der Waalsmonitor}, with $\kappa = 2,  \sigma_1 = \rho, \sigma_2 = G, \alpha_1  = \alpha_2= 1200, \alpha_3 = 0$.
	Figure \ref{fig:3DJetMesh} gives the close-up of the adaptive mesh, three offset 2D slices, the iso-surfaces
of $\rho = 0.88$,  and two surface meshes at $t = 4$.
Figure \ref{fig:3DJetRHO}
compares the schlieren
images on the slice $x_3 = 0$ given by $\Phi=\exp \left(-5|\nabla \rho| /|\nabla \rho|_{\max }\right)$, where the top  parts are the results obtained by {\tt MM} with $150 \times 100 \times 100$ cells, while the left and right bottom
half parts are those obtained by {\tt UM} with $150 \times 100 \times 100$ cells and $300 \times 200 \times 200$ cells, respectively. It can be seen that {\tt MM}
is efficient since the CPU time of {\tt MM} with $150 \times 100 \times 100$ cells is $34.9\%$ of  {\tt UM} with $300 \times 200 \times 200$ cells shown in Table \ref{time:3DJet}, and the solution obtained by  {\tt MM} is superior to that obtained by {\tt UM}  with $300 \times 200 \times 200$ cells.
	
	\begin{table}
		\centering
			\begin{tabular}{l|ccc}
			\hline & {\tt MM}  & {\tt UM}& {\tt UM}  \\
			\hline Example \ref{ex:3DShock}&  $ 9 \mathrm{m} 45\mathrm{s}$ ($100^3$)    & $ 3 \mathrm{m} 45\mathrm{s}$ ($100^3$)  & $1 \mathrm{h} 21\mathrm{m}$ ($200^3$)\\
			Example \ref{ex:3DJetInjection}&  $ 4 \mathrm{h} 32\mathrm{m}$ ($150 \times 100 \times 100$)    & $ 49 \mathrm{m} 2\mathrm{s}$ ($150 \times 100 \times 100$)  & $13 \mathrm{h} 0\mathrm{m}$ ($300 \times 200 \times 200$)\\
			\hline
		\end{tabular}
		\caption{ CPU times of Examples 	\ref{ex:3DShock}-\ref{ex:3DJetInjection} (32 cores).
		}
		\label{time:3DJet}
	\end{table}

	\begin{figure}[!ht]
		\centering
		
		\begin{subfigure}[b]{0.45\textwidth}
		\centering
		\includegraphics[width=1.0\linewidth,trim=2 2 2 2, clip ]{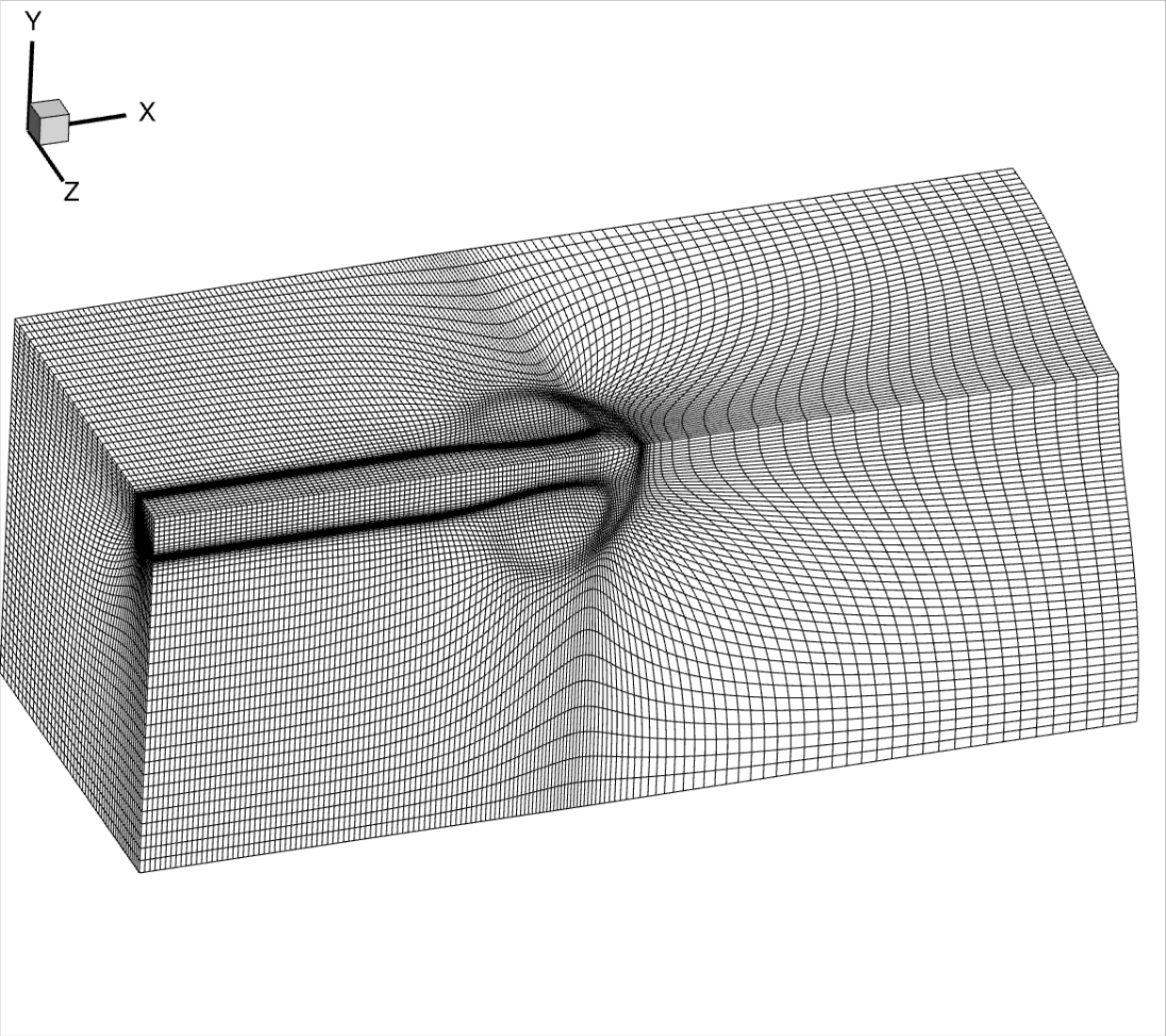}
	\end{subfigure}
	\begin{subfigure}[b]{0.45\textwidth}
		\centering
		\includegraphics[width=1.0\linewidth,trim=2 10 2 2, clip ]{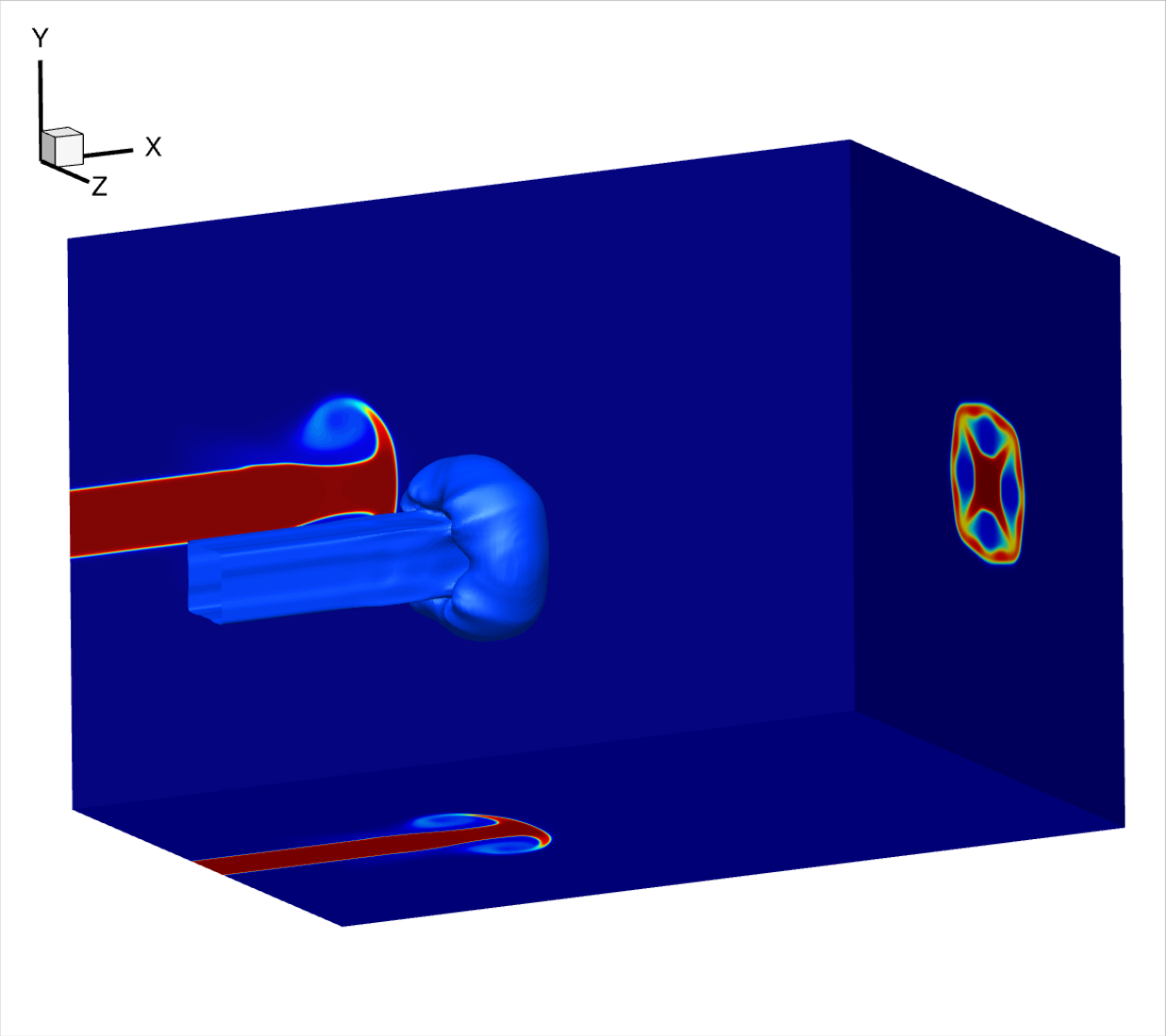}
	\end{subfigure}
	
	\begin{subfigure}[b]{0.45\textwidth}
		\centering
		\includegraphics[width=1.\linewidth,trim=4 4 4 4, clip]{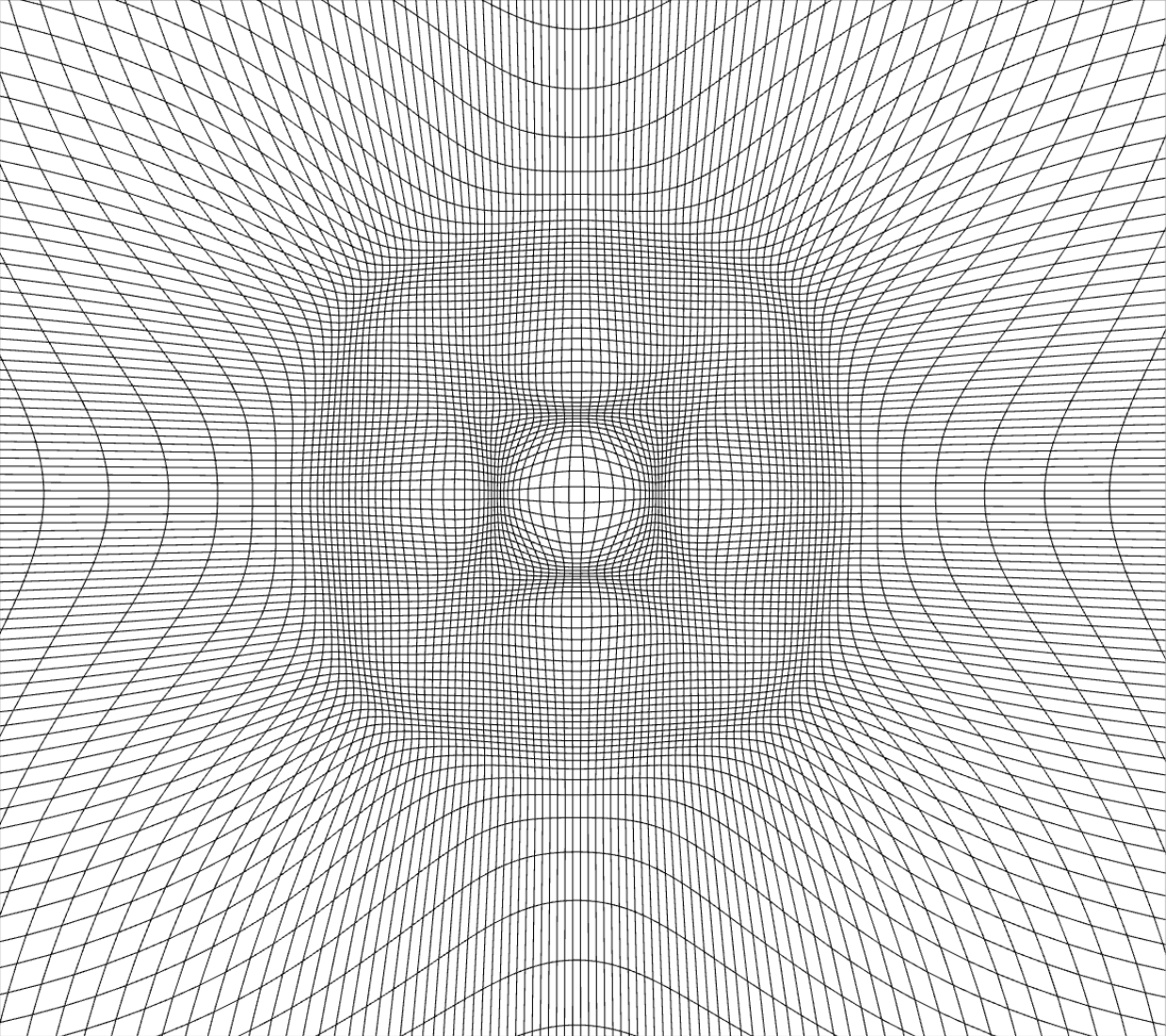}
	\end{subfigure}
	\begin{subfigure}[b]{0.45\textwidth}
		\centering
		\includegraphics[width=0.95\linewidth,trim=12 0 12 0, clip]{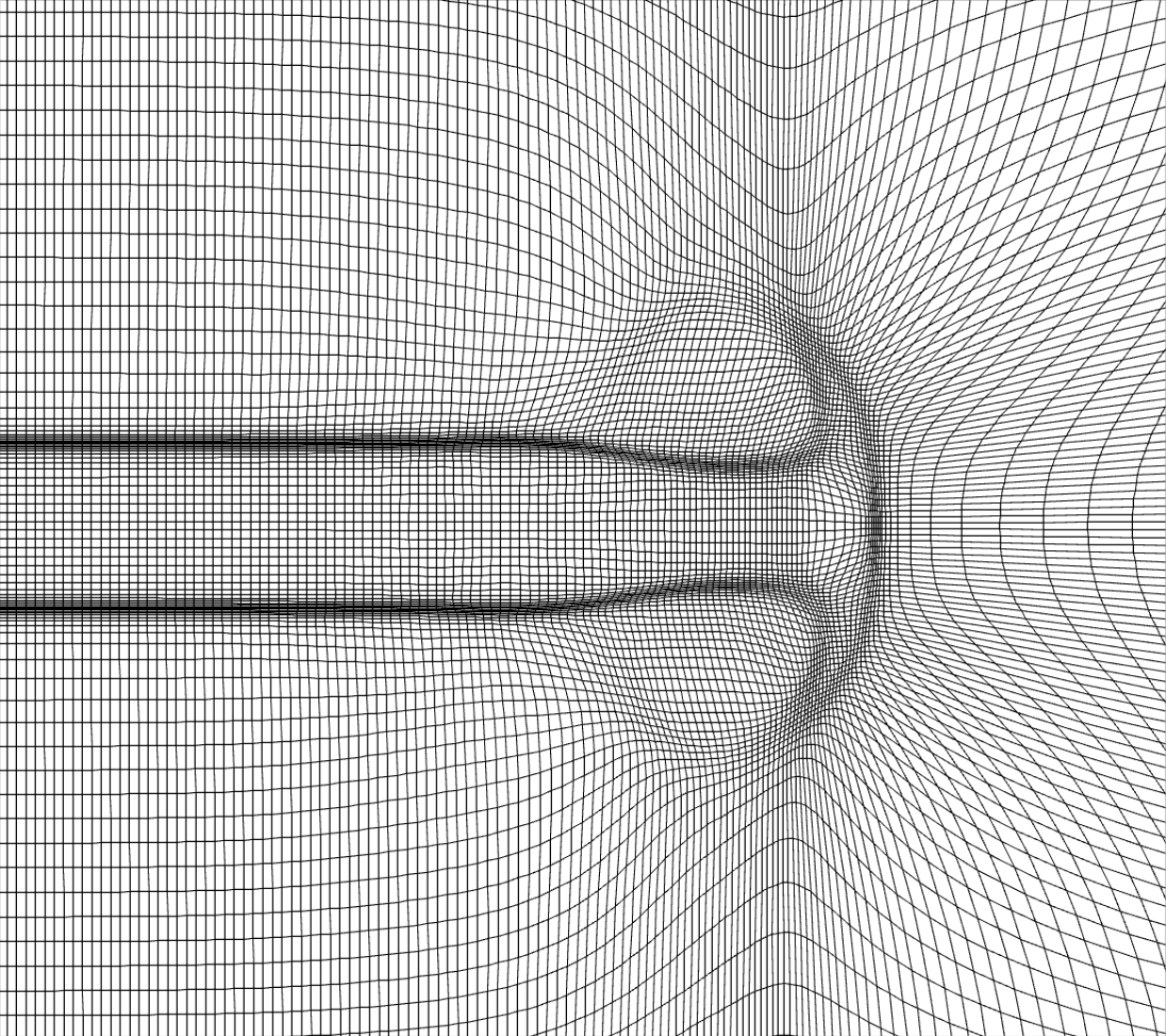}
	\end{subfigure}
		\caption{Example \ref{ex:3DJetInjection}.
			Adaptive meshes and $\rho$ at $t = 4$. Top left: close-up of the adaptive mesh, $i_1 \in [1, 140], i_2 \in [1, 50], i_3 \in [1, 50]$; top right: the iso-surface of $\rho =0.88$; bottom left: the surface mesh with $i_1 = 95$; bottom right: the
surface mesh with $i_3 = 50$.
		}
		\label{fig:3DJetMesh}
	\end{figure}
	
	\begin{figure}[!ht]
		\centering
		\begin{subfigure}[b]{0.42\textwidth}
		\centering
		\includegraphics[width=1.0\linewidth ]{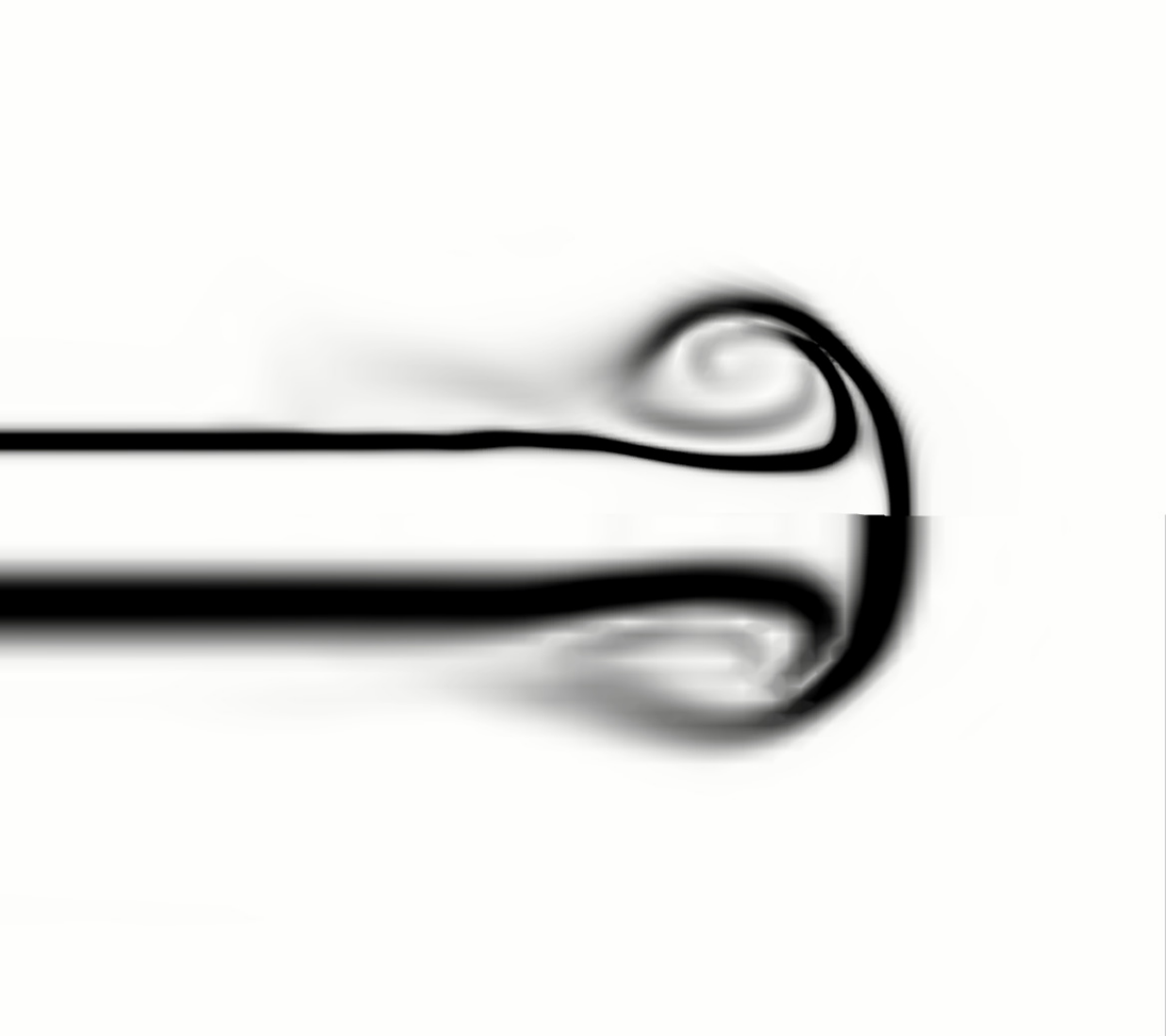}
	\end{subfigure}
	\begin{subfigure}[b]{0.42\textwidth}
		\centering
		\includegraphics [width=1.0\linewidth]{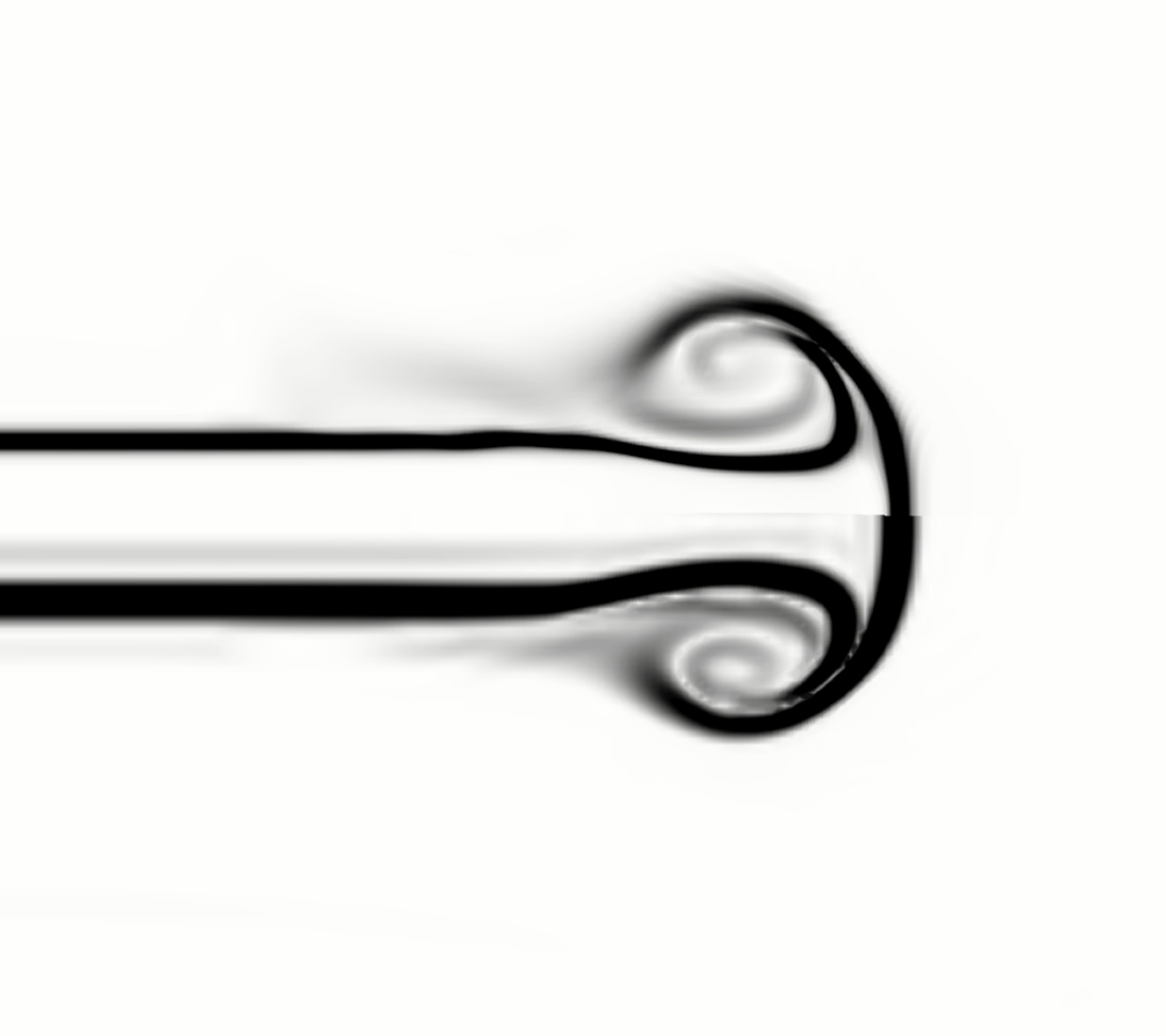}
	\end{subfigure}
		\caption{ Example \ref{ex:3DJetInjection}.
			Schlieren images on the slice $x_3 = 0$ at $t = 4$. Left:  {\tt MM}  (top half) and {\tt UM} (bottom half) with $150 \times 100 \times 100$; right: {\tt MM} with
 $150 \times 100 \times 100$  (top half) and {\tt UM} with $300 \times 200 \times 200$
 (bottom half).}
		\label{fig:3DJetRHO}
	\end{figure}

	
\end{example}

\section{Conclusion}\label{section:Conc}
This paper developed   the high-order entropy stable (ES)  finite difference schemes for multi-dimensional compressible Euler equations with the van der Waals equation of state (EOS) on adaptive moving meshes. 
 Similar to the schemes proposed in \cite{duan2021highorder}, 
the two-point symmetric EC flux was first technically constructed with the thermodynamic entropy and suitably chosen  parameter vectors.
It is worth reminding that the nonlinearity arising from the van der Waals EOS brings challenges for the  derivation of such EC fluxes.
 The high-order  EC fluxes in curvilinear coordinates were derived
by using the high-order discrete GCLs and the linear combination of the two-point EC fluxes. To suppress the oscillation near the shock waves,  the
high-order dissipation terms based on the newly derived complex scaled eigenmatrix and multi-resolution WENO
reconstruction were  added to the EC schemes to derive  the high-order semi-discrete ES schemes which satisfied the semi-discrete entropy inequality. The explicit third-order SSP-RK time discretization and the adaptive moving mesh scheme were implemented to construct  the fully-discrete ES adaptive moving mesh finite difference schemes. Several 1D, 2D and 3D  numerical tests were used to verify the
accuracy, effectiveness, and  ability to capture
the classical and non-classical waves
 of  the proposed
schemes on the parallel computer system with the MPI programming.

\appendix
\section{Different forms of linear equations to obtain two-points EC flux}\label{section:DiffLinearSysm}
This appendix presents the main results on
three different forms of linear equations
obtained with three different parameter vectors
to construct the explicit expression of  two-points symmetric EC flux for the 1D compressible Euler equations.

 If choosing the parameter vector
	$$\boldsymbol{z}=\left(z_{1}, z_{2}, z_{3} \right)^{\mathrm{T}}=(\frac{3-\rho}{T}, \frac{T}{\left(3-\rho\right)^2}, v_1)^{\mathrm{T}},$$
and the identities
	\begin{align*}
	\jump{-s} &= \frac{1}{\delta \meanln{z_1^2z_2}} \jump{\frac{1}{T}} + \dfrac{3\jump{z_1z_2}}{\meanln{3z_1z_2-1}},\\
	\jump{\frac{1}{T}}& = \jump{\left(z_1z_2^{\frac{1}{2}}\right)^2} =
	2\mean{z_1z_2^{\frac{1}{2}}}
	\left(\dfrac{\mean{z_1}}{2\mean{z_2^{\frac{1}{2}}}}\jump{z_2} + \mean{z_2^{\frac{1}{2}}}\jump{z_1}
	\right) \\
	& =\mean{z_1^2}\jump{z_2} + 2\mean{z_1}\mean{z_2}\jump{z_1}\\
	& = \mean{z_1}\mean{z_2}\jump{z_1} + \mean{z_1z_2}\jump{z_1} + \mean{z_1}^2\jump{z_2}\\
	&=:A\jump{z_1} + B\jump{z_2},
	\end{align*}
    then the jump of $\bV$ can be expressed as
	\begin{equation*}
	\left\{
	\begin{aligned}
	\jump{\bV_1}& =  \left(\dfrac{1}{\delta\meanln{1/T}} A + \dfrac{3\mean{z_2}}{\meanln{3z_1z_2-1}}+
	3\mean{z_2} - \frac{27}{4}A  - \frac{\mean{z_3^2}}{2}A +\frac{9}{4}
	\right)\jump{z_1}\\
	~&+ \left(\dfrac{1}{\delta\meanln{1/T}} B +
	\dfrac{3\mean{z_1}}{\meanln{3z_1z_2-1}}
	+ 3\mean{z_1} - \frac{27}{4}B - \frac{\mean{z_3^2}}{2}B  \right)\jump{z_2} -\mean{\frac{1}{T}}\mean{z_3}\jump{z_3},\\
	\jump{\bV_2}& =\mean{z_3}A\jump{z_1} +
	\mean{z_3}B\jump{z_2} + \mean{\frac{1}{T}}\jump{z_3},  \\
	\jump{\bV_3}&= -A\jump{z_1} - B\jump{z_2}.
	\end{aligned}
	\right.
	\end{equation*}
	The coefficient matrix of the above system of linear equations is given as
	\begin{align*}
	&
	\begin{bmatrix}
	\dfrac{1}{\delta\meanln{1/T}} A + \dfrac{3\mean{z_2}}{\meanln{3z_1z_2-1}}+
	3\mean{z_2} - \frac{27}{4}A  - \frac{\mean{z_3^2}}{2}A +\frac{9}{4}  & \mean{z_3}A                      &  -A                      \\
	\dfrac{1}{\delta\meanln{1/T}} B +
	\dfrac{3\mean{z_1}}{\meanln{3z_1z_2-1}}
	+ 3\mean{z_1} - \frac{27}{4}B - \frac{\mean{z_3^2}}{2}B& \mean{z_3}B                      &  -B        \\
	-\mean{\frac{1}{T}}\mean{z_3} & \mean{\frac{1}{T}}                    & 0
	\end{bmatrix}^{\mathrm{T}},
	\end{align*}
whose determinant  is
	\begin{align*}
	&\mean{\frac{3}{T}}   \left(
	\dfrac{B\mean{z_2} - A\mean{z_1}}{\meanln{3z_1z_2-1}}  + B\mean{z_2} - A\mean{z_1} + \frac{3}{4}
	\right)\\
	=& \mean{\frac{3}{T}}   \left(
	\dfrac{ - \mean{z_1z_2}\mean{z_1}}{\meanln{3z_1z_2-1}}  - \mean{z_1z_2}\mean{z_1} + \frac{3}{4}
	\right).
	\end{align*}
{It is not difficult to find that even in the case of
	$T_l = T_r  = \mean{T} >1$, there exist admissible states $\rho_l, \rho_r$ such that the determinant is zero.}

If choosing the following parameter vector
	$$\boldsymbol{z}=\left(z_{1}, z_{2}, z_{3} \right)^{\mathrm{T}}=(\rho,\dfrac{1}{T}, v_1)^{\mathrm{T}},$$
then the jump of $\bV$ can be { similarly} expressed as
	\begin{equation*}
	\left\{
	\begin{aligned}
	\jump{\bV_1}& =  \left(\dfrac{1}{\meanln{z_1}} +\dfrac{1}{\meanln{3-z_1}} + \dfrac{3}{(3-z_{1,l})(3-z_{1,r})}-\dfrac{9}{4} \mean{z_2} \right)\jump{z_1}\\
	~&+ \left(\dfrac{1}{\delta\meanln{z_2}}  - \dfrac{9}{4}\mean{z_1}  - \dfrac{\mean{z_3^2}}{2} \right)\jump{z_2} -\mean{z_2}\mean{z_3}\jump{z_3},\\
	\jump{\bV_2}& =\mean{z_3}\jump{z_2} +
	\mean{z_2}\jump{z_3},  \\
	\jump{\bV_3}&= -\jump{z_2}.
	\end{aligned}
	\right.
	\end{equation*}
	The coefficient matrix of the above system of linear equations is given by
	\begin{align*}
	&
	\begin{bmatrix}
	\dfrac{1}{\meanln{z_1}} +\dfrac{1}{\meanln{3-z_1}} + \dfrac{3}{(3-z_{1,l})(3-z_{1,r})}-\dfrac{9}{4} \mean{z_2} & 0                 &  0                     \\
	\dfrac{1}{\delta\meanln{z_2}}  - \dfrac{9}{4}\mean{z_1}  - \dfrac{\mean{z_3^2}}{2} & \mean{z_3}                    &  -1        \\
	-\mean{z_2}\mean{z_3} & \mean{z_2}                    & 0
	\end{bmatrix}^{\mathrm{T}},
	\end{align*}
and  the corresponding determinant  is
	\begin{align*}
	&\mean{z_2}   \left(
		\dfrac{1}{\meanln{z_1}} +\dfrac{1}{\meanln{3-z_1}} + \dfrac{3}{(3-z_{1,l})(3-z_{1,r})}-\dfrac{9}{4} \mean{z_2}
	\right),
	\end{align*}
	{which is smaller than the determinant obtained using the  parameter vector \eqref{eq:parameter_vector}. }

If choosing the following parameter vector
	$$\boldsymbol{z}=\left(z_{1}, z_{2}, z_{3} \right)^{\mathrm{T}}=(-\frac{9}{8T} + \frac{1}{3-\rho},\frac{1}{3-\rho}, v_1)^{\mathrm{T}},$$
and the identity
	$
	\jump{\frac{1}{T}} = \frac{8}{9}\jump{z_2}  - \frac{8}{9} \jump{z_1}=:A\jump{z_1} + B\jump{z_2}
	$,
	then the jump of $\bV$ can be expressed as
	\begin{equation*}
	\left\{
	\begin{aligned}
	\jump{\bV_1}& =  \left(\dfrac{1}{\delta\meanln{1/T}} A +6 -\mean{\frac{2}{z_2}}  - \frac{\mean{z_3^2}}{2}A\right)\jump{z_1}\\
	~&+ \left(\dfrac{1}{\delta\meanln{1/T}} B  - 3 + \frac{3}{\meanln{\frac{\rho}{3-\rho}}} + \dfrac{\mean{z_1}}{\mean{z_2}}\mean{\frac{2}{z_2}}- \frac{\mean{z_3^2}}{2}B \right)\jump{z_2} -\mean{\frac{1}{T}}\mean{z_3}\jump{z_3},\\
	\jump{\bV_2}& =\mean{z_3}A\jump{z_1} +
	\mean{z_3}B\jump{z_2} + \mean{\frac{1}{T}}\jump{z_3},  \\
	\jump{\bV_3}&= -A\jump{z_1} - B\jump{z_2}.
	\end{aligned}
	\right.
	\end{equation*}
The coefficient matrix of the above system of linear equations is given by
	\begin{align*}
	&
	\begin{bmatrix}
	\dfrac{1}{\delta\meanln{1/T}} A +6 -\mean{\frac{2}{z_2}}  - \frac{\mean{z_3^2}}{2}A & \mean{z_3}A                      &  -A                      \\
	\dfrac{1}{\delta\meanln{1/T}} B - 3 + \frac{3}{\meanln{\frac{\rho}{3-\rho}}}- \frac{\mean{z_3^2}}{2}B + \dfrac{\mean{z_1}}{\mean{z_2}}\mean{\dfrac{2}{z_2}}& \mean{z_3}B                      &  -B        \\
	-\mean{\frac{1}{T}}\mean{z_3} & \mean{\frac{1}{T}}                    & 0
	\end{bmatrix}^{\mathrm{T}},
	\end{align*}
whose	 determinant  is
	\begin{align*}
	&\mean{\frac{1}{T}}   \left( 6B +3A- B\mean{\frac{2}{z_2}} -\frac{3A}{\meanln{\frac{\rho}{3-\rho}}} - \dfrac{\mean{z_1}}{\mean{z_2}}\mean{\frac{2}{z_2}}A
	\right)  \\
	&=\frac{8}{9}\mean{\frac{1}{T}}   \left[3 -
	\mean{\frac{2}{z_2}} +\frac{3}{\meanln{\frac{\rho}{3-\rho}}} + \dfrac{\mean{z_1}}{\mean{z_2}}\mean{\frac{2}{z_2}}
	\right].
	\end{align*}
It is worth noting that the existence and uniqueness of solutions for the systems of linear equations are crucial for obtaining an explicit form of the two-point EC fluxes.
	The above three linear systems  may not obtain  the well-defined  explicit expressions of the  two-point EC fluxes, due to the possibility of their determinant being zero.
	So the construction of two-point EC fluxes is nontrivial, especially for the van der Waals  gas under consideration.


\section*{Acknowledgment}
The work was partially supported by the National Key R\&D
Program of China (Project Numbers 2020YFA0712000 \& 2020YFE0204200) and 
the National Natural Science Foundation
of China (Nos. 12171227, 12288101, \& 12326314).

\bibliographystyle{plain}

\bibliographystyle{lst}
\bibliography{}

\begin{thebibliography}{99}

\bibitem{ABGRALL1991171}
R.~Abgrall, An extension of {R}oe's upwind scheme to algebraic equilibrium real
gas models, \emph{Comput. $\&$ Fluids}, 19 (1991),  171-182.

\bibitem{anders1999transonic}
J.~Anders, W.~Anderson, and A.~Murthy, Transonic similarity theory applied to a
supercritical airfoil in heavy gas, \emph{J. Aircr.}, 36 (1999),  957-964.

\bibitem{Argrow1996}
B.M. Argrow, Computational analysis of dense gas shock tube flow, \emph{Shock
	Waves}, 6 (1996),  241-248.

\bibitem{Bethe1998}
H.A. Bethe, On the theory of shock waves for an arbitrary equation of
	state, In:
	\emph{Classic Papers in Shock Compression Science},
	Johnson, J.N., Ch\'{e}ret, R. (eds),
	Springer New York, 1998, 421-495.

\bibitem{Bhoriya2020Entropy}
D.~Bhoriya and H.~Kumar, {Entropy-stable schemes for relativistic hydrodynamics
	equations}, \emph{Z. Angew. Math. Phys.}, 71 (2020),  1-29.

\bibitem{Biswas2018Low}
B.~Biswas and R.K. Dubey, Low dissipative entropy stable schemes using third
order WENO and TVD reconstructions, \emph{Adv. Comput. Math.}, 44 (2018),
1153-1181.

\bibitem{Brackbill1993An}
J.U. Brackbill, An adaptive grid with directional control, \emph{J. Comput.
	Phys.}, 108 (1993),  38-50.

\bibitem{Brackbill1982Adaptive}
J.U. Brackbill and J.S. Saltzman, Adaptive zoning for singular problems in two
dimensions, \emph{J. Comput. Phys.}, 46 (1982),  342-368.

\bibitem{2000Two}
B.P. Brown and B.M. Argrow, Two-dimensional shock tube flow for dense gases,
\emph{J. Fluid Mech.}, 349 (1997),  95-115.

\bibitem{brown1998nonclassical}
B.P. Brown and B.M. Argrow, Nonclassical dense gas flows for simple geometries,
\emph{AIAA J.}, 36 (1998),  1842-1847.

\bibitem{Budd2009Adaptivity}
C.J. Budd, W.Z. Huang, and R.D. Russell, Adaptivity with moving grids,
\emph{Acta Numer.}, 18 (2009),  111-241.

\bibitem{CAO1999221}
W.~Cao, W.~Huang, and R.D. Russell, An $r$-adaptive finite element method based
upon moving mesh {PDE}s, \emph{J. Comput. Phys.}, 149 (1999),  221-244.

\bibitem{Carpenter2014Entropy}
M.H. Carpenter, T.C. Fisher, E.J. Nielsen, and S.H. Frankel, {Entropy stable
	spectral collocation schemes for the Navier-Stokes equations: Discontinuous
	interfaces}, \emph{SIAM J. Sci. Comput.}, 36 (2014),  B835-B867.

\bibitem{CENICEROS2001609}
H.D. Ceniceros and T.Y. Hou, An efficient dynamically adaptive mesh for
potentially singular solutions, \emph{J. Comput. Phys.}, 172 (2001),
609-639.

\bibitem{chandrashekar_2013}
P.~Chandrashekar, Kinetic energy preserving and entropy stable finite volume
schemes for compressible {E}uler and {N}avier-{S}tokes equations,
\emph{Commun. Comput. Phys.}, 14 (2013),  1252-1286.

\bibitem{Chen2020Review}
T.H. Chen and C.W. Shu, {Review of entropy stable discontinuous Galerkin
	methods for systems of conservation laws on unstructured simplex meshes},
\emph{CSIAM Trans. Appl. Math.}, 1 (2020),  1-52.

\bibitem{CINNELLA20061264}
P.~Cinnella, Roe-type schemes for dense gas flow computations, \emph{Comput.
	$\&$ Fluids}, 35 (2006),  1264-1281.

\bibitem{Colonna2007}
P.~Colonna, A.~Guardone, and N.R. Nannan, Siloxanes: A new class of candidate
{B}ethe-{Z}el'dovich-{T}hompson fluids, \emph{Phys. Fluids}, 19 (2007),
086102.

\bibitem{doi:10.1063/1.866082}
M.S. Cramer, Structure of weak shocks in fluids having embedded regions of
negative nonlinearity, \emph{Phys. Fluids}, 30 (1987),  3034-3044.

\bibitem{doi:10.1063/1.857855}
M.S. Cramer and L.M. Best, Steady, isentropic flows of dense gases, \emph{Phys.
	Fluids}, 3 (1991),  219-226.

\bibitem{Crandall1980Monotone}
M.G. Crandall and A.~Majda, Monotone difference approximations for scalar
conservation laws, \emph{Math. Comp.}, 34 (1980),  1-21.

\bibitem{DUAN2021109949}
J.M. Duan and H.Z. Tang, Entropy stable adaptive moving mesh schemes for 2{D}
and 3{D} special relativistic hydrodynamics, \emph{J. Comput. Phys.},
 426(2021),  109949.

\bibitem{duan2021highorder}
J.M. Duan and H.Z. Tang, High-order accurate entropy stable adaptive moving
mesh finite difference schemes for special relativistic
(magneto)hydrodynamics, \emph{J. Comput. Phys.}, 456 (2022),  111038.




\bibitem{Fjordholm2011Well}
U.S. Fjordholm, S.~Mishra, and E.~Tadmor, {Well-balanced and energy stable
	schemes for the shallow water equations with discontinuous topography},
\emph{J. Comput. Phys.}, 230 (2011),  5587-5609.

\bibitem{Fjordholm2012Arbitrarily}
U.S. Fjordholm, S.~Mishra, and E.~Tadmor, {Arbitrarily high-order accurate
	entropy stable essentially non-oscillatory schemes for systems of
	conservation laws}, \emph{SIAM J. Numer. Anal.}, 50 (2012),  544-573.

\bibitem{Gassner2013A}
G.J. Gassner, {A skew-symmetric discontinuous Galerkin spectral element
	discretization and its relation to SBP-SAT finite difference methods},
\emph{SIAM J. Sci. Comput.}, 35 (2013),  1233-1253.

\bibitem{GLAISTER1988382}
P.~Glaister, An approximate linearised {R}iemann solver for the {E}uler
equations for real gases, \emph{J. Comput. Phys.}, 74 (1988),  382-408.

\bibitem{Gottlieb2001Strong}
S.~Gottlieb, C.W. Shu, and E.~Tadmor, Strong stability-preserving high-order
time discretization methods, \emph{SIAM Review}, 43 (2001),  89-112.

\bibitem{Guardone2001}
A.~Guardone and L.~Quartapelle, \emph{Exact {R}oe linearization for van der
	{W}aals' gas}, Springer US, New York, NY (2001), 419-424.

\bibitem{GUARDONE200250}
A.~Guardone and L.~Vigevano, Roe linearization for the van der {W}aals gas,
\emph{J. Comput. Phys.}, 175 (2002),  50-78.

\bibitem{Harten1976}
A.~Harten, J.M. Hyman, and P.D. Lax, On finite-difference approximations and
entropy conditions for shocks, \emph{Comm. Pure Appl. Math.}, 29 (1976),
297-322.

\bibitem{He2012RHD}
P.~He and H.Z. Tang, {An adaptive moving mesh method for two-dimensional
	relativistic hydrodynamics}, \emph{Commun. Comput. Phys.}, 11 (2012),
114-146.

\bibitem{He2012RMHD}
P.~He and H.Z. Tang, {An adaptive moving mesh method for two-dimensional
	relativistic magnetohydrodynamics}, \emph{Comput. $\&$ Fluids}, 60 (2012),
1-20.

\bibitem{Hiltebrand2014Entropy}
A.~Hiltebrand and S.~Mishra, {Entropy stable shock capturing space-time
	discontinuous Galerkin schemes for systems of conservation laws},
\emph{Numer. Math.}, 126 (2014),  103-151.

\bibitem{Ismail2009Affordable}
F.~Ismail and P.L. Roe, {Affordable, entropy-consistent Euler
flux functions II: Entropy production at shocks}, \emph{J. Comput. Phys.}, 228 (2009),
5410-5436.

\bibitem{Thompson1972}
K.C. Lambrakis and P.A. Thompson, Existence of real fluids with a negative
fundamental derivative {$\Gamma$}, \emph{Phys. Fluids}, 15 (1972),  933-935.

\bibitem{Lefloch2002Fully}
P.G. LeFloch, J.M. Mercier, and C.~Rohde, {Fully discrete entropy conservative
	schemes of arbitraty order}, \emph{SIAM J. Numer. Anal.}, 40 (2002),
1968-1992.

\bibitem{Li2001Moving}
R.~Li, T.~Tao, and P.W. Zhang, Moving mesh methods in multiple dimensions based
on harmonic maps, \emph{J. Comput. Phys.}, 170 (2001),  562-588.

\bibitem{LI1997368}
S.~Li and L.~Petzold, Moving mesh methods with upwinding schemes for
time-dependent {PDE}s, \emph{J. Comput. Phys.}, 131 (1997),  368-377.

\bibitem{li2022}
S.T. Li, J.M. Duan, and H.Z. Tang, High-order accurate entropy stable adaptive
moving mesh finite difference schemes for (multi-component) compressible
euler equations with the stiffened equation of state, \emph{Comput. Methods
	Appl. Mech. Engrg.}, 399 (2022),  115311.

\bibitem{2007PLUTO}
A.~Mignone, G.~Bodo, S.~Massaglia, T.~Matsakos, O.~Tesileanu, C.~Zanni, and
A.~Ferrari, {PLUTO}: {A} numerical code for computational astrophysics,
\emph{Astrophys. J. Suppl. Ser.}, 170 (2007),  228-242.

\bibitem{Osher1984Riemann}
S.~Osher, Riemann solvers, the entropy condition, and difference
approximations, \emph{SIAM J. Numer. Anal.}, 21 (1984),  217-235.

\bibitem{SOsher1988}
S.~Osher and E.~Tadmor, On the convergence of difference approximations to
scalar conservation laws, \emph{Math. Comp.}, 50 (1988),  19-51.

\bibitem{PRODANOV2022414077}
E.M. Prodanov, Mathematical analysis of the van der {W}aals equation,
\emph{Phys. B: Condens. Matter}, 640 (2022),  414077.

\bibitem{Ren2000An}
W.Q. Ren and X.P. Wang, An iterative grid redistribution method for singular
problems in multiple dimensions, \emph{J. Comput. Phys.}, 159 (2000),
246-273.

\bibitem{doi:10.1063/1.858183}
G.H. Schnerr and P.~Leidner, Diabatic supersonic flows of dense gases,
\emph{Phys. Fluids}, 3 (1991),  2445-2458.

\bibitem{SHYUE199943}
K.M. Shyue, A fluid-mixture type algorithm for compressible multicomponent flow
with van der {W}aals equation of state, \emph{J. Comput. Phys.}, 156 (1999),
43-88.

\bibitem{Stockie2001}
J.M. Stockie, J.A. Mackenzie, and R.D. Russell, A moving mesh method for
one-dimensional hyperbolic conservation laws, \emph{SIAM J. Sci. Comput}, 22
(2001),  1791-1813.

\bibitem{Tadmor1987The}
E.~Tadmor, {The numerical viscosity of entropy stable schemes for systems of
	conservation laws, I}, \emph{Math. Comp.}, 49 (1987),  91-103.
	
\bibitem{Tadmor2003Entropy}
E.~Tadmor, {Entropy stability theory for difference approximations of nonlinear
	conservation laws and related time-dependent problems}, \emph{Acta Numer.},
12 (2003),  451-512.


\bibitem{Tang2003Adaptive}
H.Z. Tang and T.~Tang, Adaptive mesh methods for one- and two-dimensional
hyperbolic conservation laws, \emph{SIAM J. Numer. Anal.}, 41 (2003),
487-515.


\bibitem{Tang2003An}
H.Z. Tang, T.~Tang, and P.W. Zhang, An adaptive mesh redistribution method for
nonlinear {Hamilton-Jacobi} equations in two- and three-dimensions, \emph{J.
	Comput. Phys.}, 188 (2003),  543-572.
	
\bibitem{Tang2005Moving}
T.~Tang, Moving mesh methods for computational fluid dynamics, \emph{Contemp.
	Math.}, 383 (2005),  141-173.

\bibitem{thompson1971fundamental}
P.A. Thompson, A fundamental derivative in gasdynamics, \emph{Phys. Fluids}, 14
(1971),  1843-1849.

\bibitem{VISBAL2002155}
M.R. Visbal and D.V. Gaitonde, On the use of higher-order finite-difference
schemes on curvilinear and deforming meshes, \emph{J. Comput. Phys.}, 181
(2002),  155-185.

\bibitem{wagner1978theoretical}
B.~Wagner and W.~Schmidt, Theoretical investigations of real gas effects in
cryogenic wind tunnels, \emph{AIAA J.}, 16 (1978),  580-586.


\bibitem{Wang2004A}
D.S. Wang and X.P. Wang, A three-dimensional adaptive method based on the
iterative grid redistribution, \emph{J. Comput. Phys.}, 199 (2004),
423-436.

\bibitem{WANG2021105138}
Z.~Wang, J.~Zhu, L.~Tian, Y.~Yang, and N.~Zhao, An efficient fifth-order finite
difference multi-resolution {WENO} scheme for inviscid and viscous flow
problems, \emph{Comput. $\&$ Fluids}, 230 (2021),  105138.

\bibitem{Winslow1967Numerical}
A.M. Winslow, Numerical solution of the quasilinear {P}oisson equation in a
nonuniform triangle mesh, \emph{J. Comput. Phys.}, 1 (1967),  149-172.



\bibitem{Wu2020Entropy}
K.L. Wu and C.-W. Shu, Entropy symmetrization and high-order accurate entropy
stable numerical schemes for relativistic MHD equations, \emph{SIAM J. Sci.
	Comput.}, 42 (2020),  A2230-A2261.



\bibitem{Zeldovich1946}
Y.B. Zel'dovich, On the possibility of rarefaction shock waves, \emph{Zh. Eksp.
	Teor. Fiz.}, 4 (1946),  363-364.




\bibitem{Zhang1993Discrete}
H.~Zhang, M.~Reggio, J.Y. Tr\'{e}panier, and R.~Camarero, Discrete form of the
{GCL} for moving meshes and its implementation in {CFD} schemes,
\emph{Comput. {$\&$} Fluids}, 22 (1993),  9-23.



\bibitem{zhang2023highorder}
Z.H. Zhang, J.M. Duan, and H.Z. Tang,
 High-order accurate well-balanced energy
stable adaptive moving mesh finite difference schemes for the shallow water
equations with non-flat bottom topography,
\emph{J. Comput. Phys.}, 492(2023), 112451.

\bibitem{zhang2024highorder}
Z.H. Zhang, H.Z. Tang, and J.M. Duan, High-order accurate well-balanced energy stable finite difference schemes for multi-layer shallow water equations on fixed and adaptive moving meshes,  \emph{arXiv: 2311.08124}, 2024.

\end{thebibliography}

\end{document}